\documentclass[11pt,a4paper]{article}
\usepackage{geometry}
\usepackage[T1]{fontenc}
\usepackage[utf8]{inputenc}
\usepackage{authblk}
\usepackage{graphics}
\usepackage{tikz,pgf}
\usetikzlibrary{decorations.markings, decorations.pathmorphing}
\usetikzlibrary{arrows,decorations,backgrounds,scopes,plotmarks}
\usepackage{caption}
\usepackage{subfigure}
\usepackage{float}

\usepackage[english]{babel}
\usepackage[all]{xy}
\usepackage{marvosym}
\usepackage{stmaryrd,mathrsfs,amsmath,amssymb,bm,amsfonts}
\usepackage{enumerate}
\usepackage{amsmath,amsthm,amssymb,amscd}

\newtheorem{thm}{Theorem}[subsection]
\newtheorem{defn}[thm]{Definition}
\newtheorem{prop}[thm]{Proposition}
\newtheorem{lem}[thm]{Lemma}
\newtheorem{rem}[thm]{Remark}
\newtheorem{ex}[thm]{Example}
\newtheorem{cor}[thm]{Corollary}

\numberwithin{equation}{subsection}

\begin{document}

\title{Causal-net category}

\author{Xuexing Lu}

\affil{\small School of Mathematical Sciences,
Zaozhuang University\\
Zaozhuang, Shandong, China}

\renewcommand\Authands{ and }

\maketitle

\begin{abstract}
A causal-net is a finite acyclic directed graph. In this paper, we introduce a category, denoted by $\mathbf{Cau}$ and called causal-net category, whose objects are causal-nets and morphisms between two causal-nets are the functors between their path categories. The category  $\mathbf{Cau}$ is in fact the Kleisli category of the "free category on a causal-net" monad. Firstly, we motivate the study of $\mathbf{Cau}$ and illustrate its application in the framework of causal-net condensation. We show that there are exactly six types of indecomposable morphisms, which correspond to six conventions of graphical calculi for monoidal categories. Secondly, we study several composition-closed classes of morphisms in $\mathbf{Cau}$, which characterize interesting partial orders among causal-nets, such as coarse-graining, merging, contraction, immersion-minor, topological minor, etc., and prove several useful decomposition theorems.  Thirdly, we introduce a categorical framework for minor theory and use it to study several types of generalized minors in $\mathbf{Cau}$. In addition, we prove a fundamental theorem that any morphism in $\mathbf{Cau}$ is a composition of the six types of indecomposable morphisms, and show that the notions of coloring and exact minor can be understood as special kinds of minimal-quotient and sub-quotient in $\mathbf{Cau}$, respectively. Base on these results, we conclude that $\mathbf{Cau}$ is a natural setting for studying causal-nets, and the theory of $\mathbf{Cau}$ should shed new light on the category-theoretic understanding of graph theory.
\end{abstract}



\tableofcontents

\section{Introduction}
In this paper, a \textbf{causal-net} is a finite acyclic directed graph, which may be empty, or has isolated vertices and multi-edges. For each causal-net $G$, there is a \textbf{path category} $\mathbf{P}(G)$ with vertices of $G$ as objects,  directed paths in $G$ as morphisms, directed paths of length zero as identity morphisms and  juxtaposition of directed paths as composition. The source and target of a directed path are the starting-vertex and ending-vertex, respectively. We define a \textbf{morphism} between two causal-nets $G_1$ and $G_2$ to be a functor between their path categories,  that is, $Hom(G_1, G_2)\triangleq Fun\big(\mathbf{P}(G_1),\mathbf{P}(G_2)\big)$. All causal-nets and their morphisms form a category, called \textbf{causal-net category}, and denoted by $\mathbf{Cau}$. It is a full sub-category of the category $\mathbf{Cat}$ of small categories. The construction of path category is same as that of free category for a general directed graph (II.7, \cite{[M98]}) and the category $\mathbf{Cau}$ is in fact the Kleisli category of the "free category on a causal-net" monad (VI.5, \cite{[M98]}).

\begin{ex}
The following figure shows an example of a causal-net. The objects of its path category are $v_1,v_2, v_3, v_4$, and morphisms are listed as follows:
$Hom(v_1,v_1)=\{Id_{v_1}\}$, $Hom(v_2,v_2)=\{Id_{v_2}\}$, $Hom(v_3,v_3)=\{Id_{v_3}\},$ $Hom(v_4,v_4)=\{Id_{v_4}\},$
$Hom(v_1,v_2)=\{e_1,e_2\}$, $Hom(v_2,v_3)=\{e_3\}$, $Hom(v_1,v_3)=\{e_3e_1,e_3e_2,e_4\},$ $Hom(v_2,v_1)=Hom(v_3,v_2)=Hom(v_3,v_1)=\emptyset,$  $Hom(v_i,v_4)=Hom(v_4,v_i)=\emptyset$, $i=1,2,3$.

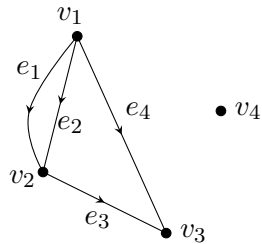
\begin{figure}[H]
\centering
\begin{tikzpicture}[scale=0.9]
\node (v1) at (-0.5,0.5) {};
\node (v2) at (-1,-1.5) {};
\node (v3) at (0.8,-2.4) {};
\draw [fill](v1) circle [radius=0.07];
\draw [fill](v2) circle [radius=0.07];
\draw [fill](v3) circle [radius=0.07];
\draw  (-0.5,0.5) -- (-1,-1.5) [postaction={decorate, decoration={markings,mark=at position .5 with {\arrow[black]{stealth}}}}];
\draw  (-1,-1.5)  -- (0.8,-2.4)[postaction={decorate, decoration={markings,mark=at position .5 with {\arrow[black]{stealth}}}}];
\draw  (-0.5,0.5) -- (0.8,-2.4)[postaction={decorate, decoration={markings,mark=at position .5 with {\arrow[black]{stealth}}}}];
\draw  plot[smooth, tension=.7] coordinates {(-0.5,0.5) (-1.2,-0.6) (-1,-1.5)}[postaction={decorate, decoration={markings,mark=at position .6 with {\arrow[black]{stealth}}}}];
\node at (-0.6,0.8) {$v_1$};
\node at (-1.3,-1.6) {$v_2$};
\node at (1.2,-2.4) {$v_3$};
\node at (-1.2,0) {$e_1$};
\node at (-0.6,-0.8) {$e_2$};
\node at (-0.2,-2.2) {$e_3$};
\node at (0.4,-0.6) {$e_4$};
\node at (1.6,-0.6) {};
\draw [fill](1.6,-0.6) circle [radius=0.07];
\node at (2,-0.6) {$v_4$};
\end{tikzpicture}
\end{figure}
\end{ex}

\subsection{Relation with graphical calculi}\label{Relation}
The category $\mathbf{Cau}$ has a close connection with the graphical calculi for monoidal categories \cite{[JS91],[S10],[BS10],[HLY16]}.  A morphism in $\mathbf{Cau}$ is called \textbf{fundamental} or \textbf{indecomposable} if it can not be decomposed non-trivially as a composition of two morphisms.
There are exactly six types of fundamental morphisms in $\mathbf{Cau}$, which correspond to six basic conventions in graphical calculi for monoidal categories. Note that the first three types have appeared in the Baez construction of spin network states for quantum gauge theories \cite{[B96]}.

The first type is called \textbf{subdividing an edge}, which coincides with the \textbf{identity convention} \cite{[HLY16]}.
\begin{center}
\begin{tikzpicture}[scale=0.5]
\node (v1) at (0.5,0) {};
\node (v2) at (3.5,0) {};
\draw [fill](v1) circle [radius=0.09];
\draw [fill](v2) circle [radius=0.09];
\draw (0.5,0) node (v7) {} -- (3.5,0) node (v8) {}[postaction={decorate, decoration={markings,mark=at position .55 with {\arrow[black]{stealth}}}}];
\node [scale=0.7]at (0.5,0.5) {$f$};
\node [scale=0.7]at (2,0.5) {$X$};
\node [scale=0.7]at (3.5,0.5) {$g$};
\node at (5.5,0) {$\Rightarrow$};
\node (v3) at (7.5,0) {};
\node at (9.5,0) {};
\node (v4) at (9.5,0) {};
\node (v5) at (11.5,0) {};
\draw (7.5,0) node (v11) {} -- (9.5,0)[postaction={decorate, decoration={markings,mark=at position .55 with {\arrow[black]{stealth}}}}];
\draw (9.5,0)--(11.5,0) node (v12) {}[postaction={decorate, decoration={markings,mark=at position .55 with {\arrow[black]{stealth}}}}];
\node [scale=0.7]at (7.5,0.5) {$f$};
\node [scale=0.7]at (8.5,0.5) {$X$};
\node [scale=0.7]at (9.5,0.5) {$Id_X$};
\node [scale=0.7]at (10.5,0.5) {$X$};
\node [scale=0.7]at (11.5,0.5) {$g$};
\draw [fill](v3) circle [radius=0.09];
\draw [fill](v4) circle [radius=0.09];
\draw [fill](v5) circle [radius=0.09];
\node (v6) at (0,0) {};
\node (v9) at (4,0) {};
\draw  (0,0) -- (0.5,0);
\draw (3.5,0) -- (4,0);
\node (v10) at (7,0) {};
\node (v13) at (12,0) {};
\draw (7,0) --(7.5,0);
\draw  (11.5,0) -- (12,0);
\end{tikzpicture}

\end{center}

The second type is called \textbf{adding an edge}, which coincides  with the \textbf{unit convention} \cite{[HLY16]}.

\begin{center}
\begin{tikzpicture}[scale=0.5]
\node at (5.25,0) {$\Rightarrow$};
\node (v2) at (1,0) {};
\node (v4) at (4,0) {};
\node (v3) at (3,0) {};
\node (v1) at (0,0) {};
\draw [fill](v2)  circle [radius=0.09];
\draw [fill](v3)  circle [radius=0.09];
\draw (0,0)--(1,0)[postaction={decorate, decoration={markings,mark=at position .45 with {\arrow[black]{stealth}}}}];
\draw  (3,0)--(4,0)[postaction={decorate, decoration={markings,mark=at position .55 with {\arrow[black]{stealth}}}}];
\node (v5) at (7.5,0) {};
\node (v6) at (10.5,0) {};
\draw  (7.5,0) --(10.5,0)[postaction={decorate, decoration={markings,mark=at position .55 with {\arrow[black]{stealth}}}}];
\draw [fill](v6)  circle [radius=0.09];
\draw [fill](7.5,0)  circle [radius=0.09];
\node [scale=0.7]at (7.5,0.5) {$f$};
\node [scale=0.7]at (9,0.5) {$I$};
\node [scale=0.7]at (10.5,0.5) {$g$};
\node (v7) at (6.5,0) {};
\node (v8) at (11.5,0) {};
\draw (6.5,0)-- (7.5,0)[postaction={decorate, decoration={markings,mark=at position .55 with {\arrow[black]{stealth}}}}];
\draw (10.5,0) --(11.5,0)[postaction={decorate, decoration={markings,mark=at position .55 with {\arrow[black]{stealth}}}}];
\node [scale=0.7]at (1,0.5) {$f$};
\node [scale=0.7]at (3,0.5) {$g$};
\end{tikzpicture}
\end{center}

The third type is called \textbf{adding an isolated vertex}, which is a consequence of the unit convention and the identity convention \cite{[HLY16]}.

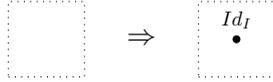
\begin{center}
\begin{tikzpicture}[scale=0.5]
\node at (5.5,0) {$\Rightarrow$};
\draw[dotted]  (2,1) rectangle (4,-1);
\node at (7.5,0) {};
\node [scale=0.7] at (8,0.5) {$Id_I$};
\draw [fill](8,0)  circle [radius=0.09];
\draw[dotted]  (7,1) rectangle (9,-1);
\end{tikzpicture}
\end{center}

The fourth type is called \textbf{merging two vertices}, which coincides with the tensor product of two morphisms.

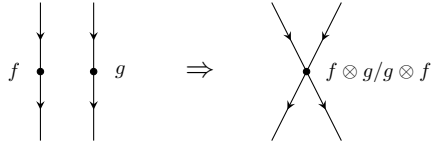
\begin{center}
\begin{tikzpicture}[scale=0.7]
\node (v2) at (2.5,-2) {};
\node (v1) at (2.5,-0.5) {};
\node (v3) at (2.5,-3.5) {};
\node (v5) at (3.5,-2) {};
\node (v4) at (3.5,-0.5) {};
\node (v6) at (3.5,-3.5) {};
\node at (5.5,-2) {$\Rightarrow$};
\node (v8) at (7.5,-2) {};
\node (v7) at (6.75,-0.5) {};
\node (v9) at (8.25,-0.5) {};
\node (v10) at (6.75,-3.5) {};
\node (v11) at (8.25,-3.5) {};
\draw  (v1) --(2.5,-2)[postaction={decorate, decoration={markings,mark=at position .55 with {\arrow[black]{stealth}}}}];
\draw  (2.5,-2)-- (v3)[postaction={decorate, decoration={markings,mark=at position .55 with {\arrow[black]{stealth}}}}];
\draw  (v4) --(3.5,-2)[postaction={decorate, decoration={markings,mark=at position .55 with {\arrow[black]{stealth}}}}];
\draw  (3.5,-2)-- (v6)[postaction={decorate, decoration={markings,mark=at position .55 with {\arrow[black]{stealth}}}}];
\draw  (v7) -- (7.5,-2) [postaction={decorate, decoration={markings,mark=at position .55 with {\arrow[black]{stealth}}}}];
\draw  (v9) --  (7.5,-2) [postaction={decorate, decoration={markings,mark=at position .55 with {\arrow[black]{stealth}}}}];
\draw   (7.5,-2) --(v10) [postaction={decorate, decoration={markings,mark=at position .55 with {\arrow[black]{stealth}}}}];
\draw   (7.5,-2) -- (v11)[postaction={decorate, decoration={markings,mark=at position .55 with {\arrow[black]{stealth}}}}];
\node [scale=0.7]at (2,-2) {$f$};
\node [scale=0.7]at (4,-2) {$g$};
\node [scale=0.7]at (8.85,-2) {$f\otimes g/g\otimes f$};
\draw [fill](v2)  circle [radius=0.06];
\draw [fill](v5)  circle [radius=0.06];
\draw [fill](v8)  circle [radius=0.06];
\end{tikzpicture}
\end{center}

The fifth type is called \textbf{coarse-graining two parallel edges}, which coincides with the tensor product of two objects.

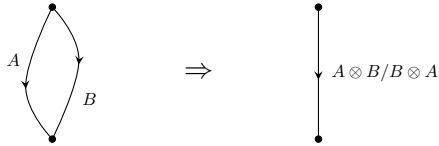
\begin{center}
\begin{tikzpicture}[scale=0.7]
\node (v1) at (0,2) {};
\node (v2) at (0,-0.5) {};
\draw  plot[smooth, tension=.7] coordinates {(0,2) (-0.5,0.5) (0,-0.5)}[postaction={decorate, decoration={markings,mark=at position .6 with {\arrow[black]{stealth}}}}];
\draw  plot[smooth, tension=.7] coordinates {(0,2) (0.5,1) (0,-0.5) }[postaction={decorate, decoration={markings,mark=at position .45 with {\arrow[black]{stealth}}}}];
\node (v3) at (5,2) {};
\node (v4) at (5,-0.5) {};
\draw (5,2)-- (5,-0.5)[postaction={decorate, decoration={markings,mark=at position .55 with {\arrow[black]{stealth}}}}];
\node at (2.75,0.75) {$\Rightarrow$};
\draw [fill](v1) circle [radius=0.06];
\draw [fill](v2) circle [radius=0.06];
\draw [fill](v3) circle [radius=0.06];
\draw [fill](v4) circle [radius=0.06];
\node[scale=0.6] at (-0.73,1) {$A$};
\node [scale=0.6]at (0.7,0.25) {$B$};
\node[scale=0.6] at (6.25,0.75) {$A\otimes B/ B\otimes A$};
\end{tikzpicture}
\end{center}

The sixth type is called \textbf{contracting an edge}, which coincides with the composition of two morphisms.

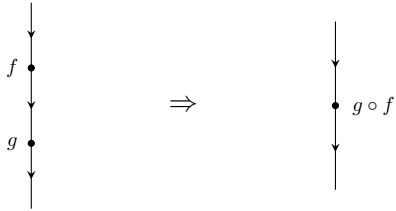
\begin{center}
\begin{tikzpicture}
\node (v2) at (0,1.5) {};
\node (v1) at (0,2.5) {};
\node (v3) at (0,0.5) {};
\node (v4) at (0,-0.5) {};
\node (v5) at (4,2.25) {};
\node (v6) at (4,1) {};
\node (v7) at (4,-0.25) {};
\draw  (v1) -- (0,1.5)[postaction={decorate, decoration={markings,mark=at position .55 with {\arrow[black]{stealth}}}}];
\draw  (0,1.5) -- (0,0.5)[postaction={decorate, decoration={markings,mark=at position .55 with {\arrow[black]{stealth}}}}];
\draw  (0,0.5) -- (v4)[postaction={decorate, decoration={markings,mark=at position .55 with {\arrow[black]{stealth}}}}];
\draw  (v5) -- (4,1)[postaction={decorate, decoration={markings,mark=at position .55 with {\arrow[black]{stealth}}}}];
\draw(4,1)-- (v7)[postaction={decorate, decoration={markings,mark=at position .55 with {\arrow[black]{stealth}}}}];
\draw [fill](v2) circle [radius=0.04];
\draw [fill](v3) circle [radius=0.04];
\draw [fill](v6) circle [radius=0.04];
\node at (2,1) {$\Rightarrow$};
\node [scale=0.7]at (-0.25,1.5) {$f$};
\node [scale=0.7]at (-0.25,0.5) {$g$};
\node [scale=0.7]at (4.5,1) {$g\circ f$};
\end{tikzpicture}
\end{center}

One main result of this paper is the fundamental theorem of $\mathbf{Cau}$  (Theorem \ref{com}), which says that any morphism of causal-nets can be decomposed as a composition of fundamental morphisms. In other words, the six types of fundamental morphisms generate all morphisms of $\mathbf{Cau}$.

The six types of fundamental morphisms can be classified into three classes:
$(1)$ the first type of fundamental morphisms are called \textbf{fundamental subdivisions};
$(2)$ the second and third types of fundamental morphisms are called \textbf{fundamental embeddings}; $(3)$ the fourth, fifth and sixth types of fundamental morphisms are called \textbf{fundamental coarse-grainings}. Any subdivision (or embedding, coarse-graining) is a composition of fundamental subdivisions(or embeddings, coarse-grainings). As shown first by Power \cite{[Po90]} that there is  a substantial connection between graph theory and 2-category theory (further explained in Remark 4.3 of \cite{[HLY16]}), here the unity of the six types of fundamental morphisms of $\mathbf{Cau}$ and the six basic conventions of graphical calculi shows a fundamental connection between graph theory and monoidal category theory.

\subsection{Causal-net condensation}

Motivated by the Baez construction of spin network states for quantum gauge theories \cite{[B96]}, we proposed a mathematical framework, called \textbf{causal-net condensation}, for the theory of quantum gravity. The main motivation for introducing the category $\mathbf{Cau}$ is that it plays a key role in the framework of causal-net condensation. The framework of causal-net condensation is parallel to that of factorization homology \cite{[AF15]}, where the roles of $\mathbf{Cau}$ and small symmetric monoidal categories in this framework are same as those of the little $n$-disc operad and $E_n$-algebras in factorization homology, respectively. We show the framework below, where $\mathbf{Cau}$, $\mathbf{Set}$, $\mathbf{Top}$ and $\mathbf{Cat}$ denote the causal-net category, the category of sets and mappings, the category of topological spaces and continuous mappings and the category of small categories and functors, respectively; $\mathbf{I}: \mathbf{Cau}\to \mathbf{Cat}$ is the embedding functor;  $\mathcal{P}: \mathbf{Top}\to \mathbf{Cat}$ is the functor taking a topological space to its path category; $\mathcal{F}_{\mathbf{S}}:\mathbf{Cau}\to \mathbf{Set}$ is a functor associated with an arbitrarily fixed small symmetric monoidal category $\mathbf{S}$; the functor $\mathcal{L}=Lan_{\mathbf{I}}\mathcal{F}_{\mathbf{S}}:\mathbf{Cat}\to \mathbf{Set}$ is the left Kan extension of $\mathcal{F}_{\mathbf{S}}$ along $\mathbf{I}$; and $\eta:\mathcal{F}_{\mathbf{S}}\to \mathcal{L}\circ \mathbf{I}$ is the natural transformation of the left Kan extension. The composition $\mathcal{L}\circ \mathcal{P}:\mathbf{Top}\to \mathbf{Set}$ is called the causal-net condensation associated with $\mathbf{S}$.

\begin{center}
\begin{tikzpicture}
\node (v1) at (0,0) {$\mathbf{Cau}$};
\node (v2) at (3.5,0) {$\mathbf{Set}$};
\node (v3) at (0,-2.5) {$\mathbf{Cat}$};
\draw [-latex] (v1) edge (v2);
\draw [-latex] (v1) edge (v3);
\draw [dashed,-latex] (v3) edge (v2);
\node at (2.7,-1.5) {$\mathcal{L}=Lan_{\mathbf{I}}\mathcal{F}_{\mathbf{S}}$};
\node at (1.6,0.25) {$\mathcal{F}_{\mathbf{S}}$};
\node at (-0.25,-1.25) {$\mathbf{I}$};
\node [rotate=-50]at (1.25,-0.75) {$\Longrightarrow$};
\node at (1.5,-0.5){$\eta$};
\node (v4) at (-3.5,-2.5) {$\mathbf{Top}$};
\draw [-latex] (v4) edge (v3);
\node at (-1.75,-2.25) {$\mathcal{P}$};
\node at (-1.75,-2.75) {path category};
\node at (0,-3.5) {$\mathbf{S}$$=$ small symmetric monoidal category};
\node at (0,-4) {$\mathcal{L}\circ \mathcal{P}$ $=$ causal-net condensation associated with $\mathbf{S}$};
\end{tikzpicture}
\end{center}

Recall that for any topological space $X$, its path category $\mathcal{P}(X)$ is the small category with points of $X$ as objects and with unparametrised oriented paths as morphisms. If we view a directed graph as a combinatorial model of a "finite topological space", where vertices and directed paths play the roles of points and unparametrised oriented paths, respectively, then the topological and the algebraic constructions of  path categories coincide. Moreover, a morphism of directed graphs (defined similarly as that of causal-nets in $\mathbf{Cau}$) models a "continuous mapping of finite topological spaces". It is this analogy that inspires us to introduce the category $\mathbf{Cau}$.

Given a small symmetric monoidal category $\mathbf{S}$, we can define a functor $\mathcal{F}_{\mathbf{S}}:\mathbf{Cau}\to \mathbf{Set}$, called the \textbf{causal-algebra associated with $\mathbf{S}$}. To give a precise definition of $\mathcal{F}_{\mathbf{S}}$, we need some preliminaries.

A \textbf{closed-diagram} (or simply \textbf{diagram}) in $\mathbf{S}$ consists of a causal-net $G$, a polarization  of $G$ (a choice of linear orders on each incoming-edge set $I(v)$ and outgoing-edge set $O(v)$, \cite{[JS91]}), an \textbf{edge-decoration} $d_o:E(G)\to Ob(\mathbf{S})$ (a mapping from the edge set to the object set) and a \textbf{vertex-decoration}
$d_m:V(G)\to Mor(\mathbf{S})$ (a mapping from the vertex set to the morphism set), such that  for each vertex $v\in V(G)$, $d_m(v)$ is a morphism from the object $d_o(h_1)\otimes\cdots\otimes d_o(h_k)$ to the object $d_o(h'_1)\otimes\cdots\otimes d_o(h'_l)$, where $h_1<\cdots<h_k$ and $h'_1<\cdots<h'_l$ are the linearly ordered incoming-edges and outgoing-edges of $v$, respectively; especially, when $v$ is  a source vertex, the domain of $d_m(v)$ is defined to be the unit object and when $v$ is a sink vertex, the codomain of $d_m(v)$ is defined to be the unit object. This is consistent with the unit convention. When a diagram has $G$ as the underlying causal-net, we call it a diagram on $G$.

An \textbf{open-diagram} in $\mathbf{S}$ is just a closed-diagram in the \textbf{zero-extension} of $\mathbf{S}$, which is the small symmetric monoidal category obtained by canonically adding a \textbf{zero-object} $\mathbf{0}$ (or called a \textbf{tensor zero}) to $\mathbf{S}$ such that for any object $X$ of $\mathbf{S}$, $X\otimes \mathbf{0}=\mathbf{0}=\mathbf{0}\otimes X$.    The correspondence between open diagrams and \textbf{progressive polarised diagrams} \cite{[JS91]} is given by the \textbf{zero-convention}, which requires to remove those source and sink vertices if they are decorated by zero-morphisms. The following picture shows an example of the zero-convention.

\begin{center}
\begin{tikzpicture}[scale=1]
\node (v1) at (0.5,0.5) {};
\node (v2) at (0,-1) {};
\node (v3) at (1.5,-1) {};
\draw [fill](v1) circle [radius=0.05];
\draw [fill](v2) circle [radius=0.05];
\draw [fill](v3) circle [radius=0.05];
\draw  (0.5,0.5) -- (0,-1)[postaction={decorate, decoration={markings,mark=at position .65 with {\arrow[black]{stealth}}}}];
\draw  (0.5,0.5) -- (1.5,-1)[postaction={decorate, decoration={markings,mark=at position .65 with {\arrow[black]{stealth}}}}];
\node at (3,-0.25) { $\Longleftrightarrow$};
\node (v4) at (4.5,0.5) {};
\node (v5) at (4.5,-1) {};
\node (v6) at (5.5,0.5) {};
\node (v7) at (5.5,-1) {};
\draw  (4.5,0.5) -- (4.5,-1)[postaction={decorate, decoration={markings,mark=at position .65 with {\arrow[black]{stealth}}}}];
\draw  (5.5,0.5)-- (5.5,-1)[postaction={decorate, decoration={markings,mark=at position .65 with {\arrow[black]{stealth}}}}];
\draw [fill](v5) circle [radius=0.05];
\draw [fill](v7) circle [radius=0.05];
\node [scale=0.7]at (0.2,0.2) {$1$};
\node  [scale=0.7]at (1,0.2) {$2$};
\node  [scale=0.7]at (0,-0.4) {$X$};
\node  [scale=0.7]at (1.4,-0.4) {$Y$};
\node [scale=0.7] at (0.8,0.8) {$0_{I,\  X\otimes Y}:I\to X\otimes Y$};
\node [scale=0.7]at (4.2,-0.4) {$X$};
\node [scale=0.7]at (5.8,-0.4) {$Y$};
\node at (6.8,-0.25) {$\Longleftrightarrow$};
\node (v8) at (9,0.4) {};
\node (v9) at (8.4,-1) {};
\node (v10) at (10,-1) {};
\draw [fill](v8) circle [radius=0.05];
\draw [fill](v9) circle [radius=0.05];
\draw [fill](v10) circle [radius=0.05];
\draw  (9,0.4)-- (8.4,-1)[postaction={decorate, decoration={markings,mark=at position .65 with {\arrow[black]{stealth}}}}];
\draw  (9,0.4)-- (10,-1)[postaction={decorate, decoration={markings,mark=at position .65 with {\arrow[black]{stealth}}}}];
\node  [scale=0.7] at (8.6,0.2) {$2$};
\node  [scale=0.7] at (9.4,0.2) {$1$};
\node  [scale=0.7] at (8.4,-0.4) {$X$};
\node  [scale=0.7] at (9.8,-0.4) {$Y$};
\node [scale=0.7]  at (9.2,0.8) {$0_{I,\  Y\otimes X}:I\to Y\otimes X$};
\end{tikzpicture}
\end{center}

Two (open) diagrams on $G$ are called \textbf{externally gauge equivalent} if they have  the same edge-decoration $d_o$ and for each vertex $v$ of $G$ ,  its two vertex-decorations $d_m(v), d'_m(v)$ satisfy the relation $$d_m'(v)=s_\tau\circ d_m(v)\circ s_{\sigma^{-1}},$$ where $\sigma:I(v)\to I(v)$, $\tau: O(v)\to O(v)$ are two permutations connecting its two polarizations and $$s_{\sigma^{-1}}: d_o(h_{\sigma(1)})\otimes\cdots d_o(h_{\sigma(k)})\to d_o(h_{1})\otimes\cdots d_o(h_{k}),$$
$$s_\tau:d_o(h'_1)\otimes\cdots d_o(h'_l)\to d_o(h'_{\tau(1)})\otimes\cdots d_o(h'_{\tau(l)})$$
are the natural isomorphisms given by the symmetry of $\mathbf{S}$. The external gauge equivalence relation can be depicted as follows.
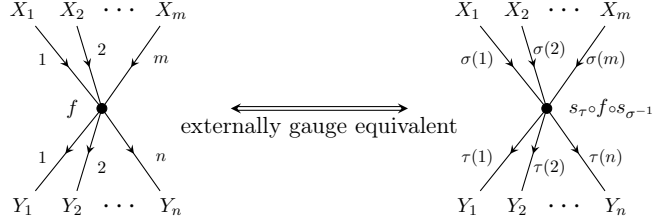
\begin{center}
\begin{tikzpicture}[scale=1.3]
\node (v11) at (-2.2,0){};
\draw [fill](v11) circle [radius=0.05];
\node [scale=0.7]at (-2.5,0){$f$};
\node (v2) at (2.3,0) {};
\node [scale=0.7] at (2.96,0) {$s_\tau\scriptstyle{\circ}\textstyle f\scriptstyle{\circ}\textstyle s_{\sigma^{-1}}$};
\draw [fill](v2) circle [radius=0.05];
\node [scale=0.7](v10) at (-3,1) {$X_1$};
\node[scale=0.7] (v12) at (-2.5,1) {$X_2$};
\node at (-2,1) {$\cdots$};
\node[scale=0.7] (v13) at (-1.5,1) {$X_m$};
\node [scale=0.7](v14) at (-3,-1) {$Y_1$};
\node[scale=0.7] (v15) at (-2.5,-1) {$Y_2$};
\node at (-2,-1) {$\cdots$};
\node [scale=0.7](v16) at (-1.5,-1) {$Y_n$};
\node[scale=0.7] (v1) at (1.5,1) {$X_1$};
\node[scale=0.7] (v3) at (2,1) {$X_2$};
\node (v4) at (2.5,1) {$\cdots$};
\node [scale=0.7](v5) at (3,1) {$X_m$};
\node [scale=0.7](v6) at (1.5,-1) {$Y_1$};
\node [scale=0.7](v7) at (2,-1) {$Y_2$};
\node (v8) at (2.5,-1) {$\cdots$};
\node [scale=0.7](v9) at (3,-1) {$Y_n$};
\draw  (v1) -- (2.3,0)[postaction={decorate, decoration={markings,mark=at position .5 with {\arrow[black]{stealth}}}}];
\draw  (v3) -- (2.3,0)[postaction={decorate, decoration={markings,mark=at position .5 with {\arrow[black]{stealth}}}}];
\draw  (v5) --(2.3,0)[postaction={decorate, decoration={markings,mark=at position .5 with {\arrow[black]{stealth}}}}];
\draw  (2.3,0) -- (v6)[postaction={decorate, decoration={markings,mark=at position .55 with {\arrow[black]{stealth}}}}];

\draw  (2.3,0) -- (v9)[postaction={decorate, decoration={markings,mark=at position .55 with {\arrow[black]{stealth}}}}];
\draw  (v10) -- (-2.2,0)[postaction={decorate, decoration={markings,mark=at position .5 with {\arrow[black]{stealth}}}}];
\draw  (v12) -- (-2.2,0)[postaction={decorate, decoration={markings,mark=at position .5 with {\arrow[black]{stealth}}}}];
\draw  (v13) -- (-2.2,0)[postaction={decorate, decoration={markings,mark=at position .5 with {\arrow[black]{stealth}}}}];
\draw  (-2.2,0) -- (v14)[postaction={decorate, decoration={markings,mark=at position .55 with {\arrow[black]{stealth}}}}];
\draw  (-2.2,0) -- (v15)[postaction={decorate, decoration={markings,mark=at position .55 with {\arrow[black]{stealth}}}}];
\draw  (-2.2,0) -- (v16)[postaction={decorate, decoration={markings,mark=at position .55 with {\arrow[black]{stealth}}}}];
\draw  (2.3,0) --(v7)[postaction={decorate, decoration={markings,mark=at position .55 with {\arrow[black]{stealth}}}}];
\node[scale=0.6] at (-2.8,0.5) {$1$};
\node [scale=0.6]at (-2.2,0.6) {$2$};
\node[scale=0.6] at (-1.6,0.5) {$m$};
\node[scale=0.6] at (-2.8,-0.5) {$1$};
\node [scale=0.6]at (-2.2,-0.6) {$2$};
\node [scale=0.6]at (-1.6,-0.5) {$n$};
\node [scale=0.6]at (1.6,0.5) {$\sigma(1)$};
\node[scale=0.6] at (2.33,0.6) {$\sigma(2)$};
\node [scale=0.6]at (2.9,0.5) {$\sigma(m)$};
\node [scale=0.6]at (1.6,-0.5) {$\tau(1)$};
\node [scale=0.6]at (2.33,-0.6) {$\tau(2)$};
\node [scale=0.6]at (2.9,-0.5) {$\tau(n)$};
\node (v17) at (-1,0) {};
\node (v18) at (1,0) {};
\draw [double distance=1pt,stealth-stealth] (v17) -- node[sloped,scale=0.8, below] {externally gauge equivalent}(v18);
\end{tikzpicture}
\end{center}

Two (open) diagrams on $G$ are called \textbf{internally gauge equivalent} if they have the same polarizations and their decorations $d_o,d_o', d_m, d_m'$ satisfy the following conditions:
$(1)$ for any edge $e$ of $G$, its decorations $d_o(e), d_o'(e)$ are naturally isomorphic, that is, there exist finite objects $A_1,\cdots, A_n$ and a permutation $\sigma:\{1,\cdots,n\}\to \{1,\cdots,n\}$ such that  $d_o(e)$ and $d'_o(e)$ can be represented as $A_1\otimes \cdots\otimes A_n$ and $A_{\sigma(1)}\otimes \cdots\otimes A_{\sigma(n)}$, respectively, with the natural isomorphism $s_\sigma:A_1\otimes \cdots\otimes A_n \to A_{\sigma(1)}\otimes \cdots\otimes A_{\sigma(n)}$ given by the symmetry of $\mathbf{S}$;
$(2)$ for each vertex $v$ of $G$ with incoming-edges $h_1<\cdots< h_k$ and outgoing-edges $h'_1<\cdots<h'_l$,   its two vertex-decorations $d_m(v), d'_m(v)$ satisfy the relation $$d_m'(v)=(s_{\tau_1}\otimes\cdots\otimes s_{\tau_l})\circ d_m(v)\circ (s_{\sigma_1^{-1}}\otimes\cdots\otimes s_{\sigma_k^{-1}}),$$ where $s_{\sigma_i^{-1}}$ $(1\leq i\leq k)$ is the natural isomorphism, associated with $h_i$, from the decoration $\lambda'_o(h_i)=A_{i,\sigma_i(1)}\otimes\cdots\otimes A_{i,\sigma_i(\alpha_i)}$ to the decoration $\lambda_o(h_i)=A_{i,1}\otimes\cdots\otimes A_{i,\alpha_i}$, and  $s_{\tau_j}$ $(1\leq j\leq l)$ is the natural isomorphism, associated with $h'_j$, from the decoration $\lambda_o(h'_j)=B_{j,1}\otimes\cdots\otimes B_{j,\beta_j}$ to the decoration  $\lambda'_o(h'_j)=B_{j,\tau_j(1)}\otimes\cdots\otimes B_{j,\tau_j(\beta_j)}$. The following picture shows an example of internal gauge equivalence.
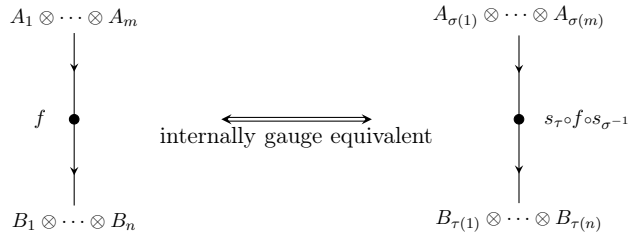
\begin{center}
\begin{tikzpicture}[scale=0.9]
\node[scale=0.7] (v3) at (-2.5,2) {$B_1\otimes\cdots\otimes B_n$};
\node (v2) at (-2.5,3.5) {};
\draw [fill](v2) circle [radius=0.07];
\node [scale=0.7] (v1) at (-2.5,5) {$A_1\otimes \cdots\otimes A_m$};
\node (v5) at (4,3.5) {};
\node [scale=0.7] (v4) at (4,5) {$A_{\sigma(1)}\otimes \cdots\otimes A_{\sigma(m)}$};
\draw [fill](v5) circle [radius=0.07];
\node [scale=0.7] (v6) at (4,2) {$B_{\tau(1)}\otimes\cdots\otimes B_{\tau(n)}$};

\draw  (v1) --(-2.5,3.5)[postaction={decorate, decoration={markings,mark=at position .45 with {\arrow[black]{stealth}}}}];
\draw (-2.5,3.5) -- (v3)[postaction={decorate, decoration={markings,mark=at position .65 with {\arrow[black]{stealth}}}}];
\draw  (v4) -- (4,3.5)[postaction={decorate, decoration={markings,mark=at position .45 with {\arrow[black]{stealth}}}}];
\draw (4,3.5) -- (v6)[postaction={decorate, decoration={markings,mark=at position .65 with {\arrow[black]{stealth}}}}];
\node [scale=0.7] at (-3,3.5) {$f$};
\node [scale=0.7] at (5,3.5) {$s_{\tau}\scriptstyle{\circ}\textstyle f\scriptstyle{\circ}\textstyle s_{\sigma^{-1}}$};
\node(v17) at (-0.5,3.5) {};
\node (v18)at (2,3.5) {};
\draw [double distance=1pt,stealth-stealth] (v17) -- node[sloped,scale=0.8, below] {internally gauge equivalent}(v18);
\end{tikzpicture}
\end{center}

The \textbf{gauge equivalence relation} of (open) diagrams is the minimal equivalence relation generated by the external and internal gauge equivalence relations. It is not difficult to see that the zero-convention is compatible with the gauge equivalence relation. For any causal-net $G$, the set $\mathcal{F}_\mathbf{S}(G)$ is defined to be the set of equivalence classes of (open) diagrams on $G$ under the gauge equivalence relation. There are two versions of $\mathcal{F}_\mathbf{S}(G)$: one is defined through close-diagrams; the other is through open-diagrams. The first version is called a \textbf{closed causal-algebra} and the second version is called an $\textbf{open causal-algebra}$.  When the versions are irrelevant, we simply call them causal-algebras. By the definition of $\mathcal{F}_\mathbf{S}$, the fundamental theorem (Theorem \ref{com}) and  the graphical calculus for symmetric monoidal categories \cite{[JS91]}, it is not difficult to show that both versions of $\mathcal{F}_\mathbf{S}$ are functors (or pre-cosheaves). Moreover, we have the following theorem.

\begin{thm}
The construction of causal-algebra $\mathcal{F}_{\mathbf{S}}$ defines a functor from the category of small symmetric monoidal categories and symmetric monoidal functors to the category of pre-cosheaves on $\mathbf{Cau}$.
\end{thm}

From our point of view, the Baez construction for quantum gauge theories \cite{[B96],[B94]} and the wave function renormalization scheme for topological orders \cite{[LW04]} are just two sides of one coin, that is,  the studies of different aspects of background independent quantum measures, and both of them can be put into the framework of causal-net condensation. There are many interesting features of causal-net condensation, such as coarse-graining, causality and  homeomorphic invariance, etc., which show its potential applications in the theory of quantum gravity. It is expected that there are many concrete connections among the Baez construction, wave function renormalization, causal-net condensation and the theory of quantum gravity, where the category $\mathbf{Cau}$ will play an indispensable role.
\begin{rem}
It is well-known that causal-nets play important roles in the field of artificial intelligence and it is not surprising that causality would serve as a fundamental structure in the theory of quantum gravity. Therefore, it is natural to expect that the category $\mathbf{Cau}$ and the framework of causal-net condensation would be beneficial both for the development of artificial intelligence and for the enhancement of connections among artificial intelligence, quantum gravity and category theory.
\end{rem}
\begin{rem}
Just as that appearing in topological graph theory \cite{[A96]}, a polarization structure \cite{[JS91]} (linear orders on $I(v)$ and $O(v)$ for each vertex $v$) on a causal-net $G$ turns out to be equivalent to an upward embedding of $G$ in a surface. We guess that it is a hint about some connections between causal-net condensation and string theory.
\end{rem}

\subsection{Categorical understanding of graph theory}\label{cug}
In this paper, we mainly focus on the other application of $\mathbf{Cau}$, that is, to demonstrate a new idea of bringing together graph theory and category theory, which is summarized as the following slogan.
\begin{center}
  Graph Theory $=$ Kleisli Category
\end{center}

It is better to explain this slogan at first from a geometrical perspective. It is well-known that a \textbf{conformal space-time} $\mathcal{M}$ can be totally characterized by its causal structure \cite{[Ma77]}, which is usually represented by the causal relation of physical events (points of $\mathcal{M}$). From a categorical point of view, a natural way to represent the causal structure of a conformal space-time should be the category of causal curves (non-spacelike curves), just like the way to represent a topological space by its path category (as mentioned in the former subsection). Meanwhile, a \textbf{continuous causal-preserving  mapping}  between two conformal space-times  should then be represented as a functor between their categories of causal curves. This way of categorifying conformal space-times is  similar in spirit to  Baez's categorical understanding of gauge theory \cite{[B96]}, which interprets a connection on a principle fiber bundle as a functor from the path category of the base space to the category of right principal homogeneous spaces of the structure group, and interprets a gauge transformation as a natural transformation.

Moreover, taking points as objects and causal curves as morphisms, the category of causal curves of a conformal space-time turns out to be an interesting mathematical object, which is called an \textbf{acyclic category}  \cite{[K08]}  (see Definition \ref{acy}),  a small category with all endo-morphisms and isomorphisms being identity morphisms. The famous theorem of Penrose, Hawking, Malament, etc., \cite{[P72],[HKM76],[Ma77]} tells us that the category of conformal space-times and continuous causal-preserving mappings can be fully faithfully embedded into the category of acyclic categories and their functors, which means that acyclic categories are indeed good categorical models of conformal space-times.

Thinking alone this line, nothing would prevent us from viewing a free acyclic category as a categorical model of an affine conformal space-time, from viewing a causal-net as a model of a local coordinate chart (or a finite conformal space-time), and from viewing a morphism of causal-nets as a model of a coordinate transformation (or a continuous causal-preserving mapping between finite conformal space-times). These similarities provide us a new vision of space-time, where a causal-net in a conformal space-time is analogous to an open set in a topological space, a morphism of causal-nets is analogous to a restriction of open sets. In this topological scenario, causal-net condensation can be understood as a kind of cosheaf theory on conformal space-times. We list these analogues in the following table.

\begin{center}

\begin{tikzpicture}[scale=1]
\draw  (-3.5,0) rectangle (8,-4);
\node [scale=0.8] (v9) at (-1,-0.5) {acyclic category};
\node [scale=0.8] (v10) at (5,-0.5) {CST};
\node (v1) at (2,0) {};
\node (v2) at (2,-4) {};
\draw  (2,0) edge (2,-4);
\node (v3) at (0,-1) {};
\node (v4) at (4,-1) {};
\draw  (-3.5,-1) edge (8,-1);
\node (v5) at (0,-2) {};
\node (v6) at (4,-2) {};
\node (v7) at (0,-3) {};
\node (v8) at (4,-3) {};
\draw  (-3.5,-2) edge (8,-2);
\draw  (-3.5,-3) edge (8,-3);
\node [scale=0.8] (v11) at (-1,-1.5) {free acyclic category};
\node  [scale=0.8] (v12) at (5,-1.5) {affine CST};
\node  [scale=0.8] (v15) at (-1,-3.5) {morphism of causal-nets};
\node  [scale=0.8] (v16) at (5,-3.5) {coordinate transformation};
\node [scale=0.8] (v13) at (-1,-2.5) {causal-net};
\node [scale=0.8] (v14) at (5,-2.5) {finite coordinate chart};
\draw  (-3.5,1) rectangle (2,0);
\draw  (2,1) rectangle (8,0);
\node[scale=0.8] at (-1,0.5) {\textbf{Combinatorics}};
\node [scale=0.8] at (5,0.5) {\textbf{Geometry}};
\draw [latex-latex] (v9) edge (v10);
\draw   [latex-latex](v11) edge (v12);
\draw  [latex-latex] (v13) edge (v14);
\draw   [latex-latex](v15) edge (v16);
\draw  (-3.5,-4) rectangle (2,-6);
\draw  (2,-4) rectangle (8,-5);
\node [scale=0.8] (v18) at (-1,-4.5) { functor between acyclic categories};
\node [scale=0.8] (v19) at (5,-4.5) {CCM between CSTs};
\draw[latex-latex]  (v18) edge (v19);
\draw  (-3.5,-5) rectangle (8,-6);
\node [scale=0.8] (v17) at (-1,-5.5) {causal-net condensation};
\node [scale=0.8] (v20) at (5,-5.5) {cosheaf on CST};
\draw  [latex-latex]  (v17) edge (v20);
\node  [scale=0.8] at (2,-6.5) {CST = conformal space-time};
\node  [scale=0.8] at (2,-7) {CCM = continuous causal-preserving mapping};
\end{tikzpicture}
\end{center}

Reasonably, the construction of the category of causal curves of a conformal space-time and the construction of the path category of a topological space can be viewed, respectively, as a geometrical and a topological analogue of the Kleisli construction in category theory. We think that these analogies are not accidental and these path constructions are in fact different aspects of one unified theme concerning the nature of space-time, gauge fields and the mathematical language of category theory.

As a brief summation, the slogan "graph theory = Kleisli category" just emphasizes the methodological aspect of these ideas, but what is more convincing should be its geometrical background. It is the categorification of conformal space-times that leads to this new categorical understanding of graph theory. A more intuitive and enlightening slogan  would be stated as follows.
\begin{center}
Causal-net = Finite conformal space-time
\end{center}

Treating these analogies seriously, we benefit a lot. Concretely, we find that many basic graph-theoretic operations on causal-nets can be naturally represented by morphisms of causal-nets (functors between path categories), and some common relations of causal-nets, such as coarse-graining, merging, contraction, immersion-minor, topological minor, etc., can be characterized by composition-closed classes of morphisms. For minor relation, we introduce a categorical framework and use it to study two types of generalized minors: coarse-graining minors and contraction minors. This categorical framework is rather general and meaningful for any small categories.
We also invent the notions of a causal-coloring and a linear-coloring, both of which are fruits of categorical thinking.
We show that a causal-coloring can be naturally interpreted as a special kind of minimal-quotient and an exact minor as a kind of sub-quotient. We believe that these facts are sufficient to justify the fundamental role of $\mathbf{Cau}$ in the theory of causal-nets and this  new category-theoretic understanding of graph theory.

Evidently, the ideas in this paper can be directly applied to general directed graphs. As for undirected graphs, the ideas may be considered in the framework of dagger categories \cite{[K19]}, and the general philosophy should be stated as follows.
\begin{center}
Ordinary Graph Theory $=$ Dagger Kleisli Category
\end{center}

This paper is organized as follows. In Section $2$,  we introduce several special types of quotients, including coarse-graining, vertex-coarse-graining, edge-coarse-graining, merging and contraction, fusion, coloring, linear-coloring, etc., explain their graph-theoretic means and  show their relations through several decomposition theorems.  We also interpret the notion of a causal-coloring of a causal-net as a special kind of minimal-quotient, that is, as a minimal fusion under edge-coarse-grainings. In Section $3$, we introduce several special but natural inclusion morphisms, such as immersion, topological embedding, subdivision and embedding, etc., explain their graph-theoretic means and  show several decomposition theorems. In Section $4$, we introduce a categorical framework for minor theory and use it to study several types of generalized minors in $\mathbf{Cau}$.

\section{Quotient=epimorphism}
In this section, we study various type of quotients, some of which are essential for understanding graphical calculi for  kinds of monoidal categories and are naturally appears in many physical theories.
In $\mathbf{Cau}$, there are two types of quotients: one is category-theoretic (epimorphism) and the other is graph-theoretic (coarse-graining). The later type turns out to be a special case of the former one.

We start with the category-theoretic notion of a quotient.

\begin{defn}\label{quo}
A morphism of causal-nets is called a \textbf{quotient} in $\mathbf{Cau}$ if it is surjective both on vertices (objects) and on edges (morphisms of length one).
\end{defn}
Definition  \ref{quo}  here coincides with the one defined in \cite{[BBP99]},
where a functor $F:\mathcal{C}_1\to \mathcal{C}_2$ between two small categories $\mathcal{C}_1$ and $\mathcal{C}_2$ is called a \textbf{quotient} in $\mathbf{Cat}$ if the images of objects and morphisms of $\mathcal{C}_1$ in $\mathcal{C}_2$, via $F$, \textbf{generate} $\mathcal{C}_2$.
In $\mathbf{Cau}$, if $\lambda:G_1\to G_2$ is a quotient, then the image of $G_1$ must generate the path category of $G_2$, therefore $\lambda$ can be viewed as a \textbf{quotient functor} \cite{[BBP99]} in  $\mathbf{Cat}$. Conversely, if a morphism $\lambda:G_1\to G_2$ of causal-nets is a quotient functor in $\mathbf{Cat}$, then by the freeness of the target of $\lambda$, we see that $\lambda$ must be a quotient in $\mathbf{Cau}$. In a word, the notions of quotients coincide in $\mathbf{Cau}$ and in $\mathbf{Cat}$.

Please make a distinction between Definition \ref{quo} and  the following more restricted notion, which also coincides with the one defined in \cite{[BBP99]}.
\begin{defn}
A morphism of causal-nets is called a \textbf{surjection} in $\mathbf{Cau}$ if it is surjective both on vertices (objects) and on directed paths (morphisms).
\end{defn}
In \cite{[BBP99]}, a functor $F:\mathcal{C}_1\to \mathcal{C_2}$ between small categories is called a \textbf{surjection} in $\mathbf{Cat}$ if the images of objects and morphisms of $\mathcal{C}_1$ in $\mathcal{C}_2$, via $F$, \textbf{span} $\mathcal{C}_2$. The notions of surjections also coincide in $\mathbf{Cau}$ and in $\mathbf{Cat}$.

Clearly, a surjection is a quotient, but the opposite is not true in general.

\begin{ex}\label{str}
The following shows an example of a quotient, which is not a surjection. The quotient $\lambda$ is defined as follows: $\lambda(v_1)=w_1$, $\lambda(v_2)=\lambda(v_3)=w_2$, $\lambda(v_4)=w_3$, $\lambda(e_1)=h_1$ and $\lambda(e_2)=h_2$. Note that the directed path $h_2h_1$ has no pre-images, so $\lambda$ is not a surjection.

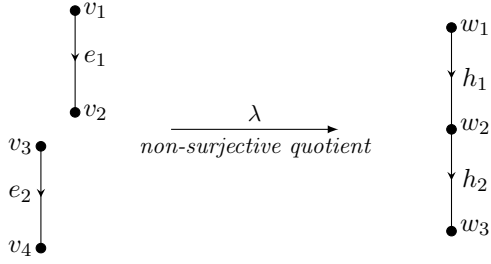
\begin{center}
\begin{tikzpicture}[scale=0.9]
\node (v1) at (0,1.5) {};
\node (v2) at (0,0) {};
\node (v3) at (-0.5,-0.5) {};
\node (v4) at (-0.5,-2) {};
\node (v5) at (5.5,1.25) {};
\node (v6) at (5.5,-0.25) {};
\node (v7) at (5.5,-1.75) {};
\draw [fill](v1) circle [radius=0.07];
\draw [fill](v2) circle [radius=0.07];
\draw [fill](v3) circle [radius=0.07];
\draw [fill](v4) circle [radius=0.07];
\draw [fill](v5) circle [radius=0.07];
\draw [fill](v6) circle [radius=0.07];
\draw [fill](v7) circle [radius=0.07];
\draw  (0,1.5) -- (0,0)[postaction={decorate, decoration={markings,mark=at position .5 with {\arrow[black]{stealth}}}}];
\draw (-0.5,-0.5) -- (-0.5,-2)[postaction={decorate, decoration={markings,mark=at position .5 with {\arrow[black]{stealth}}}}];
\draw (5.5,1.25) -- (5.5,-0.25)[postaction={decorate, decoration={markings,mark=at position .5 with {\arrow[black]{stealth}}}}];
\draw  (5.5,-0.25) -- (5.5,-1.75)[postaction={decorate, decoration={markings,mark=at position .5 with {\arrow[black]{stealth}}}}];
\node (v8) at (1.25,-0.25) {};
\node (v9) at (4,-0.25) {};
\draw[-latex]  (v8) -- node[sloped,scale=0.8, below] {non-surjective quotient}node[sloped,scale=0.8, above] {$\lambda$}(v9);
\node [scale=0.9] at (0.3,1.5) {$v_1$};
\node [scale=0.9]at (0.3,0) {$v_2$};
\node  [scale=0.9]at (-0.8,-0.5) {$v_3$};
\node  [scale=0.9]at (0.3,0.8) {$e_1$};
\node  [scale=0.9]at (-0.8,-2) {$v_4$};
\node  [scale=0.9]at (-0.8,-1.2) {$e_2$};
\node  [scale=0.9]at (5.85,1.25) {$w_1$};
\node  [scale=0.9]at (5.85,-1.7) {$w_3$};
\node  [scale=0.9]at (5.85,-0.2) {$w_2$};
\node  [scale=0.9]at (5.85,0.5) {$h_1$};
\node  [scale=0.9]at (5.85,-1) {$h_2$};
\end{tikzpicture}
\end{center}
\end{ex}

For a graph-theoretic understanding of quotients, we introduce some terminologies.
Given a morphism $\lambda:G_1\to G_2$ of causal-nets. Edges of $G_2$ are classified into three classes: $(1)$ $e\in E(G_2)$ is called a \textbf{null-edge} of $\lambda$ if the pre-image $\lambda^{-1}(e)$ is an empty-set; $(2)$ $e\in E(G_2)$ is called a \textbf{simple-edge} of $\lambda$ if $\lambda^{-1}(e)$ has exactly one element; $(3)$  $e\in E(G_2)$ is called a \textbf{multiple-edge} of $\lambda$ if $\lambda^{-1}(e)$ has at least two elements.

Similarly, vertices of $G_2$ are classified into three classes: $(1)$ $v\in V(G_2)$ is called a \textbf{null-vertex} of $\lambda$ if the pre-image $\lambda^{-1}(v)$ is an empty-set; $(2)$ $v\in V(G_2)$ is called a \textbf{simple-vertex} of $\lambda$ if $\lambda^{-1}(v)$ has exactly one element; $(3)$  $v\in V(G_2)$ is called a \textbf{multiple-vertex} of $\lambda$ if $\lambda^{-1}(v)$ has at least two elements.

By definition, a morphism is a quotient if and only if it has no null-vertices and no null-edges. Since any morphism $\lambda:G_1\to G_2$ of causal-nets is a functor between free categories, then any pre-image of an edge $e$ of $G_2$, if exists,  must be an edge of $G_1$. By this, it is not difficult to see that quotients are closed under composition. Clearly, surjections are also closed under composition.

The following result shows that quotients are exactly epimorphisms in $\mathbf{Cau}$.
\begin{thm}
In $\mathbf{Cau}$, a morphism is a quotient if and only if it is an epimorphism.
\end{thm}
\begin{proof}
$(\Rightarrow)$ Let $\lambda:G_1\to G_2$ be a quotient and $\lambda_1:G_2\to G_3$, $\lambda_2:G_2\to G_3$ be any two morphisms such that $\lambda_1\circ \lambda=\lambda_2\circ\lambda$. We want to show that $\lambda_1=\lambda_2$. For this, we only need to show that for each $w\in V(G_2)$, $\lambda_1(w)=\lambda_2(w)$ and for each $h\in E(G_2)$, $\lambda_1(h)=\lambda_2(h)$. Since $\lambda$ is a quotient, then  for each $w\in V(G_2)$ and each $h\in E(G_2)$, there exist a $v\in V(G_1)$ and an  $e\in E(G_1)$, such that $w=\lambda(v)$ and $h=\lambda(e)$. Since $\lambda_1\circ \lambda=\lambda_2\circ\lambda$, then we have $\lambda_1(w)=\lambda_1(\lambda(v))=\lambda_2(\lambda(v))=\lambda_2(w)$ and  $\lambda_1(h)=\lambda_1(\lambda(e))=\lambda_2(\lambda(e))=\lambda_2(h)$.

$(\Leftarrow)$ Let $\lambda:G_1\to G_2$ be an epimorphism, we prove by contradiction that it has no null-vertices and no null-edges. Suppose $v_0\in V(G_2)$ be a null-vertex of $\lambda$, we consider the causal-net $G_3=G_2+\{w_1,w_2\}$, which is obtained from $G_2$ by adding two isolated vertices $w_1,w_2$, and define two morphisms $\lambda_1,\lambda_2:G_2\to G_3$ as follows. For any $v\in V(G_2)-\{v_0\}$, $\lambda_1(v)=\lambda_2(v)=v$, $\lambda_1(v_0)=w_1,\lambda_2(v_0)=w_2$; for any $e\in E(G_2)$, $\lambda_1(e)=\lambda_2(e)=e$. Clearly, $\lambda_1\circ\lambda=\lambda_2\circ \lambda$, but $\lambda_1\neq\lambda_2$, which leads a contradiction.

Similarly, suppose $e_0\in E(G_2)$ be a null-edge of $\lambda$, we consider the causal-net $G_3=G_2+\{e_1,e_2\}$, which is obtained from $G_2$ by adding two directed paths $e_1,e_2$ with $s(e_1)=s(e_2)=s(e_0)$ and $t(e_1)=t(e_2)=t(e_0)$, and define two morphisms $\lambda_1,\lambda_2:G_2\to G_3$ as follows. For any $v\in V(G_2)$, $\lambda_1(v)=\lambda_2(v)=v$; for any $e\in E(G_2)-\{e_0\}$, $\lambda_1(e)=\lambda_2(e)=e$, $\lambda_1(e_0)=e_1, \lambda_2(e_0)=e_2$. Clearly, $\lambda_1\circ\lambda=\lambda_2\circ \lambda$, but $\lambda_1\neq\lambda_2$, which leads a contradiction.
\end{proof}
\subsection{Coarse-graining and quotient causal-net}
In this subsection, we introduce the notion of a coarse-graining, which characterizes the graph-theoretic notion of a quotient-causal-net and naturally appears in the monadic description of a monoidal category \cite{[HLY15],[HLY16]}.

As before, given a morphism $\lambda:G_1\to G_2$, edges of $G_1$ are classified into three classes: $(1)$ $e\in E(G_1)$ is called a \textbf{contraction} of $\lambda$ if $l(\lambda(e))=0$, that is, the length $l(\lambda(e))$ of the directed path $\lambda(e)$ is $0$, which means that $\lambda(e)$ is an identity morphism; $(2)$ $e\in E(G_1)$ is called a \textbf{segment} of $\lambda$ if $l(\lambda(e))=1$, that is, $\lambda(e)$ is an edge; $(3)$   $e\in E(G_1)$ is called a \textbf{subdivision} of $\lambda$ if $l(\lambda(e))\geq 2$, that is, $\lambda(e)$ is a directed path of length no less than two. We have $$E(G_1)=Con(\lambda)\sqcup Seg(\lambda)\sqcup Subd(\lambda),$$
where $Con(\lambda), Seg(\lambda)$ and $Subd(\lambda)$ denote the sets of contractions, segments and subdivisions of $\lambda$, respectively.
\begin{defn}
A quotient is called a \textbf{coarse-graining} if it has no subdivisions.
\end{defn}

In other words, a coarse-graining is a morphism with no null-vertices, no null-edges and no subdivisions. If $\lambda:G\to H$ is a coarse-graining, then $Subd(\lambda)=\emptyset$ and $E(G)=Seg(\lambda)\sqcup Con(\lambda)$. Moreover, $\lambda$ maps $V(G)$ and  $Seg(\lambda)$ surjectively onto $V(H)$ and $E(H)$, respectively. Clearly, coarse-grainings are closed under composition.

\begin{ex}
In the following figure, $\lambda_1,\lambda_2$ show two examples of coarse-grainings of causal nets, where $\lambda_1(v_1)=\lambda_1(v_2)=w_1$, $\lambda_1(v_3)=w_2$, $\lambda_1(e_1)=\lambda_1(e_2)=Id_{w_1}$, $\lambda_1(e_3)=\lambda_1(e_4)=h$; and
$\lambda_2(v_1)=w_1$, $\lambda_2(v_2)=\lambda_2(v_3)=w_2$, $\lambda_2(e_1)=\lambda_2(e_2)=\lambda_2(e_4)=h$, $\lambda_2(e_3)=Id_{w_2}$.
\begin{figure}[H]
\centering
\begin{tikzpicture}[scale=0.9]
\node (v1) at (-0.5,0.5) {};
\node (v2) at (-1,-1.5) {};
\node (v3) at (0.8,-2.4) {};
\draw [fill](v1) circle [radius=0.07];
\draw [fill](v2) circle [radius=0.07];
\draw [fill](v3) circle [radius=0.07];
\draw  (-0.5,0.5) -- (-1,-1.5) [postaction={decorate, decoration={markings,mark=at position .5 with {\arrow[black]{stealth}}}}];
\draw  (-1,-1.5)  -- (0.8,-2.4)[postaction={decorate, decoration={markings,mark=at position .5 with {\arrow[black]{stealth}}}}];
\draw  (-0.5,0.5) -- (0.8,-2.4)[postaction={decorate, decoration={markings,mark=at position .5 with {\arrow[black]{stealth}}}}];
\draw  plot[smooth, tension=.7] coordinates {(-0.5,0.5) (-1.2,-0.6) (-1,-1.5)}[postaction={decorate, decoration={markings,mark=at position .6 with {\arrow[black]{stealth}}}}];
\node at (-0.6,0.8) {$v_1$};
\node at (-1.3,-1.6) {$v_2$};
\node at (1.2,-2.4) {$v_3$};
\node at (-1.2,0) {$e_1$};
\node at (-0.6,-0.8) {$e_2$};
\node at (-0.2,-2.2) {$e_3$};
\node at (0.4,-0.6) {$e_4$};
\node (v4) at (4,0.6) {};
\node (v5) at (4.8,-2.4) {};
\draw [fill](v4) circle [radius=0.07];
\draw [fill](v5) circle [radius=0.07];
\draw  (4,0.6) --(4.8,-2.4)[postaction={decorate, decoration={markings,mark=at position .5 with {\arrow[black]{stealth}}}}];
\node at (4.4,0.6) {$w_1$};
\node at (5.2,-2.4) {$w_2$};
\node at (1.6,-0.6) {};
\node at (3.4,-0.6) {};
\node at (4.6,-0.8) {$h$};
\draw[-latex]  (1.3,-0.6) -- node[sloped,scale=0.8, below] {coarse-graining}node[sloped,scale=0.8, above] {$\lambda_1$ or $\lambda_2$}(3.4,-0.6);
\end{tikzpicture}
\end{figure}
\end{ex}

\begin{rem}
One feature of our approach to graph theory is that it provides a natural way (using a morphism) to represent the graphical operation of contracting an edge. In our setting, "contracting an edge $e$" is equivalently formulated as "mapping  $e$ to an identity".
\end{rem}

\begin{ex} \label{r2}
The following morphism is a surjection but not a coarse-graining, where $\lambda(v_1)=w_1$, $\lambda(v_2)=w_2$, $\lambda(v_3)=w_3$, $\lambda(e_1)=h_1$, $\lambda(e_2)=h_2$ and $\lambda(e_3)=h_2h_1$.  The edge $e_3$ is a subdivision of $\lambda$.

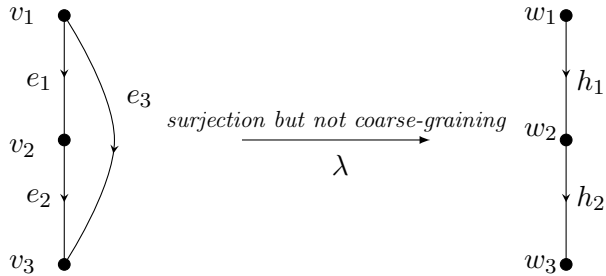
\begin{center}
\begin{tikzpicture}[scale=1.1]

\node (v1) at (-0.5,1) {};
\node (v2) at (-0.5,-0.5) {};
\node (v3) at (-0.5,-2) {};
\draw  (-0.5,1)-- (-0.5,-0.5)[postaction={decorate, decoration={markings,mark=at position .5 with {\arrow[black]{stealth}}}}];
\draw  (-0.5,-0.5) -- (-0.5,-2)[postaction={decorate, decoration={markings,mark=at position .5 with {\arrow[black]{stealth}}}}];
\draw  plot[smooth, tension=.7] coordinates {(-0.5,1) (0.1,-0.5) (-0.5,-2)}[postaction={decorate, decoration={markings,mark=at position .55 with {\arrow[black]{stealth}}}}];

\node (v4) at (5.5,1) {};
\node (v5) at (5.5,-0.5) {};
\node (v6) at (5.5,-2) {};

\draw  (5.5,1) -- (5.5,-0.5)[postaction={decorate, decoration={markings,mark=at position .5 with {\arrow[black]{stealth}}}}];
\draw  (5.5,-0.5)-- (5.5,-2)[postaction={decorate, decoration={markings,mark=at position .5 with {\arrow[black]{stealth}}}}];
\node (v7) at (1.5,-0.5) {};
\node (v8) at (4,-0.5) {};
\draw[-latex]  (v7) --node[sloped,scale=0.8, above] {surjection but not coarse-graining} (v8);
\node at (-1,1) {$v_1$};
\node at (-1,-0.6) {$v_2$};
\node at (-1,-2) {$v_3$};

\node at (-0.8,0.2) {$e_1$};
\node at (-0.8,-1.2) {$e_2$};
\node at (0.4,0) {$e_3$};
\node at (5.2,1) {$w_1$};
\node at (5.2,-0.4) {$w_2$};
\node at (5.2,-2) {$w_3$};
\node at (5.8,0.2) {$h_1$};
\node at (5.8,-1.2) {$h_2$};
\node at (2.8,-0.8) {$\lambda$};
\draw [fill](v1) circle [radius=0.07];
\draw [fill](v2) circle [radius=0.07];
\draw [fill](v3) circle [radius=0.07];
\draw [fill](v4) circle [radius=0.07];
\draw [fill](v5) circle [radius=0.07];
\draw [fill](v6) circle [radius=0.07];
\end{tikzpicture}
\end{center}

\end{ex}
The quotient in Example \ref{str} is a coarse-graining but not a surjection. In $\mathbf{Cau}$, the relations among quotients, surjections and coarse-grainings can be depicted as follows.

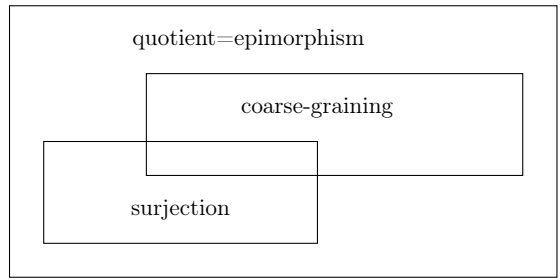
\begin{center}
\begin{tikzpicture}[scale=0.9]
\draw  (-0.5,-1) rectangle (5,-2.5);
\draw  (-2,-2) rectangle (2,-3.5);
\node [scale=0.8] at (0,-3) {surjection};
\node [scale=0.8] at (2,-1.5) {coarse-graining};
\draw  (-2.5,0) rectangle (5.5,-4);
\node [scale=0.8] at (1,-0.5) {quotient=epimorphism};
\end{tikzpicture}
\end{center}

A subset $S$ of edges of causal-net $G$ is called \textbf{induced} if  $e\in S$ implies that all edges connecting $s(e)$ and $t(e)$ are in $S$, or equivalently,  $e\in S$ if and only if the multi-edge $\varepsilon$ containing $e$ is contained inside $S$. In other words, an induced subset of edges is just a disjoint union of several multi-edges.  Immediately, a subset $S$ of edges is induced if and only if its complement $E(G)-S$ is induced. A partition of $E(G)$  is called \textbf{induced} if each of its blocks is induced. An induced partition of edge set can be equivalently viewed as a partition of the set of multi-edges.

A sub-causal-net $H$ of $G$ is called \textbf{induced} (or \textbf{vertex-induced}), if any edges of $G$ connecting two vertices of $H$ are edges of $H$. An induced sub-causal-net is totally determined by its vertex set. The edge set of a sub-causal-net is induced, but the converse is not true in general, that is, an induced subset of edges may not be the edge set of an induced sub-causal-net. A sub-causal-net $H$ of $G$ is called \textbf{edge-induced}, if the edge set of $H$ is an induced subset of edges of $G$. An edge-induced sub-causal-net is totally determined by its edge set. Clearly, any induced sub-causal-net is edge-induced.

\begin{defn}\label{quotient}
Let $H$ and $G$ be two causal-nets. $H$ is called a \textbf{quotient-causal-net} of $G$, if
there exist an equivalence relation $\sim_\nu$ on $V(G)$, an induced partition $E(G)=S\sqcup C$  and an equivalence relation $\sim_\epsilon$ on the induced block $S$, such that $(1)$ for any $e\in S$, $s(e)\not\sim_\nu t(e)$; $(2)$ for any $e\in C$, $s(e)\sim_\nu t(e)$; $(3)$ for any $e_1, e_2\in S$, $e_1\sim_\epsilon e_2$ implies that $s(e_1)\sim_\nu s(e_2)$ and $t(e_1)\sim_\nu t(e_2)$; $(4)$ $V(H)=\displaystyle{\frac{V(G)}{\sim_\nu}}$; $(5)$ $E(H)=\displaystyle{\frac{S}{\sim_\epsilon}}$.
\end{defn}
In other words, $\sim_{\varepsilon}$ is a partial equivalence relation with the induced subset $S$ of edges as domain.
We call $S$  a \textbf{segment set} of $G$, and call its complement $C$ a \textbf{contraction set}. Elements of $S$ and elements of $C$ are called \textbf{segments} and \textbf{contractions} of $G$, respectively.
We call the equivalence relation $\sim_\nu$ a \textbf{contraction-relation} of $G$ and call the equivalence relation $\sim_\epsilon$ a \textbf{segment-relation} of $G$.

The conditions $(1)$ and $(2)$ in the definition can be combined into an equivalent condition that for any $e\in E(G)$, $s(e)\sim_\nu t(e)$ if and only if $e$ is a contraction, that is, $e\in C$.
The source and target maps of the quotient-causal-net $H$ are defined in the natural and unique way that for any segment $e\in S$, $s([e])=[s(e)]$ and $t([e])=[t(e)]$, where we use the notations $[e]$ and $[v]$ to represent the equivalence classes of segment $e$ and vertex $v$, respectively. The condition $(3)$ is the compatible condition of $\sim_\nu$ and $\sim_\epsilon$, from which we can easily see that the source and target maps are well defined.

The following theorem shows that quotient-causal-nets are exactly characterized by coarse-grainings.
\begin{thm}\label{cc}
Let $H$ and $G$ be two causal-nets. $H$ is a quotient-causal-net of $G$ if and only if there is a coarse-graining $\lambda:G\to H$.
\end{thm}
\begin{proof}
$(\Leftarrow).$  Let $\lambda:G\to H$ be a coarse-graining, then it is surjective both on vertices and edges. Therefore, we  have $$V(H)=\displaystyle{\frac{V(G)}{\sim_\lambda}},$$ where we use $\sim_f$ to denote the standard equivalence relation on a set $X$ defined by a map $f:X\to Y$ of sets in the way that $x_1\sim_f x_2\Longleftrightarrow f(x_1)=f(x_2)$, $x_1,x_2\in X$.

Since $\lambda$ is surjective on edges and all contraction of $\lambda$ are mapped to identity morphisms (directed paths with zero length), so we have $$E(H)=\displaystyle{\frac{Seg(\lambda)}{\sim_\lambda}}.$$

Now we show that $Seg(\lambda)$ is induced. Let $e_1,e_2$ be two parallel edges of $G$, that is, $s(e_1)=s(e_2)$ and $t(e_1)=t(e_2)$. If $e_1\in Seg(\lambda)$, then $s(\lambda(e_1))\neq t(\lambda(e_1))$, which implies that $s(\lambda(e_2))\neq t(\lambda(e_2))$. Therefore, $\lambda(e_2)$ is either a segment or a subdivision. But $Subd(\lambda)=\emptyset$, so we must have $e_2\in Seg(\lambda)$, from which it follows that $Seg(\lambda)$ is induced.

Now we check the remain conditions in the definition of a quotient-causal-net. Clearly, for any contraction $e\in Con(\lambda)$, $s(e)\sim_\lambda t(e)$, and for any segment $e\in Seg(\lambda)$, $s(e)\not\sim_{\lambda} t(e)$. For any $e_1,e_2\in Seg(\lambda)$, if $\lambda(e_1)=\lambda(e_2)$, then we must have  $\lambda(s(e_1))=\lambda(s(e_2))$ and $\lambda(t(e_1))=\lambda(t(e_2))$, which means that $e_1\sim_\lambda e_2$ implies that $s(e_1)\sim_\lambda s(e_2)$ and $t(e_1)\sim_\lambda t(e_2)$. Therefore, $H$ is a quotient-causal-net of $G$.

$(\Rightarrow).$ Let $H$ be a quotient-causal-net of $G$, then
there exist an equivalence relation $\sim_\nu$ on $V(G)$, an induced subset $S$ of segments and an equivalence relation $\sim_\epsilon$ on $S$, such that $(1)$ for any segment $e\in S$, $s(e)\not\sim_\nu t(e)$; $(2)$ for any contraction $e\in C$, $s(e)\sim_\nu t(e)$; $(3)$ for any two segments $e_1, e_2\in S$, $e_1\sim_\epsilon e_2$ implies that $s(e_1)\sim_\nu s(e_2)$ and $t(e_1)\sim_\nu t(e_2)$; $(4)$ $V(H)=\displaystyle{\frac{V(G)}{\sim_\nu}}$; $(5)$ $E(H)=\displaystyle{\frac{S}{\sim_\epsilon}}$.

Now we define a morphism $\lambda:G\to H$ as follows. For any vertex $v$ of $G$, $\lambda(v)=[v]$, which is the equivalence class of $v$ with respect to $\sim_\nu$. For any segment $e\in S$, $\lambda(e)=[e]$, which is the equivalence class of $e$ with respect to $\sim_\epsilon$. For any contraction $e\in C$, $\lambda(e)$ is defined as the identity morphism of the equivalence class $[s(e)]$ (or equally, $[t(e)]$), which means that $e$ is a contraction of $\lambda$.

By condition $(3)$, we can check, from the above definitions, that for any segment $e\in S$, $s(\lambda(e))=s([e])=[s(e)]=\lambda(s(e))$ and $t(\lambda(e))=t([e])=[t(e)]=\lambda(t(e))$.
For any directed path $e_ne_{n-1}\cdots e_1$ in $G$, $\lambda(e_ne_{n-1}\cdots e_1)=[e_n]\circ [e_{n-1}]\circ\cdots\circ[e_1]$ is a directed path in $H$, where $[e_i]$ denotes the equivalence class of $e_i$, if $e_i\in S$, $i=1,2,\cdots,n$; otherwise, $[e_i]$ denotes the identity morphism of $[s(e_i)]$ or $[t(e_i)]$.

Defined in this way, $\lambda$ is indeed a quotient of causal-nets with $S=Seg(\lambda)$, $C=Con(\lambda)$, which means that $\lambda$ is a coarse-graining.
\end{proof}

\begin{rem}
The main motivation for introducing the notion of a coarse-graining is its natural connection with graphical calculi for monoidal categories \cite{[HLY15]}. It is a construction that also naturally appears in many physical theories, especially in theories of renormalizations \cite{[EV15]}. But in ordinary graph theory, this construction is seldom studied. We hope our work can  attract the attention of graph theorists.
The notion of a quotient-causal-net seems strange at the first glance, but it is indeed a natural quotient construction if viewed from a topological perspective.  In fact, if $H$ is a quotient-causal-net of $G$, then the geometric realization of $H$ is  a quotient topological space of the geometric realization of $G$. As explained in subsection \ref{cug},  a causal-net serves as a model of a finite conformal space-time, therefore it is natural to expect that a quotient-causal-net is a proper model of a quotient of a finite conformal space-time.
\end{rem}

\subsection{Vertex and edge coarse-graining}
In this subsection, we introduce two important sub-classes of coarse-grainings, both of which are closed under composition and naturally appear in some physical theories, such as loop quantum gravity \cite{[AKRZ14]}, tensor network renormalization \cite{[EV15]}, etc.

\begin{defn}\label{vcg}
A coarse-graining is called a \textbf{vertex-coarse-graining} if it has no multiple-edges.
\end{defn}
Clearly, a vertex-coarse-graining can be characterized as a morphism with the property that its simple-edges and segments are in bijective with each other. In other words, a vertex-coarse-graining is a coarse-graining with the corresponding segment-relation $\sim_\epsilon$ being the identity relation. That is, if $\lambda:G\to H$ is a vertex-coarse-graining, then $Seg(\lambda)$ is identical with $E(H)$.

\begin{ex}The following vertex-coarse-graining is represented by the morphism $\lambda$ with
$\lambda(v_1)=\lambda(v_2)=\lambda(v_3)=\lambda(v_4)=w$,
$\lambda(e_1)=\lambda(e_2)=\lambda(e_3)=Id_w$.
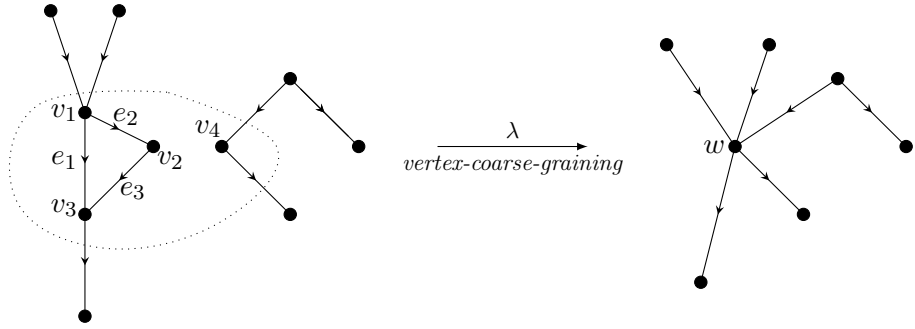
\begin{center}
\begin{tikzpicture}[scale=0.9]
\node (v2) at (-1.5,0.5) {};
\node (v4) at (-1.5,-1) {};
\node (v6) at (-0.5,0) {};
\node (v7) at (0.5,0) {};
\node (v9) at (1.5,1) {};
\node (v8) at (1.5,-1) {};
\node (v1) at (-2,2) {};
\node (v3) at (-1,2) {};
\node (v5) at (-1.5,-2.5) {};
\draw [fill](v1) circle [radius=0.09];
\draw [fill](v2) circle [radius=0.09];
\draw [fill](v3) circle [radius=0.09];
\draw [fill](v4) circle [radius=0.09];
\draw [fill](v5) circle [radius=0.09];
\draw [fill](v6) circle [radius=0.09];
\draw [fill](v7) circle [radius=0.09];
\draw [fill](v8) circle [radius=0.09];
\draw  (-2,2) --(-1.5,0.5)[postaction={decorate, decoration={markings,mark=at position .5 with {\arrow[black]{stealth}}}}];
\draw  (-1,2) --(-1.5,0.5)[postaction={decorate, decoration={markings,mark=at position .5 with {\arrow[black]{stealth}}}}];

\draw  (-1.5,0.5) --  (-1.5,-1)[postaction={decorate, decoration={markings,mark=at position .5 with {\arrow[black]{stealth}}}}];
\draw   (-1.5,-1) -- (-1.5,-2.5)[postaction={decorate, decoration={markings,mark=at position .5 with {\arrow[black]{stealth}}}}];
\draw  (-1.5,0.5) -- (-0.5,0)[postaction={decorate, decoration={markings,mark=at position .5 with {\arrow[black]{stealth}}}}];
\draw  (-0.5,0)--(-1.5,-1)[postaction={decorate, decoration={markings,mark=at position .5 with {\arrow[black]{stealth}}}}];
\draw   (0.5,0)-- (1.5,-1)[postaction={decorate, decoration={markings,mark=at position .5 with {\arrow[black]{stealth}}}}];
\draw   (1.5,1)--  (0.5,0)[postaction={decorate, decoration={markings,mark=at position .5 with {\arrow[black]{stealth}}}}];
\node (v10) at (2.5,0) {};
\draw   (1.5,1)--  (2.5,0)[postaction={decorate, decoration={markings,mark=at position .5 with {\arrow[black]{stealth}}}}];

\node (v10) at (2.5,0) {};
\draw  (v9) -- (v10)[postaction={decorate, decoration={markings,mark=at position .5 with {\arrow[black]{stealth}}}}];
\draw [dotted] plot[smooth,  tension=.7] coordinates {(-0.3,0.8) (-2.2,0.6) (-2.5,-0.8) (-1.2,-1.5) (0.5,-1.1) (1.3,0) (-0.3,0.8)};
\node (v11) at (7,1.5) {};
\node (v13) at (8.5,1.5) {};
\node (v12) at (8,0) {};
\node (v14) at (7.5,-2) {};
\node (v15) at (9.5,1) {};
\node (v17) at (10.5,0) {};
\node (v16) at (9,-1) {};
\draw  (7,1.5)  -- (8,0)[postaction={decorate, decoration={markings,mark=at position .5 with {\arrow[black]{stealth}}}}];
\draw (8.5,1.5) -- (8,0)[postaction={decorate, decoration={markings,mark=at position .5 with {\arrow[black]{stealth}}}}];
\draw (8,0)-- (7.5,-2)[postaction={decorate, decoration={markings,mark=at position .5 with {\arrow[black]{stealth}}}}];
\draw   (9.5,1) -- (8,0)[postaction={decorate, decoration={markings,mark=at position .5 with {\arrow[black]{stealth}}}}];
\draw  (8,0) -- (9,-1)[postaction={decorate, decoration={markings,mark=at position .5 with {\arrow[black]{stealth}}}}];
\draw   (9.5,1) -- (10.5,0)[postaction={decorate, decoration={markings,mark=at position .5 with {\arrow[black]{stealth}}}}];
\node (v18) at (3.5,0) {};
\node (v19) at (6,0) {};
\draw[-latex]  (v18) edge  node[sloped,scale=0.8, above] {$\lambda$}node[sloped,scale=0.8, below] {vertex-coarse-graining}(v19);
\draw [fill](v9) circle [radius=0.09];
\draw [fill](v10) circle [radius=0.09];
\draw [fill](v11) circle [radius=0.09];
\draw [fill](v12) circle [radius=0.09];
\draw [fill](v13) circle [radius=0.09];
\draw [fill](v14) circle [radius=0.09];
\draw [fill](v15) circle [radius=0.09];
\draw [fill](v16) circle [radius=0.09];
\draw [fill](v17) circle [radius=0.09];
\node at (-1.8,0.5) {$v_1$};
\node at (-1.8,-0.9) {$v_3$};
\node at (-0.25,-0.2) {$v_2$};
\node at (0.3,0.265) {$v_4$};
\node at (7.7,0) {$w$};
\node at (-0.9,0.44) {$e_2$};
\node at (-1.8,-0.2) {$e_1$};
\node at (-0.8,-0.65) {$e_3$};
\end{tikzpicture}
\end{center}
\end{ex}

Given a morphism $\lambda:G_1\to G_2$ and a vertex $v$ of $G_2$, the \textbf{fiber} $\mathcal{F}(v)$ of $\lambda$ at $v$ is defined as the induced sub-causal-net of $G_1$ with $\lambda^{-1}(v)$ as vertex set and $\lambda^{-1}(Id_v)$ as edge set. Since causal-nets are loop-free, fibers are necessarily induced sub-causal-nets.
It is easy to see that a vertex-coarse-graining is totally determined by its fibers, that is, to specify a vertex-coarse-graining is equivalent to specify all its fibers.

\begin{rem}
Just as coarse-grainings naturally appear in the graphical description of the monad that controls all  monoidal categories, vertex-coarse-grainings would naturally appear in the graphical description of the monad that controls all colored-props (monoidal categories that are free on objects). The inverse operation of a vertex-coarse-graining is also called a \textbf{substitution} \cite{[LV12]} in the theory of operads and PROPs.
\end{rem}

\begin{defn}\label{edge-coarse-graining}
A coarse-graining is called an \textbf{edge-coarse-graining} if it has no multiple-vertices and no contractions.
\end{defn}
An edge-coarse-graining can be characterized as a morphism with only simple-vertices and without null-edges, contractions and subdivisions. In other words, an edge-coarse-graining is a coarse-graining with the corresponding vertex-relation $\sim_\nu$ being the identity relation. If $\lambda:G\to H$ is an edge-coarse-graining, then $V(G)$ is identical with $V(H)$.

\begin{ex}The following edge-coarse-graining is represented by the morphism $\lambda$ with
$\lambda(e_1)=\lambda(e_2)=h_1$, $\lambda(e_3)=\lambda(e_4)=h_2$.

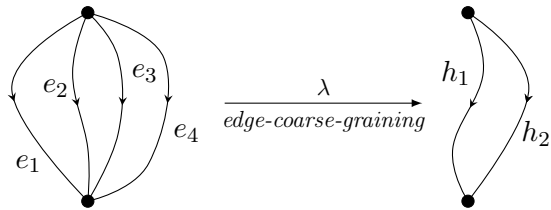
\begin{center}
\begin{tikzpicture}[scale=1]

\node (v1) at (-1,1) {};
\node (v2) at (-1,-1.5) {};
\node (v3) at (4,1) {};
\node (v4) at (4,-1.5) {};
\draw  plot[smooth, tension=.7] coordinates {(-1,1) (-2,0) (-1,-1.5)}[postaction={decorate, decoration={markings,mark=at position .5 with {\arrow[black]{stealth}}}}];
\draw  plot[smooth, tension=.7] coordinates {(-1,1)(-0.6,0.4) (-0.6,-0.6) (-1,-1.5) }[postaction={decorate, decoration={markings,mark=at position .5 with {\arrow[black]{stealth}}}}];
\draw  plot[smooth, tension=.7] coordinates {(4,1) (4.8,0)(4,-1.5)}[postaction={decorate, decoration={markings,mark=at position .5 with {\arrow[black]{stealth}}}}][postaction={decorate, decoration={markings,mark=at position .5 with {\arrow[black]{stealth}}}}];
\node (v5) at (0.8,-0.2) {};
\node (v6) at (3.4,-0.2) {};
\draw [-latex](0.8,-0.2) --  node[sloped,scale=0.8, above] {$\lambda$}node[sloped,scale=0.8, below] {edge-coarse-graining}(3.4,-0.2);
\node at (-1.78,-1) {$e_1$};
\node at (-1.42,0) {$e_2$};
\node at (-0.25,0.2) {$e_3$};
\node at (3.84,0.2) {$h_1$};
\node at (4.9,-0.6) {$h_2$};
\draw  plot[smooth, tension=.7] coordinates {(-1,1) (0,0.4) (-0.2,-1) (-1,-1.5)}[postaction={decorate, decoration={markings,mark=at position .5 with {\arrow[black]{stealth}}}}];
\node at (0.3,-0.6) {$e_4$};
\draw [fill](v1) circle [radius=0.08];
\draw [fill](v2) circle [radius=0.08];
\draw [fill](v3) circle [radius=0.08];
\draw [fill](v4) circle [radius=0.08];
\draw  plot[smooth, tension=.7] coordinates {(v1) (-1.2,0.2) (-1,-0.8) (v2)}[postaction={decorate, decoration={markings,mark=at position .5 with {\arrow[black]{stealth}}}}];
\draw  plot[smooth, tension=.7] coordinates {(v3) (4.2,0.2) (3.8,-0.8) (v4)}[postaction={decorate, decoration={markings,mark=at position .5 with {\arrow[black]{stealth}}}}];
\end{tikzpicture}
\end{center}
\end{ex}

We call a causal-net \textbf{simple} if it has no non-trivial multi-edges, or for any two vertices there is at most one edge connecting them. Note that edge-coarse-grainings define a partial order among causal-nets, and simple causal-nets are exactly those minimal ones under the partial order.

Edge-coarse-grainings can be characterized by congruences.

\begin{defn}
A \textbf{congruence} on a causal-net $G$ is an equivalence relation $\sim$ on the edge set $E(G)$ satisfying that  $e_1\sim e_2$ implies that $s(e_1)=s(e_2)$ and $t(e_1)=t(e_2)$.
\end{defn}
The following characterization is immediate.
\begin{prop}
For any causal-net, its edge-coarse-grainings are in bijective with its congruences.
\end{prop}

The following decomposition theorem is just a simple graph-theoretic fact and is useful for the presentation of the monad of the graphical calculus for monoidal categories \cite{[JS91],[HLY16],[HLY15]}.
\begin{thm}\label{v-e}
Any coarse-graining $\lambda$ can be uniquely represented as a composition of  a vertex-coarse-graining $\lambda_\nu$ and an edge-coarse-graining $\lambda_\varepsilon$.
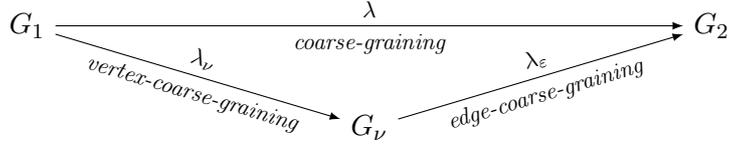
\begin{center}
\begin{tikzpicture}[scale=1.5]
\node (v1) at (-2,0.5) {$G_1$};
\node (v2) at (4,0.5) {$G_2$};
\draw[-latex]  (v1) edge node[sloped,scale=0.8, below] {coarse-graining} node[sloped,scale=0.8, above] {$\lambda$}(v2);
\node (v4) at (1,-0.4) {$G_\nu$};
\draw[-latex]  (v1) edge node[sloped,scale=0.8, above] {$\lambda_\nu$}node[sloped,scale=0.8, below] {vertex-coarse-graining} (v4);
\draw[-latex]  (v4) edge node[sloped,scale=0.8, above] {$\lambda_\varepsilon$}node[sloped,scale=0.8, below] {edge-coarse-graining}(v2);
\end{tikzpicture}
\end{center}
\end{thm}

\begin{proof}
The proof is easy. We only point out that
$$V(G_\nu)=\displaystyle{\frac{V(G)}{\sim_\lambda}},\ \ E(G_\nu)=Seg(\lambda),$$
with notations as in the proof of Theorem \ref{cc}.
\end{proof}

\begin{rem}
To abuse of terminologies, we are free to say that $H$ is a coarse-graining, vertex-coarse-graining, edge-coarse-graining, or any other types of quotients of $G$, in case that there is a corresponding quotient morphism from $G$ to $H$, respectively.
\end{rem}

\subsection{Merging and coclique}\label{MC}

In this and next two subsections, we introduce several special classes of vertex-coarse-grainings, which are characterized by their fibers.

The following notion is parallel to that of a vertex-identification in ordinary graph theory.
\begin{defn}
A vertex-coarse-graining is called a \textbf{merging} if it has no contractions.
\end{defn}

If $\lambda$ is a merging, then for any $v\in V(G_2)$, the fiber $\mathcal{F}(v)$ is discrete, i.e., a sub-causal-net with no edges. Clearly, a merging can be characterized as a vertex-coarse-graining with all fibers discrete.

We say causal-net $H$ is a merging of $G$ if there is a merging $\lambda:G\to H$. Clearly, mergings are closed under composition, which implies that the merging relation is a partial order.

For any causal-net $G$, its vertex set $V(G)$ equipped with the reachable order $\rightarrow$ is a poset, called the \textbf{vertex poset} of $G$, where the reachable order $v_1\to v_2$ means that there is a directed path starting from $v_1$ and ending with $v_2$.
Two vertices $v_1$ and $v_2$ are called \textbf{comparable} if either $v_1\to v_2$ or $v_2\to v_1$; otherwise they are called \textbf{incomparable}.
\begin{defn}
A \textbf{coclique} (or \textbf{incomparable set}) of a causal-net is a subset of vertices with no two vertices comparable.
\end{defn}

For any two vertices $v_1$, $v_2$ of an incomparable set, we have neither $v_1\to v_2$ nor $v_2\to v_1$, which implies that
all cocliques are  independent set (a subset of vertices with no two vertices sharing an edge).

The following result shows that cocliques are exactly fibers of mergings.
\begin{thm}\label{inc}
Let $S$ be a subset of vertices of causal-net $G$. $S$ is a fiber of a merging $\lambda:G\to H$ if and only if $S$ is a coclique of $G$.
\end{thm}
\begin{proof}
$(\Rightarrow)$
We prove this direction by contradiction. Suppose $v_1,v_2\in S$ and $v_1\to v_2$, then there is a directed path $e_ne_{n-1}\cdots e_1$ starting with $v_1$ and ending with $v_2$.  Since $\lambda$ is a merging, $e_1,\cdots, e_n$ all must be segments of $\lambda$, hence $\lambda(e_n)\lambda(e_{n-1})\cdots\lambda(e_1)$ must be a directed path in $H$ starting with $\lambda(v_1)$ and ending with $\lambda(v_2)$. Since $S$ is a fiber of $\lambda$, then $\lambda(v_1)=\lambda(v_2)$, which implies that $\lambda(e_n)\lambda(e_{n-1})\cdots\lambda(e_1)$ is a directed cycle. This contracts with the acyclicity of $H$.

$(\Leftarrow)$ Let $S$ be an incomparable set of $G$. We define a quotient-causal-net $H=G/S$ and a merging $\lambda:G\to H$ as follows. $(1)$ $V(H)=(V(G)-S)\sqcup\{w_0\}$, $E(H)=E(G)$; $(2)$ for any edge $e$ of $H$, if $s_G(e)\in S$ in $G$, then define $s_H(e)=w_0$  in $H$, otherwise define $s_H(e)=s_G(e)$; the definition of $t_H(e)$ is similar as $s_H(e)$. $(3)$ for any $v\in V(G)-S$, $\lambda(v)=v$, and for any $v\in S$, $\lambda(v)=w_0$; on edges, $\lambda$ is the identity mapping.

Now we prove by contradiction that $H$ is acyclic. Without loss of generality, suppose $\lambda(e_n)\cdots\lambda(e_1)$ is an directed cycle in $H$ with only $s_H(\lambda(e_1))=t_H(\lambda(e_n))=w_0$, then $e_n\cdots e_1$ must be a directed path in $G$ with $s_G(e_1),t_G(e_n)\in S$, which means that $s_G(e_1)\to t_G(e_n)$ in $G$. This contradicts the fact that $S$ is an incomparable set.

By the definition, it is easy to see that $\lambda$ is a merging.
\end{proof}

\subsection{Contraction and connectivity}
A causal-net is called \textbf{connected} if its underlying graph is connected. Unlike mergings with all fibers discrete, in the opposite direction, we introduce the following notion.

\begin{defn}
A vertex-coarse-graining is called a \textbf{contraction} if all its fibers are connected.
\end{defn}

Both mergings and contractions are closed under composition.
\begin{thm}\label{c-m}
Any vertex-coarse-graining $\lambda$ can be uniquely represented as a composition of  a merging $\lambda_m$ and a contraction $\lambda_{con}$.
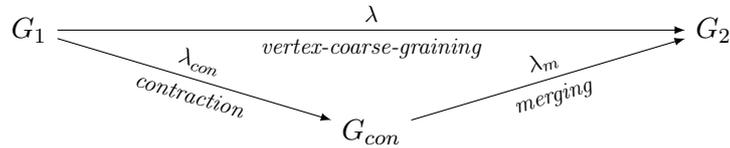
\begin{center}
\begin{tikzpicture}[scale=1.5]
\node (v1) at (-2,0.5) {$G_1$};
\node (v2) at (4,0.5) {$G_2$};
\draw[-latex]  (v1) edge node[sloped,scale=0.8, below] {vertex-coarse-graining} node[sloped,scale=0.8, above] {$\lambda$}(v2);
\node (v4) at (1,-0.4) {$G_{con}$};
\draw[-latex]  (v1) edge node[sloped,scale=0.8, above] {$\lambda_{con}$}node[sloped,scale=0.8, below] {contraction} (v4);
\draw[-latex]  (v4) edge node[sloped,scale=0.8, above] {$\lambda_m$}node[sloped,scale=0.8, below] {merging}(v2);
\end{tikzpicture}
\end{center}
\end{thm}
\begin{proof}
We prove this theorem by induction on the number $n$ of vertices of $G_1$. If $n=2$, a non-trivial vertex coarse-graining is either a merging or a contraction and this theorem is evidently true. Assume this theorem is true for $n\leq k$. Let $G_1$ be causal-net with $k+1$ vertices, $\lambda:G_1\to G_2$ be a vertex-coarse-graining and $v$ be a maximal vertex of $G_1$ under the reachable order $\to$. Without loss of generality, we assume that the fiber of $\lambda(v)$ has at least two vertices. Then the restricted morphism $\overline{\lambda}:G_1-\{v\}\to G_2$ of $\lambda$ is still a vertex-coarse-graining, where $G_1-\{v\}$ is the causal-net obtained form $G_1$ by removing the vertex $v$ and all edges ending with $v$. By the induction hypothesis, $\overline{\lambda}$ is a composition of contraction $\overline{\lambda}_1:G_1-\{v\}\to G_{con}$ and  merging $\overline{\lambda}_2:G_{con}\to G_2$. Here we have two cases. If $v$ is an isolated vertex of $G_1$, then the natural extension $\lambda_1:G_1\to G_{con}\sqcup \{v\}$ of $\overline{\lambda}_1$ is still a contraction,  the natural extension $\lambda_2:G_{con}\sqcup \{v\}\to G_2$ of $\overline{\lambda}_2$ with $\lambda_2(v)=\lambda(v)$ is still a merging, and $\lambda=\lambda_2\circ\lambda_1$, which complete the proof. Otherwise,  the natural extension $\lambda_1:G_1\to G_{con}$ of $\overline{\lambda}_1$ with $\lambda_1(v)=\lambda(v)$ is still a contraction and $\lambda=\overline{\lambda}_2\circ\lambda_1$, which also complete the proof.
\end{proof}

\begin{defn}
A contraction is called \textbf{simple} if it has a multi-edge as the only non-trivial fiber.
\end{defn}

The operation of contracting a multi-edge $\epsilon=\{e_1,\cdots, e_n\}$ of $G$ is called a \textbf{multi-edge-contraction},  which can be uniquely represented by the simple contraction $\mathcal{C}_\epsilon:G\to G/\{\epsilon\}$ with the unique non-trivial fiber being the multi-edge $\epsilon$, where $G/\{\epsilon\}$ denotes the resulting causal-net.

A causal-net is called a \textbf{point} if it has exactly one vertex and no edges.
Connectivity of causal-nets can be characterized in terms of multi-edge-contractions.
\begin{thm}\label{connect}
A causal-net is connected if and only if it can be transformed into a point through a series of multi-edge-contractions.
\end{thm}

\begin{proof}
The $(\Leftarrow)$ direction is obvious. We prove the $(\Rightarrow)$ direction by induction on the numbers $n$ of vertices of connected causal-nets.

For $n=2$, this theorem is clear. Assume this theorem is true for $n=k$.
Let $G$ be a connected causal-net with $k+1$ vertices, and $v$ be a maximal vertex of $G$ under the reachable order $\rightarrow$. Let $X$ be the set of all vertices of $G$ that share an edge with $v$, which is non-empty by the facts that $G$ is connected and has more than two vertices. Let $w$ be a maximal element of $X$ under  $\rightarrow$ (as shown below), then contracting the multi-edge connecting $w$ and $v$ will not produce an oriented cycle and the resulting causal-net $G'$ is connected with $k$ vertices.
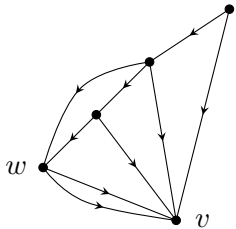
\begin{center}
  \begin{tikzpicture}[scale=0.7]
\node (v2) at (0,-1.5) {};
\node (v3) at (1,2.5) {};
\node (v4) at (-0.5,1.5) {};
\node (v5) at (-1.5,0.5) {};
\node (v1) at (-2.5,-0.5) {};
\draw [fill](v1) circle [radius=0.08];
\draw [fill](v2) circle [radius=0.08];
\draw [fill](v3) circle [radius=0.08];
\draw [fill](v4) circle [radius=0.08];
\draw [fill](v5) circle [radius=0.08];
\draw  (-2.5,-0.5) -- (0,-1.5)[postaction={decorate, decoration={markings,mark=at position .5 with {\arrow[black]{stealth}}}}];
\draw  (1,2.5)-- (0,-1.5)[postaction={decorate, decoration={markings,mark=at position .5 with {\arrow[black]{stealth}}}}];
\draw   (-0.5,1.5) -- (0,-1.5)[postaction={decorate, decoration={markings,mark=at position .5 with {\arrow[black]{stealth}}}}];
\draw   (-1.5,0.5)-- (0,-1.5) node (v6) {}[postaction={decorate, decoration={markings,mark=at position .5 with {\arrow[black]{stealth}}}}];
\draw  (1,2.5) --  (-0.5,1.5)[postaction={decorate, decoration={markings,mark=at position .5 with {\arrow[black]{stealth}}}}];
\draw   (-0.5,1.5) -- (-1.5,0.5)[postaction={decorate, decoration={markings,mark=at position .5 with {\arrow[black]{stealth}}}}];
\draw   (-1.5,0.5) --  (-2.5,-0.5)[postaction={decorate, decoration={markings,mark=at position .5 with {\arrow[black]{stealth}}}}];
\draw  plot[smooth, tension=.7] coordinates { (-0.5,1.5)  (-1.8,1) (-2.5,-0.5)}[postaction={decorate, decoration={markings,mark=at position .5 with {\arrow[black]{stealth}}}}];
\node at (-3,-0.5) {$w$};
\node at (0.5,-1.5) {$v$};
\draw  plot[smooth, tension=.7] coordinates {(-2.5,-0.5) (-1.7,-1.15) (0,-1.5)}[postaction={decorate, decoration={markings,mark=at position .5 with {\arrow[black]{stealth}}}}];
\end{tikzpicture}
\end{center}

By the induction hypothesis, $G'$ and hence $G$ can be transformed into a point through a series of multi-edge-contractions. This completes the proof.
\end{proof}

Using this theorem, we can prove the following result.
\begin{thm}\label{contraction}
A morphism is a contraction if and only if it is a composition of simple contractions.
\end{thm}
\begin{proof}

The $(\Leftarrow)$ direction is obvious and we only need to prove the $(\Rightarrow)$ direction. Let $\lambda:G_1\to G_2$ be a non-trivial contraction (i.e., not an identity), we will prove by induction on the numbers $n$ of vertices of $G_1$ and $m$ of vertices of $G_2$ that $\lambda$ is a composition of simple contractions.

For $n=2$ and $m=1$, $\lambda$ is clearly a simple contraction. Assume the result is true for $n\leq k, m\leq  k-1$. Now let $G_1$, $G_2$ be two causal-nets with the numbers of vertices no larger than $k+1$ and $ k$, respectively,  and $v$ be a maximal vertex of $G_2$ under the reachable order $\to$. We have two cases: $(1)$ the fiber of $v$ has exactly one vertex $w$ of $G_1$; $(2)$ the fiber of $v$ has no less than two vertices.

In the first case, $w$ must be a maximal vertex of $G_1$ under $\to$. We consider the causal-nets $G_1-\{w\}$ and $G_2-\{v\}$, which are obtained from $G_1$ and $G_2$ by removing $w$, $v$ and the edges connecting to them, respectively. The restricted morphism $\overline{\lambda}:G_1-\{w\}\to G_2-\{v\}$ is also a contraction. Since contractions would not change the number of connected components,  then $G_1-\{w\}$ and $G_2-\{v\}$ must have the same number of connected components. Suppose $G_1-\{w\}=\bigsqcup_{i\in I} H_i$ and $G_2-\{v\}=\bigsqcup_{i\in I} K_i$, then the restrictions $\overline{\lambda}_i:H_i\to K_i$ $(i\in I)$ of $\lambda$ are also contractions. For each $i\in I$, the numbers of vertices of $H_i$ and $K_i$ are not larger than $k$ and $k-1$, respectively. By the induction hypothesis, each $\lambda_i$ is a composition of simple contractions. Then $\overline{\lambda}=\bigsqcup_{i\in I}\overline{\lambda}_i:G_1-\{w\}\to G_2-\{v\}$ is a composition of simple contractions. Notice that the sets $I(w)$ and $I(v)$ are in bijective with each other through $\lambda$ and both $w$ and $v$ are maximal vertices, it is not difficult to see that $\lambda$, same as $\overline{\lambda}$, is a composition of simple contractions.

In the second case, $\lambda$ has a natural  decomposition, as a composition of $\lambda_1:G_1\to G_1/\mathcal{F}(v)$ and $\lambda_2:G_1/\mathcal{F}(v)\to G_2$, where $G_1/\mathcal{F}(v)$ is the quotient causal-net obtained by contracting the fiber $\mathcal{F}(v)$ of $v$.  Evidently, the number of vertices of $G_1/\mathcal{F}(\lambda(v))$ is not larger than $k$ and the number of vertices of $G_2$ is not larger than $k-1$, therefore, by the induction hypothesis, $\lambda_2$ is a composition of simple contractions. Since $\lambda$ is a contraction, then $\mathcal{F}(v)$ is a connected induced sub-causal-net. By Theorem \ref{connect}, it is not difficult to see that $\lambda_1$ is a composition of simple contractions. Therefore, $\lambda=\lambda_2\circ\lambda_1$ is a composition of simple contractions.

\end{proof}

\subsection{Causal-Tree and tree-contraction}

A causal-net is called a \textbf{causal-tree} (or a \textbf{poly-tree} \cite{[RP87]}), if its underlying undirected graph is a tree.  Note that a tree contains no cycles, we can arbitrarily define its orientation to produce a causal-tree. Especially, we view a point as a trivial causal-tree.

Each causal-net has a \textbf{simplification}, which is the unique simple causal-net obtaining by coarse-graining all its multi-edges.
The following generalization may be more useful.
\begin{defn}\label{cau-tree}
A causal-net is called a \textbf{causal-Tree} (or a \textbf{poly-Tree}), if its simplification is a causal-tree.
\end{defn}
Straightforward, a causal-Tree is a causal-tree possibly with multi-edges.

\begin{ex}
The following figure shows a causal-Tree and one of its multi-edge-contraction, which is represented by a morphism $\lambda$. The multi-edge $\{e_1,e_2\}$ is contracted, which is represented by the condition $\lambda(e_1)=\lambda(e_2)=Id_v$.
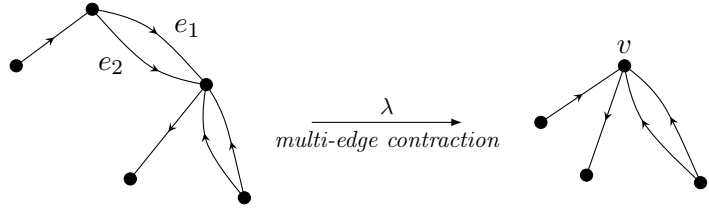
\begin{center}
\begin{tikzpicture}[scale=1]
\node (v2) at (-1.5,0) {};
\node (v1) at (-2.5,-0.75) {};
\node (v3) at (0,-1) {};
\node (v4) at (-1,-2.25) {};
\node (v5) at (0.5,-2.5) {};
\draw [fill](v1) circle [radius=0.08];
\draw [fill](v2) circle [radius=0.08];
\draw [fill](v3) circle [radius=0.08];
\draw [fill](v4) circle [radius=0.08];
\draw [fill](v5) circle [radius=0.08];
\draw  (-2.5,-0.75) -- (-1.5,0)[postaction={decorate, decoration={markings,mark=at position .5 with {\arrow[black]{stealth}}}}];
\draw  (0,-1) -- (-1,-2.25)[postaction={decorate, decoration={markings,mark=at position .5 with {\arrow[black]{stealth}}}}];
\draw  plot[smooth, tension=.7] coordinates {(-1.5,0) (-0.75,-0.75) (0,-1)}[postaction={decorate, decoration={markings,mark=at position .65 with {\arrow[black]{stealth}}}}];
\draw  plot[smooth, tension=.7] coordinates {(-1.5,0) (-0.75,-0.25) (0,-1)}[postaction={decorate, decoration={markings,mark=at position .5 with {\arrow[black]{stealth}}}}];
\draw  plot[smooth, tension=.7] coordinates {(0.5,-2.5) (0,-1.75) (0,-1)}[postaction={decorate, decoration={markings,mark=at position .61 with {\arrow[black]{stealth}}}}];
\draw  plot[smooth, tension=.7] coordinates {(0.5,-2.5) (0.25,-1.5) (0,-1)}[postaction={decorate, decoration={markings,mark=at position .48 with {\arrow[black]{stealth}}}}];
\node (v7) at (5.5,-0.75) {};
\node (v8) at (4.4,-1.5) {};
\node (v9) at (5,-2.2) {};
\node (v6) at (6.5,-2.3) {};
\draw [fill](v7) circle [radius=0.08];
\draw [fill](v8) circle [radius=0.08];
\draw [fill](v9) circle [radius=0.08];
\draw [fill](v6) circle [radius=0.08];
\draw  (4.4,-1.5) -- (5.5,-0.75)[postaction={decorate, decoration={markings,mark=at position .5 with {\arrow[black]{stealth}}}}];
\draw  (5.5,-0.75) -- (5,-2.2)[postaction={decorate, decoration={markings,mark=at position .5 with {\arrow[black]{stealth}}}}];
\draw  plot[smooth, tension=.7] coordinates {(6.5,-2.3) (5.75,-1.5) (5.5,-0.75)}[postaction={decorate, decoration={markings,mark=at position .6 with {\arrow[black]{stealth}}}}];
\draw  plot[smooth, tension=.7] coordinates {(6.5,-2.3) (6,-1.25) (5.5,-0.75)}[postaction={decorate, decoration={markings,mark=at position .51 with {\arrow[black]{stealth}}}}];
\node (v10) at (1.25,-1.5) {};
\node (v11) at (3.5,-1.5) {};
\draw [-latex] (v10) edge node[sloped,scale=0.8, above] {$\lambda$}node[sloped,scale=0.8, below] {multi-edge contraction} (v11);
\node at (-0.25,-0.25) {$e_1$};
\node at (-1.25,-0.75) {$e_2$};
\node at (5.5,-0.5) {$v$};
\end{tikzpicture}
\end{center}
\end{ex}

To characterize causal-Trees, we introduce the following notion.
\begin{defn}
Let $\lambda:G\to H$ be a contraction. A multi-edge $\varepsilon=\{e_1,\cdots, e_n\}$ of $H$ is called \textbf{G-primitive} (or simply  \textbf{primitive}) if its pre-image $\lambda^{-1}(\varepsilon)=\{\lambda^{-1}(e_1),\cdots,\lambda^{-1}(e_n)\}$ is a multi-edge of $G$.
\end{defn}

Since isomorphisms are trivially contractions, according to the definition, we can say all multi-edges of $G$ are $G$-primitive.

\begin{ex}\label{r3}
The operation of contracting a multi-edge does not produce new edges but may produce new multi-edges, which are non-primitive. In the following example, all edges of $H$ are (naturally identified with) edges of $G$. $\{e_1,e_3\}$ is a multi-edge of $H$, but not a multi-edge of $G$, so it is a non-primitive multi-edge produced by the contraction.

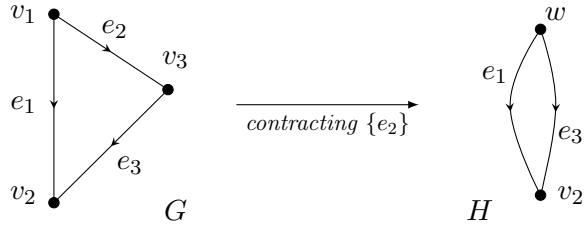
\begin{center}
\begin{tikzpicture}[scale=1]

\node (v1) at (0,0) {};
\node (v2) at (0,-2.5) {};
\node (v3) at (1.5,-1) {};
\draw [fill](v1) circle [radius=0.068];
\draw [fill](v2) circle [radius=0.068];
\draw [fill](v3) circle [radius=0.068];
\draw (0,0) -- (0,-2.5)[postaction={decorate, decoration={markings,mark=at position .5 with {\arrow[black]{stealth}}}}];
\draw(0,0)  -- (1.5,-1) [postaction={decorate, decoration={markings,mark=at position .5 with {\arrow[black]{stealth}}}}];
\draw  (1.5,-1)  -- (0,-2.5)[postaction={decorate, decoration={markings,mark=at position .5 with {\arrow[black]{stealth}}}}];
\node (v4) at (6.4,-0.2) {};
\node (v5) at (6.4,-2.4) {};
\node at (1.6,-2.6) {$G$};
\node at (5.6,-2.6) {$H$};
\draw  plot[smooth, tension=.7] coordinates {(6.4,-0.2)  (6,-1.2) (6.4,-2.4) }[postaction={decorate, decoration={markings,mark=at position .5 with {\arrow[black]{stealth}}}}];
\draw  plot[smooth, tension=.7] coordinates {(6.4,-0.2) (6.6,-1.2) (6.4,-2.4)}[postaction={decorate, decoration={markings,mark=at position .5 with {\arrow[black]{stealth}}}}];
\node (v6) at (2.4,-1.2) {};
\node (v7) at (4.8,-1.2) {};
\draw [fill](v4) circle [radius=0.068];
\draw [fill](v5) circle [radius=0.068];
\draw  [-latex](2.4,-1.2)  --node[sloped,scale=0.8, below] { contracting $\{e_2\}$}(4.8,-1.2) ;
\node at (-0.4,0) {$v_1$};
\node at (-0.4,-2.4) {$v_2$};
\node at (1.6,-0.6) {$v_3$};
\node at (6.6,0) {$w$};
\node at (6.8,-2.4) {$v_2$};
\node at (-0.4,-1.2) {$e_1$};
\node at (0.8,-0.2) {$e_2$};
\node at (1,-2) {$e_3$};
\node at (5.8,-0.8) {$e_1$};
\node at (6.8,-1.6) {$e_3$};
\end{tikzpicture}
\end{center}

\end{ex}

The following theorem gives a characterization of causal-Trees.
\begin{thm}\label{c-t}
A causal-net is a causal-Tree if and only if it can be transformed into a point through a series of operations of contracting a primitive multi-edge.
\end{thm}
\begin{proof}
Let $G$ be a causal-net and $\overline{G}$ be the simplification of $G$. By definition, $G$ is a causal-Tree if and only if $\overline{G}$ is a causal-tree. So to prove this theorem, we only need to show that $\overline{G}$ is a causal-tree if and only if $G$ can be transformed into a point through a series of operations of contracting a primitive multi-edge.

 It is well known that a graph is a tree if and only if it can be transformed into a point through a series of operations of contracting an edge and through all these processes, no cycle is produced. So $\overline{G}$ is a causal-tree if and only if it can be transformed into a point through a series of operations of contracting an edge and through all these processes, acyclicity is preserved.

By the definition of a simplification, the set $\{e_i,i\in I\}$ of edges of $\overline{G}$ is in bijective with the set $\{\varepsilon_i,i\in I\}$ of (primitive) multi-edges of $G$. It is easy to see that the simplification of $G/\varepsilon_i$ is equal to $\overline{G}/e_i$, for all $i\in I$. Therefore, $G$ can be transformed into a point through a series of operations of contracting a primitive multi-edge if and only if $\overline{G}$ can be transformed into a point through a series of operations of contracting an edge, which, combining with the former results,  gives a proof of this theorem.
\end{proof}

There is also a simple topological characterization, which says that a causal-net is a causal-Tree if and only if the geometric realization of its simplification is a contractible space.

\begin{defn}
A contraction is called a \textbf{tree-contraction}  if all its fibers are causal-Trees.
\end{defn}
A simple contraction is called \textbf{primitive} if it contracts a primitive multi-edge. Similar to Theorem \ref{contraction}, we have the following characterization.

\begin{thm}\label{prim}
A morphism is a tree-contraction if and only if it is a composition of primitive simple contractions.
\end{thm}
\begin{proof}
Let $\lambda:G_1\to G_2$ be a morphism. $\lambda$ is a tree-contraction if and only if each fiber $\mathcal{F}(w)$, $w\in V(G_2)$ is a causal-Tree. By Theorem \ref{c-t}, each $\mathcal{F}(w)$ can be transformed into a point by a series of operations of contracting a $\mathcal{F}(w)$-primitive multi-edge. Since all $\mathcal{F}(w)$s  are induced sub-causal-nets of $G_1$, then all $\mathcal{F}(w)$-primitive multi-edges are $G_1$-primitive multi-edges. So $G_1$ can be transformed into $G_2$ by a series of operations of contracting a $G_1$-primitive multi-edge, which means that $\lambda$ is a composition of primitive simple contractions.
\end{proof}

Due to this characterization, we also name a tree-contraction as a \textbf{primitive contractions}, which can be equivalently defined as a contraction with all contracted multi-edges being primitive.
We say that causal-net $H$ is a tree-contraction of $G$, if there is a tree-contraction $\lambda:G\to H$. A multi-edge-contraction is called \textbf{primitive} if the contracted multi-edge is primitive.  Theorem \ref{prim} can be restated as follows.
\begin{thm}\label{tree}
Let $H,G$ be two causal-nets. $H$ is a tree-contraction of $G$ if and only if $H$ can be obtained from $G$ through a series of primitive multi-edge-contractions.
\end{thm}

Two causal-nets are called \textbf{homotopy equivalent} if their geometric realizations are homotopy equivalent; and are called \textbf{almost homotopy equivalent} if their simplifications are homotopy equivalent. Clearly, a causal-net is a causal-tree if and only if it is homotopy equivalent to a point, and a causal-net is a causal-Tree if and only if it is almost homotopy equivalent to a point.

Primitivity can be characterized topologically as follows.
\begin{prop}
A contraction $\lambda:G\to H$ is primitive if and only if $G$ and $H$ are almost homotopy equivalent.
\end{prop}

A causal-net is called a \textbf{directed multi-path} or \textbf{directed-Path} if its simplification is a directed path. In short, a directed multi-path is directed path with multi-edges.
The following special case of tree-contraction will  be used to understand subdivisions in Section $3.2$.
\begin{defn}
A vertex-coarse-graining is called a \textbf{path-contraction} if all its fibers are directed-Paths.
\end{defn}

The relations of various types of vertex-coarse-grainings are listed as follows.\\
\begin{center}
\begin{tikzpicture}[scale=1.5]
\draw  (-2.5,3) rectangle (0.5,2) node (v1) {};
\draw  (0.5,3) rectangle (3.5,2);
\node [scale=0.8] at (-1,2.5) {\textbf{Type}};
\node [scale=0.8] at (2,2.5) {\textbf{Fiber}};

\node[scale=0.8] (v2) at (-1,0.5) {contraction};
\node[scale=0.8] (v3) at (2,0.5) {connected causal-net};
\node [scale=0.8] (v6) at (-1,-0.5) {tree-contraction};
\node[scale=0.8] (v7) at (2,-0.5) {causal-Tree};
\node [scale=0.8] (v8) at (-1,-1.5) {path-contraction};
\node [scale=0.8] (v9) at (2,-1.5) {directed-Path};
\draw  (-2.5,2) node (v12) {} rectangle (0.5,-2);
\draw  (v1) rectangle (3.5,-2);
\node [rotate=90] (v4) at (-1,0) {$\subseteq$};
\node [rotate=90] at (-1,-1) {$\subseteq$};
\node[rotate=90] (v5) at (2,0) {$\subseteq$};
\node[rotate=90] at (2,-1) {$\subseteq$};
\draw [dashed, latex-latex] (v2) -- (v3);
\draw [dashed,latex-latex] (v6) -- (v7);
\draw  [dashed,latex-latex](v8) -- (v9);
\node [scale=0.8] (v10) at (-1,1.5) {merging};
\node [scale=0.8] (v11) at (2,1.5) {discrete causal-net};
\draw  [dashed, latex-latex](v10) edge (v11);
\draw  (v12) rectangle (3.5,1);
\end{tikzpicture}
\end{center}

\subsection{Fusion and causal-coloring}
In this subsection, we discuss fusions, which are another special kind of vertex-coarse-grainings, and use them to establish a coloring theory for causal-nets.

\begin{defn}
A coarse-graining is called a \textbf{fusion}  if it has no contractions.
\end{defn}
Clearly,  a morphism $\lambda:G_1\to G_2$ without subdivisions and  contractions, or equivalently with all edges of $G_1$ being segments is exactly a \textbf{homomorphism}  \cite{[HN04]} of causal-nets, and a  fusion is just a \textbf{surjective homomorphism} of causal-nets. We say $H$ is a fusion of $G$, if there is a fusion $\lambda:G\to H$.  Clearly, fusions are closed under composition and the fusion relation defines a partial order among causal-nets, which we call \textbf{fusion order}.

The following result is a direct consequence of Theorem \ref{v-e} and Theorem \ref{c-m}.

\begin{cor}\label{fusion-decomposition}
Any fusion $\lambda$ can be uniquely represented as a composition of  a merging $\lambda_m$ and an edge-coarse-graining $\lambda_{\varepsilon}$.
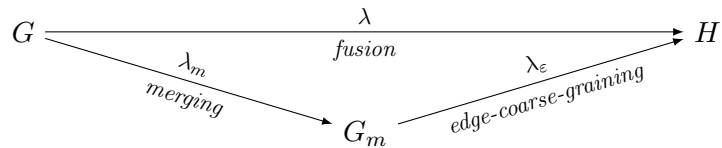
\begin{center}
\begin{tikzpicture}[scale=1.5]
\node (v1) at (-2,0.5) {$G$};
\node (v2) at (4,0.5) {$H$};
\draw[-latex]  (v1) edge node[sloped,scale=0.8, below] {fusion} node[sloped,scale=0.8, above] {$\lambda$}(v2);
\node (v4) at (1,-0.4) {$G_m$};
\draw[-latex]  (v1) edge node[sloped,scale=0.8, above] {$\lambda_m$}node[sloped,scale=0.8, below] {merging} (v4);
\draw[-latex]  (v4) edge node[sloped,scale=0.8, above] {$\lambda_{\varepsilon}$}node[sloped,scale=0.8, below] {edge-coarse-graining}(v2);
\end{tikzpicture}
\end{center}
\end{cor}
Due to this result, we also call a merging and an edge-coarse-graining a \textbf{vertex-fusion} and an \textbf{edge-fusion}, respectively.

If $\lambda:G\to H$ is a merging or a fusion, then all its fibers form a partition of the vertex set of $G$ with each block (i.e., fiber) being a coclique. This partition can be understood as a kind of vertex-coloring, such that two vertices are labelled by the same color if and only if they are in the same coclique.

\begin{defn}\label{coloring}
A fusion $\lambda:G\to H$ is called a \textbf{causal-coloring} (or simply \textbf{coloring}) of $G$, if $H$ is a simple causal-net.
\end{defn}

In other words, a coloring is just a fusion that is minimal under edge-coarse-grainings.
Note that a causal-coloring $\lambda:G\to H$ is not just a partition of cocliques of $G$, because the order relation of cocliques (the vertex poset of $H$) also provides important informations.

Recall that any causal-net has a unique simplification, we can easily see the following result from Corollary \ref{fusion-decomposition}.
\begin{cor}\label{color}
For any causal-net, its mergings are in bijective with its colorings.
\end{cor}
Clearly, the merging and edge-coarse-graining relations are sub-partial-orders of the fusion order. Conversely, by Corollary \ref{fusion-decomposition}, we may view the fusion order as a kind of product of the merging and edge-coarse-graining orders.

Now we turn to those minimal causal-nets under the fusion order.

\begin{defn}\label{harmonic}
A causal-net is called \textbf{harmonic} if it is minimal under the fusion order, that is, all its fusions are isomorphisms.
\end{defn}
Obviously, harmonic causal-nets are necessarily simple and connected.
A directed path in a causal-net  is called a \textbf{hamiltonian path} if it visits each vertex exactly once.  Harmonic causal-nets can be characterized by hamiltonian paths.
\begin{thm}
A simple causal-net is harmonic if and only if it has a hamiltonian path.
\end{thm}
\begin{proof}
$(\Leftarrow)$ This direction is obvious, that is because any non-trivial merging produces at least one directed cycle. We prove the $(\Rightarrow)$ direction by induction on the number $n$ of vertices.
If $n=2$, the conclusion is trivial. Assume the $(\Rightarrow)$ direction is true for $n=k$, we want to show it is also true for $n=k+1$. Let $G$ be a harmonic causal-net with $k+1$ vertices, and $w$ be a maximal vertex of $G$ under the reachable order $\rightarrow$. Let $G'=G-\{w\}$ be the induced sub-causal-net of $G$ obtained by removing $w$ and all edges connecting $w$. Clearly, $G'$ is simple and has $k$ vertices. Since $G$ is harmonic, that is, it has no non-trivial mergings and edge-coarse-grainings, then by the construction of $G'$, it is not difficult to see that $G'$ also has no non-trivial mergings and edge-coarse-grainings and therefore is harmonic. By the induction hypothesis, $G'$ has a hamiltonian path $v_1\to v_2\to\cdots\to v_k$, we claim that $v_1\to\cdots\to v_k\to w$ must be a hamiltonian path of $G$. If not, $v_k$ must be a maximal vertex of $G$ under $\rightarrow$. Obviously, $v_k$ and $w$ are incomparable and we can merge $v_k$ and $w$ to obtain a causal-net $H$, which is  smaller than $G$ under the fusion order. This contradicts  the fact that $G$ is harmonic.
\end{proof}

What is more special is the following notion.
\begin{defn}\label{complete}
A causal-net is called \textbf{complete} if any two of its vertices are connected by exactly one edge.
\end{defn}

In other words, a complete causal-net is just a causal-net with its underlying undirected graph being a complete graph. Complete causal-nets are necessarily harmonic.
As the following two examples shown, for each natural number $n\in \mathbb{N}$, there is exactly one isomorphism class of complete causal-nets, denoted by $K_n$.
\begin{ex}\label{K3}
There are two isomorphism classes of harmonic causal-nets with three vertices. The second is the complete causal-net $K_3$.
\begin{center}
\begin{tikzpicture}[scale=0.7]
\node (v4) at (-7,1.5) {};
\node (v5) at (-7,0) {};
\node (v6) at (-7,-1.5) {};
\draw [fill](v5) circle [radius=0.09];
\draw [fill](v6) circle [radius=0.09];
\draw [fill](v4) circle [radius=0.09];
\draw (-7,1.5) --(-7,0)[postaction={decorate, decoration={markings,mark=at position .5 with {\arrow[black]{stealth}}}}];
\draw  (-7,0) -- (-7,-1.5)[postaction={decorate, decoration={markings,mark=at position .5 with {\arrow[black]{stealth}}}}];

\node (v1) at (-1,1.5) {};
\node (v2) at (-1,0) {};
\node (v3) at (-1,-1.5) {};
\draw [fill](v1) circle [radius=0.09];
\draw [fill](v2) circle [radius=0.09];
\draw [fill](v3) circle [radius=0.09];
\draw (-1,1.5) --(-1,0)[postaction={decorate, decoration={markings,mark=at position .5 with {\arrow[black]{stealth}}}}];
\draw  (-1,0) -- (-1,-1.5)[postaction={decorate, decoration={markings,mark=at position .5 with {\arrow[black]{stealth}}}}];
\draw  plot[smooth, tension=.7] coordinates {(-1,1.5) (-0.3,0) (-1,-1.5)}[postaction={decorate, decoration={markings,mark=at position .55 with {\arrow[black]{stealth}}}}];
\end{tikzpicture}
\end{center}
\end{ex}

\begin{ex}\label{K4}
There are eight isomorphism classes of harmonic causal-nets with four vertices. The last one is the complete causal-net $K_4$.
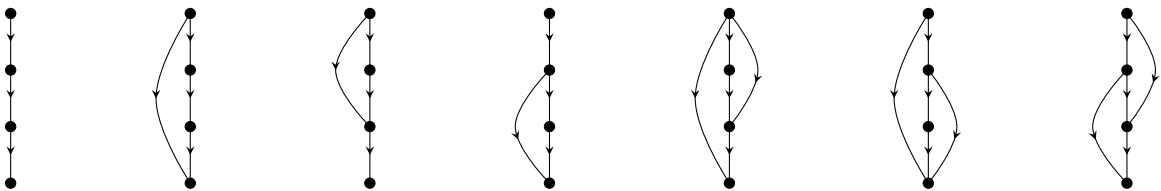
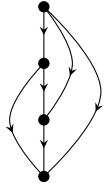
\begin{center}
\begin{tikzpicture}[scale=0.75]
\node (v1) at (0.5,1) {};
\node (v2) at (0.5,0) {};
\node (v3) at (0.5,-1) {};
\node (v4) at (0.5,-2) {};
\draw [fill](v1) circle [radius=0.09];
\draw [fill](v2) circle [radius=0.09];
\draw [fill](v3) circle [radius=0.09];
\draw [fill](v4) circle [radius=0.09];
\draw  (0.5,1)-- (0.5,0)[postaction={decorate, decoration={markings,mark=at position .5 with {\arrow[black]{stealth}}}}];
\draw  (0.5,0) -- (0.5,-1)[postaction={decorate, decoration={markings,mark=at position .5 with {\arrow[black]{stealth}}}}];
\draw  (0.5,-1) -- (0.5,-2)[postaction={decorate, decoration={markings,mark=at position .5 with {\arrow[black]{stealth}}}}];
\end{tikzpicture}
\ \ \ \ \ \ \ \ \ \ \ \
\begin{tikzpicture}[scale=0.75]
\node (v1) at (0.5,1) {};
\node (v2) at (0.5,0) {};
\node (v3) at (0.5,-1) {};
\node (v4) at (0.5,-2) {};
\draw [fill](v1) circle [radius=0.09];
\draw [fill](v2) circle [radius=0.09];
\draw [fill](v3) circle [radius=0.09];
\draw [fill](v4) circle [radius=0.09];
\draw  (0.5,1)-- (0.5,0)[postaction={decorate, decoration={markings,mark=at position .5 with {\arrow[black]{stealth}}}}];
\draw  (0.5,0) -- (0.5,-1)[postaction={decorate, decoration={markings,mark=at position .5 with {\arrow[black]{stealth}}}}];
\draw  (0.5,-1) -- (0.5,-2)[postaction={decorate, decoration={markings,mark=at position .5 with {\arrow[black]{stealth}}}}];
\draw  plot[smooth, tension=.7] coordinates {(0.5,1) (-0.1,-0.5) (0.5,-2)}[postaction={decorate, decoration={markings,mark=at position .5 with {\arrow[black]{stealth}}}}];
\end{tikzpicture}
\ \ \ \ \ \ \ \ \ \ \ \
\begin{tikzpicture}[scale=0.75]
\node (v1) at (0.5,1) {};
\node (v2) at (0.5,0) {};
\node (v3) at (0.5,-1) {};
\node (v4) at (0.5,-2) {};
\draw [fill](v1) circle [radius=0.09];
\draw [fill](v2) circle [radius=0.09];
\draw [fill](v3) circle [radius=0.09];
\draw [fill](v4) circle [radius=0.09];
\draw  (0.5,1)-- (0.5,0)[postaction={decorate, decoration={markings,mark=at position .5 with {\arrow[black]{stealth}}}}];
\draw  (0.5,0) -- (0.5,-1)[postaction={decorate, decoration={markings,mark=at position .5 with {\arrow[black]{stealth}}}}];
\draw  (0.5,-1) -- (0.5,-2)[postaction={decorate, decoration={markings,mark=at position .5 with {\arrow[black]{stealth}}}}];
\draw  plot[smooth, tension=.7] coordinates {(0.5,1) (-0.1,0) (0.5,-1)}[postaction={decorate, decoration={markings,mark=at position .5 with {\arrow[black]{stealth}}}}];
\end{tikzpicture}
\ \ \ \ \ \ \ \ \ \ \ \
\begin{tikzpicture}[scale=0.75]
\node (v1) at (0.5,1) {};
\node (v2) at (0.5,0) {};
\node (v3) at (0.5,-1) {};
\node (v4) at (0.5,-2) {};
\draw [fill](v1) circle [radius=0.09];
\draw [fill](v2) circle [radius=0.09];
\draw [fill](v3) circle [radius=0.09];
\draw [fill](v4) circle [radius=0.09];
\draw  (0.5,1)-- (0.5,0)[postaction={decorate, decoration={markings,mark=at position .5 with {\arrow[black]{stealth}}}}];
\draw  (0.5,0) -- (0.5,-1)[postaction={decorate, decoration={markings,mark=at position .5 with {\arrow[black]{stealth}}}}];
\draw  (0.5,-1) -- (0.5,-2)[postaction={decorate, decoration={markings,mark=at position .5 with {\arrow[black]{stealth}}}}];
\draw  plot[smooth, tension=.7] coordinates {(0.5,0) (-0.1,-1) (0.5,-2)}[postaction={decorate, decoration={markings,mark=at position .6 with {\arrow[black]{stealth}}}}];
\end{tikzpicture}
\ \ \ \ \ \ \ \ \ \ \ \
\begin{tikzpicture}[scale=0.75]
\node (v1) at (0.5,1) {};
\node (v2) at (0.5,0) {};
\node (v3) at (0.5,-1) {};
\node (v4) at (0.5,-2) {};
\draw [fill](v1) circle [radius=0.09];
\draw [fill](v2) circle [radius=0.09];
\draw [fill](v3) circle [radius=0.09];
\draw [fill](v4) circle [radius=0.09];
\draw  (0.5,1)-- (0.5,0)[postaction={decorate, decoration={markings,mark=at position .5 with {\arrow[black]{stealth}}}}];
\draw  (0.5,0) -- (0.5,-1)[postaction={decorate, decoration={markings,mark=at position .5 with {\arrow[black]{stealth}}}}];
\draw  (0.5,-1) -- (0.5,-2)[postaction={decorate, decoration={markings,mark=at position .5 with {\arrow[black]{stealth}}}}];
\draw  plot[smooth, tension=.7] coordinates {(0.5,1) (-0.1,-0.5) (0.5,-2)}[postaction={decorate, decoration={markings,mark=at position .5 with {\arrow[black]{stealth}}}}];
\draw  plot[smooth, tension=.7] coordinates {(0.5,1) (1,0) (0.5,-1)}[postaction={decorate, decoration={markings,mark=at position .6 with {\arrow[black]{stealth}}}}];
\end{tikzpicture}
\ \ \ \ \ \ \ \ \ \ \ \
\begin{tikzpicture}[scale=0.75]
\node (v1) at (0.5,1) {};
\node (v2) at (0.5,0) {};
\node (v3) at (0.5,-1) {};
\node (v4) at (0.5,-2) {};
\draw [fill](v1) circle [radius=0.09];
\draw [fill](v2) circle [radius=0.09];
\draw [fill](v3) circle [radius=0.09];
\draw [fill](v4) circle [radius=0.09];
\draw  (0.5,1)-- (0.5,0)[postaction={decorate, decoration={markings,mark=at position .5 with {\arrow[black]{stealth}}}}];
\draw  (0.5,0) -- (0.5,-1)[postaction={decorate, decoration={markings,mark=at position .5 with {\arrow[black]{stealth}}}}];
\draw  (0.5,-1) -- (0.5,-2)[postaction={decorate, decoration={markings,mark=at position .5 with {\arrow[black]{stealth}}}}];
\draw  plot[smooth, tension=.7] coordinates {(0.5,1) (-0.1,-0.5) (0.5,-2)}[postaction={decorate, decoration={markings,mark=at position .5 with {\arrow[black]{stealth}}}}];
\draw  plot[smooth, tension=.7] coordinates {(0.5,0) (1,-1) (0.5,-2)}[postaction={decorate, decoration={markings,mark=at position .6 with {\arrow[black]{stealth}}}}];
\end{tikzpicture}
\ \ \ \ \ \ \ \ \ \ \ \
\begin{tikzpicture}[scale=0.75]
\node (v1) at (0.5,1) {};
\node (v2) at (0.5,0) {};
\node (v3) at (0.5,-1) {};
\node (v4) at (0.5,-2) {};
\draw [fill](v1) circle [radius=0.09];
\draw [fill](v2) circle [radius=0.09];
\draw [fill](v3) circle [radius=0.09];
\draw [fill](v4) circle [radius=0.09];
\draw  (0.5,1)-- (0.5,0)[postaction={decorate, decoration={markings,mark=at position .5 with {\arrow[black]{stealth}}}}];
\draw  (0.5,0) -- (0.5,-1)[postaction={decorate, decoration={markings,mark=at position .5 with {\arrow[black]{stealth}}}}];
\draw  (0.5,-1) -- (0.5,-2)[postaction={decorate, decoration={markings,mark=at position .5 with {\arrow[black]{stealth}}}}];
\draw  plot[smooth, tension=.7] coordinates {(0.5,0) (-0.1,-1) (0.5,-2)}[postaction={decorate, decoration={markings,mark=at position .6 with {\arrow[black]{stealth}}}}];
\draw  plot[smooth, tension=.7] coordinates {(0.5,1) (1,0) (0.5,-1)}[postaction={decorate, decoration={markings,mark=at position .6 with {\arrow[black]{stealth}}}}];
\end{tikzpicture}
\ \ \ \ \ \ \ \ \ \ \ \
\begin{tikzpicture}[scale=0.75]
\node (v1) at (0.5,1) {};
\node (v2) at (0.5,0) {};
\node (v3) at (0.5,-1) {};
\node (v4) at (0.5,-2) {};
\draw [fill](v1) circle [radius=0.09];
\draw [fill](v2) circle [radius=0.09];
\draw [fill](v3) circle [radius=0.09];
\draw [fill](v4) circle [radius=0.09];
\draw  (0.5,1)-- (0.5,0)[postaction={decorate, decoration={markings,mark=at position .5 with {\arrow[black]{stealth}}}}];
\draw  (0.5,0) -- (0.5,-1)[postaction={decorate, decoration={markings,mark=at position .5 with {\arrow[black]{stealth}}}}];
\draw  (0.5,-1) -- (0.5,-2)[postaction={decorate, decoration={markings,mark=at position .5 with {\arrow[black]{stealth}}}}];
\draw  plot[smooth, tension=.7] coordinates {(0.5,0) (-0.1,-1) (0.5,-2)}[postaction={decorate, decoration={markings,mark=at position .6 with {\arrow[black]{stealth}}}}];
\draw  plot[smooth, tension=.7] coordinates {(0.5,1) (1,0) (0.5,-1)}[postaction={decorate, decoration={markings,mark=at position .6 with {\arrow[black]{stealth}}}}];
\draw  plot[smooth, tension=.7] coordinates {(0.5,1) (1.5,-0.5) (0.5,-2)}[postaction={decorate, decoration={markings,mark=at position .6 with {\arrow[black]{stealth}}}}];
\end{tikzpicture}
\end{center}
\end{ex}

It is natural to introduce the following  notions.

\begin{defn}\label{minimal}
A coloring $\lambda:G\to H$ is called \textbf{minimal}, if $H$ is minimal, that is, $H$ is harmonic; and is called \textbf{complete}, if $H$ is complete.
\end{defn}

At the end of this subsection, we depict the relations among various types of (non-isomorphic) coarse-grainings as follows.
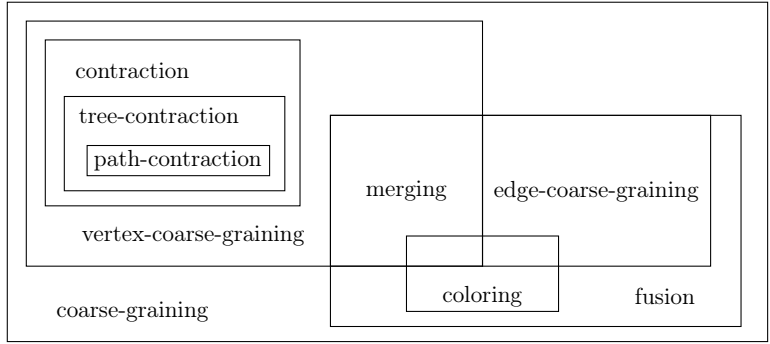
\begin{center}
\begin{tikzpicture}[scale=1]
\draw  (-2,0.25) rectangle (4,-3);
\draw  (4,-1) node (v1) {} rectangle (7,-3);
\draw  (-1.75,0) rectangle (1.6,-2.2);
\draw  (2,-1) node (v2) {} rectangle (7.4,-3.8);
\draw  (2,-1) rectangle (4,-3);
\node [scale=0.8] at (-0.6,-0.4) {contraction};
\node [scale=0.8] at (3,-2) {merging};
\node[scale=0.8] at (5.5,-2) {edge-coarse-graining};
\node [scale=0.8] at (6.4,-3.4) {fusion};
\draw  (-2.25,0.5) rectangle (7.75,-4);
\node [scale=0.8] at (-0.6,-3.6) {coarse-graining};
\node[scale=0.8] at (0.2,-2.6) {vertex-coarse-graining};
\draw  (-1.5,-0.75) rectangle (1.4,-2);
\node [scale=0.8] at (-0.25,-1) {tree-contraction};
\draw  (-1.2,-1.4) rectangle (1.2,-1.8);
\node [scale=0.8] at (0,-1.6) {path-contraction};
\draw  (3,-2.6) rectangle (5,-3.6);
\node [scale=0.8]at (4,-3.4) {coloring};
\end{tikzpicture}
\end{center}

\subsection{Linear-coloring and sorting-net}
In this subsection, we discuss a special kind of vertex-coloring, which has been considered in \cite{[HR86]} as a generalization of the notion of a \textbf{topological sorting}.
\begin{defn}
A \textbf{linear-coloring} of causal-net $G$ is a labelling of vertices of $G$ with natural numbers (or elements of any linear poset) such that for any edge $e$, the labelling number of $s(e)$ is smaller than that of $t(e)$.
\end{defn}
In other words, a linear-coloring $L:V(G)\to \mathbb{N}$ of $G$ is a strictly isotone from the vertex poset of $G$ to the linearly ordered poset $\mathbb{N}$.
A linear-coloring is a topological sorting if and only if it is an one-to-one labelling, and in this case, it is equivalent to
a \textbf{linear-extension} of the vertex poset.
Since any two comparable vertices must be labelled by two different natural numbers, a linear-coloring partitions the vertex set into cocliques. Moreover, it also produces a linear order on cocliques. The number of cocliques (or labelling numbers) is called the \textbf{order} of the linear-coloring.
Recall that a non-empty causal-net is called a path if all its edges form a directed path. A linear-coloring of causal-net $G$ can be naturally viewed as a morphism  with no contractions from $G$ to a path.

The following notion may be useful for developing a theory of causal-net complex  analog to the Kontsevich  graph complex \cite{[K93]}.
\begin{defn}
A  \textbf{sorting-net of $n$-th order} is a simple causal-net with $n$ linearly-ordered vertices $v_1<v_2<\cdots<v_n$ such that for any directed edge $e$,  $s(e)<t(e)$.
\end{defn}

In short, a sorting-net is just a simple causal-net equipped with a topological sorting. A non-simple causal-net equipped with a topological sorting is called a \textbf{sorting-Net}.

A pair  $\{v_i,v_{i+1}\}$ of successive vertices of a sorting-net $G$ is called a \textbf{gap} if they share no edges. Clearly, gaps must be cocliques.
If $\{v_i,v_{i+1}\}$ is a gap of a sorting-net $G$ with $n$-th  order, then we can define a quotient sorting-net $G//\{i,i+1\}$ to be the simplification of the merging of $G$ by  identifying $v_i$ and $v_{i+1}$. The quotient $n-1$-th sorting-net $G//\{i,i+1\}$ is called the \textbf{condensation} of $G$ at $\{i,i+1\}$. We call the process of eliminating a gap  a \textbf{condensation} of a sorting-net.

\begin{ex}
The following is an example of a sorting-net, which has three gaps $\{v_1,v_{2}\}$, $\{v_2,v_{3}\}$ and $\{v_3,v_{4}\}$.
\begin{center}
\begin{tikzpicture}
\node (v1) at (-2,0) {};
\node (v3) at (0,0) {};
\node (v2) at (2,0) {};
\node(v5) at (4,0) {};
\node (v4) at (6,0) {};
\draw  plot[smooth, tension=.7] coordinates {(-2,0) (0,0.5) (2,0)}[postaction={decorate, decoration={markings,mark=at position .5 with {\arrow[black]{stealth}}}}];
\draw  plot[smooth, tension=.7] coordinates {(0,0) (3,-0.5) (6,0)}[postaction={decorate, decoration={markings,mark=at position .5 with {\arrow[black]{stealth}}}}];
\draw (4,0)-- (6,0)[postaction={decorate, decoration={markings,mark=at position .5 with {\arrow[black]{stealth}}}}];
\draw [fill](v1) circle [radius=0.06];
\draw [fill](v2) circle [radius=0.06];
\draw [fill](v3) circle [radius=0.06];
\draw [fill](v4) circle [radius=0.06];
\draw [fill](v5) circle [radius=0.06];
\node at (-2,-0.4) {$v_1$};
\node at (0,-0.4) {$v_2$};
\node at (2,0.4) {$v_3$};
\node at (4,0.4) {$v_4$};
\node at (6,0.4) {$v_5$};
\end{tikzpicture}
\end{center}
\end{ex}

\begin{ex}
The following shows an example of a $6$-th sorting-net and its condensation at $\{2,3\}$.\\

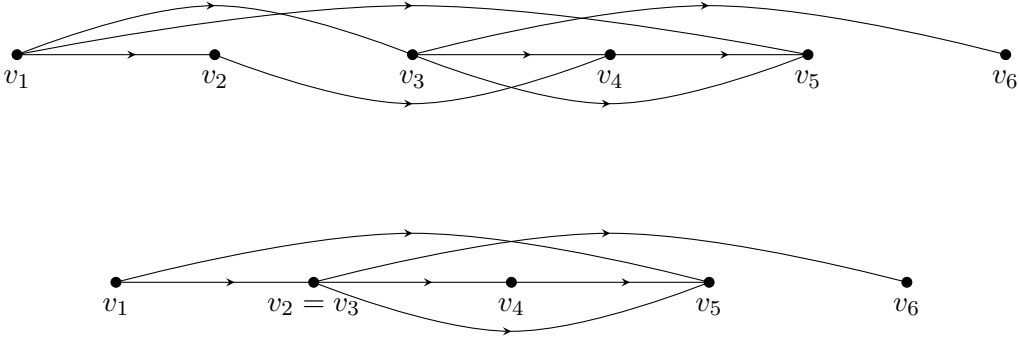
\begin{center}
\begin{tikzpicture}[scale=1.3]
\node (v1) at (-3,0) {};
\node (v3) at (-1,0) {};
\node (v2) at (1,0) {};
\node (v4) at (3,0) {};
\node (v5) at (5,0) {};
\draw  plot[smooth, tension=.7] coordinates {(-3,0)  (-1,0.5) (1,0)}[postaction={decorate, decoration={markings,mark=at position .5 with {\arrow[black]{stealth}}}}];

\draw  (-3,0)  -- (-1,0)[postaction={decorate, decoration={markings,mark=at position .6 with {\arrow[black]{stealth}}}}];
\draw  (1,0) -- (3,0)[postaction={decorate, decoration={markings,mark=at position .6 with {\arrow[black]{stealth}}}}];
\draw  (3,0) -- (5,0)[postaction={decorate, decoration={markings,mark=at position .6 with {\arrow[black]{stealth}}}}];
\draw  plot[smooth, tension=.7] coordinates {(-1,0)  (1,-0.5) (3,0)}[postaction={decorate, decoration={markings,mark=at position .5 with {\arrow[black]{stealth}}}}];
\node (v6) at (7,0) {};
\draw  plot[smooth, tension=.7] coordinates {(-3,0) (1,0.5) (5,0)}[postaction={decorate, decoration={markings,mark=at position .5 with {\arrow[black]{stealth}}}}];
\draw  plot[smooth, tension=.7] coordinates {(1,0) (4,0.5) (7,0)}[postaction={decorate, decoration={markings,mark=at position .5with {\arrow[black]{stealth}}}}];
\draw  plot[smooth, tension=.7] coordinates {(1,0) (3,-0.5) (5,0)}[postaction={decorate, decoration={markings,mark=at position .5 with {\arrow[black]{stealth}}}}];
\draw [fill](v1) circle [radius=0.05];
\draw [fill](v2) circle [radius=0.05];
\draw [fill](v3) circle [radius=0.05];
\draw [fill](v4) circle [radius=0.05];
\draw [fill](v5) circle [radius=0.05];
\draw [fill](v6) circle [radius=0.05];
\node at (-3,-0.25) {$v_1$};
\node at (-1,-0.25) {$v_2$};
\node at (1,-0.25) {$v_3$};
\node at (3,-0.25) {$v_4$};
\node at (5,-0.25) {$v_5$};
\node at (7,-0.25) {$v_6$};
\end{tikzpicture}
\end{center}\ \\ \ \

\begin{center}
\begin{tikzpicture}[scale=1.3]
\node (v1) at (-1,0) {};
\node (v3) at (-1,0) {};
\node (v2) at (1,0) {};
\node (v4) at (3,0) {};
\node (v5) at (5,0) {};
\draw  (-1,0)  -- (1,0)[postaction={decorate, decoration={markings,mark=at position .6 with {\arrow[black]{stealth}}}}];
\draw  (1,0) -- (3,0)[postaction={decorate, decoration={markings,mark=at position .6 with {\arrow[black]{stealth}}}}];
\draw  (3,0) -- (5,0)[postaction={decorate, decoration={markings,mark=at position .6 with {\arrow[black]{stealth}}}}];
\node (v6) at (7,0) {};
\draw  plot[smooth, tension=.7] coordinates {(-1,0) (2,0.5) (5,0)}[postaction={decorate, decoration={markings,mark=at position .5 with {\arrow[black]{stealth}}}}];
\draw  plot[smooth, tension=.7] coordinates {(1,0) (4,0.5) (7,0)}[postaction={decorate, decoration={markings,mark=at position .5with {\arrow[black]{stealth}}}}];
\draw  plot[smooth, tension=.7] coordinates {(1,0) (3,-0.5) (5,0)}[postaction={decorate, decoration={markings,mark=at position .5 with {\arrow[black]{stealth}}}}];
\draw [fill](v1) circle [radius=0.05];
\draw [fill](v2) circle [radius=0.05];
\draw [fill](v4) circle [radius=0.05];
\draw [fill](v5) circle [radius=0.05];
\draw [fill](v6) circle [radius=0.05];
\node at (-1,-0.25) {$v_1$};
\node at (1,-0.25) {$v_2=v_3$};
\node at (3,-0.25) {$v_4$};
\node at (5,-0.25) {$v_5$};
\node at (7,-0.25) {$v_6$};
\end{tikzpicture}
\end{center}
\end{ex}

The relation of  linear-colorings and sorting-nets is same as that of colorings and simple causal-nets.

\begin{thm}
A linear-coloring $L$ of causal-net $G$ is equivalent to a causal-coloring $\lambda_L:G\to S$ from $G$ to a sorting-net $S$.
\end{thm}
\begin{proof}
Let $L:V(G)\to \mathbb{N}$ be a linear-coloring of $G$. Without loss of generality, we assume the labelling numbers are $\{1,2,\cdots,n\}$, and define a $n$-th sorting-net $S$ as follows. The vertex set of $S$ is $\{1,2,\cdots,n\}$ with the natural order relation. For any two vertices $i<j$, there is an edge, denoted as $(i,j)$, with source $i$ and target $j$ if and only if there exists an edge of $G$ starting with an $i$-vertex (a vertex labelled by $i$), ending with a $j$-vertex. It is not difficult to check that $S$ is simple and acyclic.

The causal-coloring $\lambda_L:G\to S$ associated with $L$ is defined as follows. $\lambda_L$ maps each $i$-vertex of $G$ to the vertex $i$ of $S$, and maps each edge $e$ of $G$ to the edge $(L(s(e)), L(t(e)))$ of $S$. From this definition, we see that $\lambda_L$ has no subdivisions and no contractions. By the definition of $S$, we can see that $\lambda_L$ is surjective both on vertices and edges. So, $\lambda_L$ is a fusion from $G$ to $S$ and is totally determined by $L$.

Conversely, it is easy to see that any causal-coloring $\lambda:G\to S$ from $G$ to a sorting-net $S$ naturally defines a linear-coloring $L_{\lambda}$ of $G$. It is not difficult to see that the two constructions,  from $L$ to $\lambda_L$ and from $\lambda$ to $L_\lambda$, are inverse mappings of each other.
\end{proof}
In a word, we say that
\begin{center}
\begin{tikzpicture}[scale=0.8]
\draw  (-4.5,1) rectangle (-0.5,0);
\draw  (0.5,1) rectangle (4.5,0);
\draw  (5.5,1) rectangle (9.5,0);
\node at (-2.5,0.5) {linear-coloring};
\node at (0,0.5) {$=$};
\node at (2.5,0.5) {causal-coloring};
\node at (7.5,0.5) {topological sorting};
\node at (5,0.5) {$+$};
\end{tikzpicture}
\end{center}
\begin{defn}
Two $n$-th order linear-colorings $L_1,L_2:V(G)\to \{1,2,\cdots,n\}$ of causal-net $G$ are called \textbf{similar} if there is a permutation $\sigma:\{1,2,\cdots,n\}\to \{1,2,\cdots,n\}$ such that  $L_2=\sigma\circ L_1$.
\end{defn}
\begin{center}
\begin{tikzpicture}
\node (v1) at (-1.5,0) {$V(G)$};
\node (v3) at (1.5,-1) {$\{1,2,\cdots,n\}$};
\node (v2) at (1.5,1) {$\{1,2,\cdots,n\}$};
\draw[-latex]  (v1) edge (v2);
\draw [-latex] (v1) edge (v3);
\draw [latex-] (v3) edge (v2);
\node at (1.75,0) {$\sigma$};
\node at (-0.25,0.75) {$L_1$};
\node at (-0.25,-0.75) {$L_2$};
\end{tikzpicture}
\end{center}

\begin{ex}
The following two linear-colorings are similar.
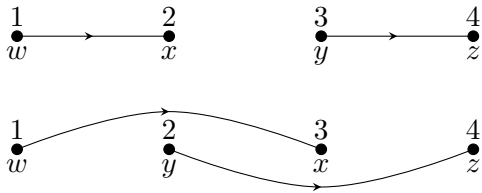
\begin{center}
\begin{tikzpicture}

\node (v1) at (-3,0) [above]{$1$};
\node (v2) at (-1,0)[above] {$2$};
\node (v3) at (1,0) [above]{$3$};
\node (v4) at (3,0)[above] {$4$};
\node  at (-3,0) [below]{$w$};
\node  at (-1,0)[below] {$x$};
\node at (1,0) [below]{$y$};
\node at (3,0)[below] {$z$};
\draw  (-3,0)--(-1,0)[postaction={decorate, decoration={markings,mark=at position .5 with {\arrow[black]{stealth}}}}];
\draw  (1,0) -- (3,0)[postaction={decorate, decoration={markings,mark=at position .5 with {\arrow[black]{stealth}}}}];
\node (v5) at (-3,-1.5)[above] {$1$};
\node (v7) at (-1,-1.5) [above]{$2$};
\node (v6) at (1,-1.5)[above] {$3$};
\node (v8) at (3,-1.5)[above] {$4$};
\node at (-3,-1.5)[below] {$w$};
\node at (-1,-1.5) [below]{$y$};
\node at (1,-1.5)[below] {$x$};
\node at (3,-1.5)[below] {$z$};
\draw  plot[smooth, tension=.7] coordinates {(-3,-1.5) (-1,-1) (1,-1.5)}[postaction={decorate, decoration={markings,mark=at position .5 with {\arrow[black]{stealth}}}}];
\draw  plot[smooth, tension=.7] coordinates {(-1,-1.5) (1,-2) (3,-1.5)}[postaction={decorate, decoration={markings,mark=at position .5 with {\arrow[black]{stealth}}}}];
\draw [fill] (-3,0)  circle [radius=0.07];
\draw [fill](-1,0) circle [radius=0.07];
\draw [fill](1,0) circle [radius=0.07];
\draw [fill](3,0) circle [radius=0.07];
\draw [fill](-3,-1.5) circle [radius=0.07];
\draw [fill](-1,-1.5) circle [radius=0.07];
\draw [fill](1,-1.5) circle [radius=0.07];
\draw [fill](3,-1.5) circle [radius=0.07];

\end{tikzpicture}
\end{center}
\end{ex}

The following proposition shows that causal-colorings can be characterized as similarity classes of linear-colorings.

\begin{prop}
For any causal-net, the set of causal-colorings is in bijective with the set of similarity classes of linear-colorings.
\end{prop}
\begin{proof}
By Corollary \ref{color}, a causal-coloring is totally determined by its vertex partition, that is, all its fibers. So
to prove this theorem, we only need to show that for any two linear-colorings $L_1,L_2:V(G)\to \{1,2,\cdots,n\}$ of causal-net $G$ are similar if and only if they produce the same partition of $V(G)$. In fact, if $L_1, L_2$ are similar, that is, there exists a permutation $\sigma:\{1,\cdots,n\}\to \{1,\cdots,n\}$, such that $\sigma\circ L_1=L_2$, then it is easy to see that for any two vertices $v_1$, $v_2$ of $G$,  $L_1(v_1)=L_1(v_2)$ if and only if  $L_2(v_1)=L_2(v_2)$, which implies that $L_1, L_2$ produces the same partition of $V(G)$. Conversely, if $L_1, L_2$ produces the same partition of $V(G)$, then we define $\sigma(L_1(B_i))=L_2(B_i)$, for each block $B_i$ $(1\leq i\leq n)$, which is a permutation.
\end{proof}

\begin{defn}
A sorting-net is called \textbf{minimal} if it has no gaps and a linear-coloring is \textbf{minimal} if its corresponding sorting-net is minimal.
\end{defn}

Clearly, a sorting-net is minimal if and only if it has a hamiltonian path.  Any sorting-net can be transformed into a minimal sorting-net through a sequence of condensations. Moreover, the resulting minimal sorting-net is unique, which is independent of the order of condensations.

Harmonic causal-nets, as shown in Example \ref{K3} and Example \ref{K4}, can be naturally viewed as minimal sorting-nets with their vertices linearly ordered by their unique topological sortings (defined through hamiltonian paths). That is, a harmonic causal-net has a unique (canonical) way to be a minimal sorting-net. Conversely, it is not difficult to see that the underlying causal-net of a minimal sorting-net must be harmonic. In summary, harmonic causal-nets and minimal sorting-nets are essentially equivalent notions.

\begin{prop}
The set of harmonic causal-nets are naturally in bijective with the set of minimal sorting-nets.
\end{prop}

\section{Inclusion=monomorphism}
In this section, we introduce various types of inclusions, whose parallel notions are of broad interest in ordinary graph theory. We start with the most general type---the purely category-theoretic inclusions (monomorphisms).

\begin{defn}
A morphism of causal-nets is called an \textbf{inclusion} if it is injective on objects and faithful on morphisms.
\end{defn}
In other words, an inclusion is a morphism which is injective both on vertices and on directed paths. An inclusion can be characterized as a morphism with no multiple-vertices, no multiple-edges and  no contractions.

A common vertex of two directed path is called a \textbf{joint-vertex}; a common edge of two directed path is called a \textbf{joint-edge}.   As the following example shown, an inclusion may create joint-vertices and joint-edges. Clear, the class of inclusions is closed under composition.

\begin{ex}\label{joint-edge}
The following figure shows an example of an inclusion with $\lambda(h_1)=e_3e_1$, $\lambda(h_2)=e_3e_2$. In this example, the inclusion creates a joint-edge $e_3$ and a joint-vertex  $v$.
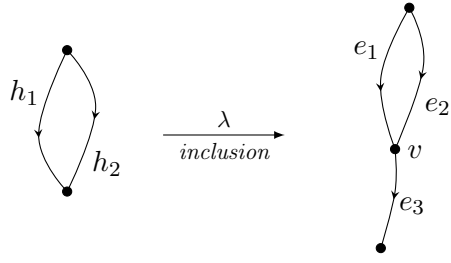
\begin{center}
\begin{tikzpicture}[scale=0.75]
\node (v1) at (-1.5,2) {};
\draw [fill](v1) circle [radius=0.08];
\node (v2) at (-1.5,-0.5) {};
\draw [fill](v2) circle [radius=0.08];
\draw  plot[smooth, tension=.7] coordinates {(-1.5,2) (-2,0.5) (-1.5,-0.5)}[postaction={decorate, decoration={markings,mark=at position .6 with {\arrow[black]{stealth}}}}];
\draw  plot[smooth, tension=.7] coordinates {(-1.5,2) (-1,1) (-1.5,-0.5)}[postaction={decorate, decoration={markings,mark=at position .5 with {\arrow[black]{stealth}}}}];
\node (v3) at (4.5,2.75) {};
\node (v4) at (4.25,0.25) {};
\node (v5) at (4,-1.5) {};
\draw [fill](v3) circle [radius=0.08];
\draw [fill](v4) circle [radius=0.08];
\draw [fill](v5) circle [radius=0.08];
\draw  plot[smooth, tension=.7] coordinates {(v3) (4,1.5) (v4)}[postaction={decorate, decoration={markings,mark=at position .6 with {\arrow[black]{stealth}}}}];
\draw  plot[smooth, tension=.7] coordinates {(4.5,2.75) (4.75,1.75) (4.25,0.25)}[postaction={decorate, decoration={markings,mark=at position .5 with {\arrow[black]{stealth}}}}];
\draw  plot[smooth, tension=.7] coordinates {(4.25,0.25) (4.25,-0.5) (v5)}[postaction={decorate, decoration={markings,mark=at position .5 with {\arrow[black]{stealth}}}}];
\node (v6) at (0,0.5) {};
\node (v7) at (2.5,0.5) {};
\draw [-latex] (v6) edge node[sloped,scale=0.8, below] {inclusion} node[sloped,scale=0.8, above] {$\lambda$} (v7);
\node at (-2.25,1.25) {$h_1$};
\node at (-0.8,0) {$h_2$};
\node at (3.75,2) {$e_1$};
\node at (5,1) {$e_2$};
\node at (4.55,-0.75) {$e_3$};
\node at (4.6,0.2) {$v$};
\end{tikzpicture}
\end{center}
\end{ex}
Inclusions are exactly monomorphisms in $\mathbf{Cau}$.
\begin{thm}
In $\mathbf{Cau}$, a morphism is an inclusion if and only if it is a monomorphism.
\end{thm}
\begin{proof}
$(\Rightarrow)$ Let $\lambda:G_1\to G_2$ be an inclusion and $\lambda_1:H\to G_1$, $\lambda_2:H\to G_1$ be any two morphisms such that $\lambda\circ \lambda_1=\lambda\circ\lambda_2$. We want to show that $\lambda_1=\lambda_2$. For this, we only need to show that for each $w\in V(H)$, $\lambda_1(w)=\lambda_2(w)$ and for each $h\in E(H)$, $\lambda_1(h)=\lambda_2(h)$. In fact, for each $w\in V(H)$ and for each $h\in E(H)$, we have $\lambda\circ \lambda_1(w)=\lambda\circ\lambda_2(w)$ and $\lambda\circ \lambda_1(h)=\lambda\circ\lambda_1(h)$. Since $\lambda$ is an inclusion, we must have $\lambda_1(w)=\lambda_2(w)$ and $\lambda_1(h)=\lambda_2(h)$.

$(\Leftarrow)$ Let $\lambda:G_1\to G_2$ be a monomorphism,  we want to prove by contradiction that it is an inclusion. If $\lambda$ is not injective on vertices, that is, there are two distinct vertices $v_1,v_2$ of $G_1$ such that $\lambda(v_1)=\lambda(v_2)$, then we consider a causal-net $H$ formed by exactly one isolated vertex $w$ and define two morphisms $\lambda_1:H\to G_1$, $\lambda_2:H\to G_1$ by $\lambda_1(w)=v_1$ and $\lambda_2(w)=v_2$. Clearly, $\lambda\circ\lambda_1=\lambda\circ\lambda_2$ but $\lambda_1\neq\lambda_2$, which leads a contradiction.

Having shown that $\lambda$ must be injective on vertices, now we prove by contradiction that it is also injective on morphisms. Without loss of generality, assume there are two distinct parallel directed paths $\overrightarrow{p_1},\overrightarrow{p_2}:v_1\to v_2$ of $G_1$ such that $\lambda(\overrightarrow{p_1})=\lambda(\overrightarrow{p_2})$, then we consider a causal-net $H$ formed by two vertices $w_1,w_2$ and a directed edge $e:w_1\to w_2$, and define two morphisms $\lambda_1:H\to G_1$, $\lambda_2:H\to G_1$ as follows. $\lambda_1(w_1)=\lambda_2(w_1)=v_1$, $\lambda_1(w_2)=\lambda_2(w_2)=v_2$ and $\lambda_1(e)=\overrightarrow{p_1},\lambda_2(e)=\overrightarrow{p_2}$. Then we have $\lambda\circ\lambda_1=\lambda\circ\lambda_2$ but $\lambda_1\neq\lambda_2$, which leads a contradiction.
\end{proof}

The following theorem is a simple graph-theoretic fact.
\begin{thm}\label{coar-inclu}
Any morphism $\lambda$ of causal-nets can be uniquely represented as a composition of  a coarse-graining $\lambda_{cg}$ and an inclusion $\lambda_{in}$.
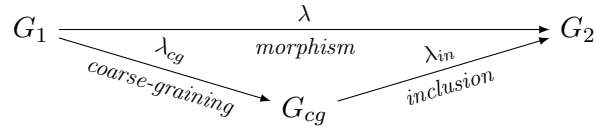
\begin{center}
\begin{tikzpicture}[scale=1.2]
\node (v1) at (-2,0.5) {$G_1$};
\node (v2) at (4,0.5) {$G_2$};
\draw[-latex]  (v1) edge node[sloped,scale=0.8, below] {morphism} node[sloped,scale=0.8, above] {$\lambda$}(v2);
\node (v4) at (1,-0.4) {$G_{cg}$};
\draw[-latex]  (v1) edge node[sloped,scale=0.8, above] {$\lambda_{cg}$}node[sloped,scale=0.8, below] {coarse-graining} (v4);
\draw[-latex]  (v4) edge node[sloped,scale=0.8, above] {$\lambda_{in}$}node[sloped,scale=0.8, below] {inclusion}(v2);
\end{tikzpicture}
\end{center}
\end{thm}

\begin{proof}
Edges of $G_1$ can be classified into three classes: contractions, segments and subdivisions, that is, $E(G_1)=Con(\lambda)\sqcup Seg(\lambda)\sqcup Sub(\lambda)$.

We define $G_{cg}$ as follows. $$V(G_{cg})=\displaystyle \frac{V(G_1)}{\sim_{\scriptscriptstyle{\lambda}}},\ \ E(G_{cg})=\displaystyle\frac{Seg(\lambda)\sqcup Subd(\lambda)}{\sim_\lambda},$$
where $v_1\sim_{\scriptscriptstyle{\lambda}} v_2  \Longleftrightarrow \lambda(v_1)=\lambda(v_2)$, $e_1\sim_{\scriptscriptstyle{\lambda}} e_2  \Longleftrightarrow \lambda(e_1)=\lambda(e_2)$ as directed paths.  Clearly, the natural quotient $\lambda_{cg}:G_1\to G_{cg}$ has no subdivisions (for any $e\in Seg(\lambda)\sqcup Subd(\lambda)$, $\lambda_{cg}(e)$ is an edge of $G_{cg}$), thus it is a coarse-graining.

The morphism $\lambda_{in}:G_{cg}\to G_2$ is defined as follows. For any $v\in V(G_2)$, $\lambda_{in}([v]_{\sim_{\scriptscriptstyle{\lambda}}})=\lambda(v)$; for any $e\in Seg(\lambda)\sqcup Subd(\lambda)$, $\lambda_{in}([e]_{\sim_{\scriptscriptstyle{\lambda}}})=\lambda(e)$.
Evidently, $\lambda_{in}$ has no contractions and is an inclusion.
\end{proof}

\subsection{Immersion and strong-immersion}
In this subsection, we introduce two types of inclusions whose parallel notions in ordinary graph theory characterize those of an immersion-minor \cite{[W18]} and a strong-immersion-minor \cite{[RS10]}, respectively.

Given a morphism $\lambda:G_1\to G_2$, two subdivisions $e_1,e_2$ of $\lambda$ are called \textbf{edge-disjoint} if their images $\lambda(e_1)$ and $\lambda(e_2)$ have no common edges. We call $\lambda$ \textbf{edge-disjoint} if any two subdivisions of $\lambda$ are edge-disjoint, or in other words, if $\lambda$ does not create any joint-edges.
\begin{defn}
An inclusion is called an \textbf{immersion} if it is edge-disjoint.
\end{defn}
An immersion is an inclusion that does not create any joint-edges. The inclusion in Example \ref{joint-edge} is not an immersion because it creates a joint-edge $e_3$.  But as shown in the following example, an immersion may create joint-vertices.

\begin{ex}\label{immersion}
The following $\lambda$ is an example of an immersion with $\lambda(h_1)=e_3e_1, \lambda(h_2)=e_4e_2$, $\lambda(h)=e$, $\lambda(w)=v$, and $v$ is the unique joint-vertex created by $\lambda$.
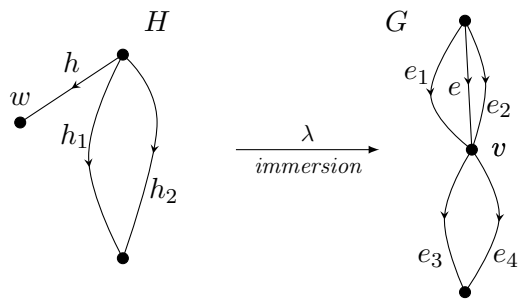
\begin{center}
\begin{tikzpicture}[scale=0.9]
\node (v1) at (-0.5,2) {};
\node (v2) at (-0.5,-1) {};
\draw [fill](v1) circle [radius=0.08];\draw [fill](v2) circle [radius=0.08];
\draw  plot[smooth, tension=.7] coordinates {(v1) (-1,0.5) (v2)}[postaction={decorate, decoration={markings,mark=at position .55 with {\arrow[black]{stealth}}}}];
\draw  plot[smooth, tension=.7] coordinates {(-0.5,2) (0,1) (-0.5,-1)}[postaction={decorate, decoration={markings,mark=at position .5 with {\arrow[black]{stealth}}}}];
\node (v3) at (4.5,2.5) {};
\node (v4) at (4.6,0.6) {};
\node (v5) at (4.5,-1.5) {};
\draw [fill](v3) circle [radius=0.08];
\draw [fill](v4) circle [radius=0.08];
\draw [fill](v5) circle [radius=0.08];
\draw  plot[smooth, tension=.7] coordinates {(v3) (4,1.4) (v4)}[postaction={decorate, decoration={markings,mark=at position .55 with {\arrow[black]{stealth}}}}];
\draw  plot[smooth, tension=.7] coordinates {(v4) (4.2,-0.4) (v5)}[postaction={decorate, decoration={markings,mark=at position .5 with {\arrow[black]{stealth}}}}];
\draw  plot[smooth, tension=.7] coordinates {(v3) (4.8,1.6) (4.6,0.6)}[postaction={decorate, decoration={markings,mark=at position .5 with {\arrow[black]{stealth}}}}];
\draw  plot[smooth, tension=.7] coordinates {(4.6,0.6) (5,-0.4) (v5)}[postaction={decorate, decoration={markings,mark=at position .5 with {\arrow[black]{stealth}}}}];
\node (v6) at (1,0.6) {};
\node (v7) at (3.4,0.6) {};
\draw   [-latex](v6) edge node[sloped,scale=0.8, below] {immersion}node[sloped,scale=0.8, above] {$\lambda$} (v7);
\node at (-1.2,0.8) {$h_1$};
\node at (0.1,0) {$h_2$};
\node at (3.8,1.7) {$e_1$};
\node at (5,1.2) {$e_2$};
\node at (4,-1) {$e_3$};
\node at (5.1,-1) {$e_4$};
\node at (0,2.5) {$H$};
\node at (3.5,2.5) {$G$};
\node at (5,0.6) {$v$};
\node [above]at (-2,1.1) {$w$};
\draw [fill](-2,1) circle [radius=0.08];
\draw(-0.5,2)--(-2,1)[postaction={decorate, decoration={markings,mark=at position .5 with {\arrow[black]{stealth}}}}];
\draw(4.5,2.5)--(4.6,0.6)[postaction={decorate, decoration={markings,mark=at position .5 with {\arrow[black]{stealth}}}}];
\node at (-1.25,1.9) {$h$};
\node at (4.35,1.5) {$e$};
\node at (5,0.6) {$v$};
\end{tikzpicture}
\end{center}
\end{ex}

A causal-net $H$ is called an \textbf{immersion-minor} \cite{[W18]} of $G$, if $H$ can be obtained from $G$ by a sequence of operations of  deleting an isolated vertex, deleting an edge and an \textbf{edge-lifting} (replacing a length two path $e_2e_1$ by a length one path $h$, with the middle vertex of $e_2e_1$ retained). All these three types of graphical operations can be represented in the opposite direction by morphisms of causal-nets.

As the following example shown, an edge-lifting can be represented in the opposite direction as a composition of a merging after an edge-subdivision. Clearly, an edge-lifting does not change the number of vertices, but reduces the number of edges by one.
\begin{ex}\label{edge-lift}
The following figure shows an example of an edge-lifting, which can be represented by a morphism in the \textbf{opposite} direction, with $\lambda(h)=e_2e_1$, $\lambda(w)=v$. In this example, we say that the directed path $e_2e_1$ is lifted to an edge $h$.

\begin{center}
\begin{tikzpicture}[scale=0.75]
\node (v1) at (-1,2) {};
\node (v2) at (0,0) {};
\node (v3) at (-1,-2) {};
\node (v4) at (2.5,0) {};
\draw [fill](v1) circle [radius=0.08];
\draw [fill](v2) circle [radius=0.08];
\draw [fill](v3) circle [radius=0.08];
\draw [fill](v4) circle [radius=0.08];
\draw  (-1,2) -- (0,0)[postaction={decorate, decoration={markings,mark=at position .5 with {\arrow[black]{stealth}}}}];
\draw  (0,0) -- (-1,-2)[postaction={decorate, decoration={markings,mark=at position .5 with {\arrow[black]{stealth}}}}];
\draw  (0,0) -- (2.5,0)[postaction={decorate, decoration={markings,mark=at position .5 with {\arrow[black]{stealth}}}}];
\node (v5) at (7.5,1.5) {};
\node (v6) at (7.5,-1.5) {};
\node (v7) at (8.5,0) {};
\node (v8) at (11,0) {};
\draw [fill](v5) circle [radius=0.08];
\draw [fill](v6) circle [radius=0.08];
\draw [fill](v7) circle [radius=0.08];
\draw [fill](v8) circle [radius=0.08];
\draw  (7.5,1.5) -- (7.5,-1.5)[postaction={decorate, decoration={markings,mark=at position .5 with {\arrow[black]{stealth}}}}];
\draw  (8.5,0) -- (11,0)[postaction={decorate, decoration={markings,mark=at position .5 with {\arrow[black]{stealth}}}}];
\node at (0.3,0.3) {$v$};
\node (v9) at (3.5,0.5) {};
\node (v10) at (6.5,0.5) {};
\draw [dashed,-latex] (v9) edge node[sloped,scale=0.8, above] {edge-lifting}(v10);
\node (v11) at (6.5,-0.5) {};
\node (v12) at (3.5,-0.5) {};
\draw [-latex] (v11) edge  node[sloped,scale=0.8, below] {$\lambda$} node[sloped,scale=0.8, above] {immersion} (v12);
\node at (8.5,0.4) {$w$};
\node at (-1,1) {$e_1$};
\node at (-1,-1) {$e_2$};
\node at (7.2,0) {$h$};
\node at (10,-1.5) {$H$};
\node at (1,-1.5) {$G$};
\node (v9) at (3.5,0.5) {};
\node (v10) at (6.5,0.5) {};
\node (v11) at (6.5,-0.5) {};
\node (v12) at (3.5,-0.5) {};
\node at (-1,1) {$e_1$};
\node at (-1,-1) {$e_2$};
\node at (7.2,0) {$h$};
\node at (10,-1.5) {$H$};
\node at (1,-1.5) {$G$};
\node (v13) at (5,-4) {};
\node (v15) at (5,-7) {};
\node (v14) at (5,-5.5) {};
\node (v16) at (6,-5.5) {};
\node (v17) at (8.5,-5.5) {};
\draw [fill](v13) circle [radius=0.08];
\draw [fill](v14) circle [radius=0.08];
\draw [fill](v15) circle [radius=0.08];
\draw [fill](v16) circle [radius=0.08];
\draw [fill](v17) circle [radius=0.08];
\draw  (5,-4) --(5,-5.5)[postaction={decorate, decoration={markings,mark=at position .5 with {\arrow[black]{stealth}}}}];
\draw  (5,-5.5) -- (5,-7)[postaction={decorate, decoration={markings,mark=at position .5 with {\arrow[black]{stealth}}}}];
\draw  (6,-5.5) -- (8.5,-5.5)[postaction={decorate, decoration={markings,mark=at position .5 with {\arrow[black]{stealth}}}}];
\node at (4.5,-5.5) {$w'$};
\node at (4.5,-4.75) {$e_1$};
\node at (4.5,-6.25) {$e_2$};
\node at (6,-5.15) {$w$};
\node (v21) at (6.5,-4) {};
\node (v20) at (8,-2.5) {};
\node (v18) at (3.25,-4) {};
\node (v19) at (1.5,-2.5) {};
\draw [-latex] (v18) edge  node[sloped,scale=0.8, above] {merging $w,w'$}(v19);
\draw [-latex] (v20) edge node[sloped,scale=0.8, above] {subdividing $h$} (v21);
\end{tikzpicture}
\end{center}

\end{ex}

Immersions are inclusions that do not create any joint-edges, hence they are closed under composition.
The following theorem shows that immersions characterize immersion-minors exactly.
\begin{thm}\label{immersion-minor}
Let $H$, $G$ be two causal-nets.
$H$ is an immersion-minor of $G$ if and only if there is an immersion $\lambda:H\to G$.
\end{thm}
\begin{proof}
$(\Rightarrow)$ If $H$ is an immersion-minor of $G$, then $H$ can be obtained from $G$ by a sequence of operations of deleting an isolated vertex, deleting an edge and an edge-lifting, which means that there is a sequence of causal-nets $G=K_0\overset{O_1}{\rightsquigarrow} K_1\overset{O_2}{\rightsquigarrow}\cdots\overset{O_{n-1}}{\rightsquigarrow} K_{n-1}\overset{O_n}{\rightsquigarrow} K_{n}=H$ connected by three types of operations $O_i:K_{i-1}\rightsquigarrow K_{i}$ $(1\leq i\leq n)$. Each $O_i:K_{i-1}\rightsquigarrow K_{i}$ of these three types of operations  can be represented in the opposite direction by an immersion $\lambda_{i}:K_{i}\to K_{i-1}$. Since immersions are closed under composition, then $\lambda=\lambda_1\circ\cdots\circ\lambda_n:H\to G$ is an immersion.

$(\Leftarrow)$ We prove this direction by induction on the number $n$ of subdivisions. Let $\lambda:H\to G$ be an immersion with $n$ subdivisions. If $n=0$, then the image of $H$ is just a sub-causal-net of $G$. In this case, $H$ can be obtained from $G$ by a series of operations of deleting an isolated vertex and deleting an edge, which means that $H$ is an immersion-minor of $G$.

Assume that for $n\leq k$, the existence of an immersion implies that the immersion-minor relation. For $n=k+1$, let $h$ be an subdivision of $\lambda$ and assume $\lambda(h)=e_{n}e_{n-1}\cdots e_0$, then we can construct a sequence of causal-nets $K_0,K_1, \cdots, K_n$ with $K_0=G$ and a sequence of edge-liftings $G=K_0\overset{O_1}{\rightsquigarrow} K_1\overset{O_2}{\rightsquigarrow}\cdots\overset{O_{n-1}}{\rightsquigarrow} K_{n-1}\overset{O_n}{\rightsquigarrow} K_{n}$, where $O_1$ lifts the path $e_{1}e_{0}$ to edge $h_1$, $O_2$ lifts the path $e_2h_1$ to edge $h_2$, $\cdots$, $O_n$ lifts the path $e_nh_{n-1}$ to edge $h_n$.  The sequence of edge-liftings can be represented in the opposite direction by a sequence of immersions $K_n\overset{\phi_n}{\to} K_{n-1}\overset{\phi_{n-1}}{\to}\cdots\overset{\phi_2}{\to} K_1\overset{\phi_1}{\to} G=K_0$.

Notice that $\lambda$ is an immersion, it does not create any joint-edges. Therefore, $\lambda$ has a sequence of natural lifts $\lambda_i$s with respect to the sequence of immersions $\phi_i$s, where $\lambda_i:H\to K_i$ with $\lambda_i(h)=e_ne_{n-1}\cdots e_{i+1}h_i$ and the images of all vertices and all other edges unchanged. Especially, $\lambda_n(h)=h_n$, which means that $h$ is a segment (not a subdivision) of $\lambda$. In summary, we have the following commutative diagram of morphisms.

\begin{center}
\begin{tikzpicture}

\node (v1) at (-2,1) {$H$};
\node (v2) at (2.5,-1.5) {$G=K_0$};
\node (v3) at (0.5,-1.5) {$K_{1}$};
\node (v4) at (-1,-1.5) {$K_{2}$};
\node (v5) at (-2.5,-1.5) {$\cdots$};
\draw [-latex] (v1) edge node[sloped, above] {$\lambda$}(v2);
\draw [-latex] (v3) edge node[sloped, below] {$\phi_1$}(v2);
\draw [-latex] (v4) edge node[sloped, below] {$\phi_2$}(v3);
\draw [-latex] (v5) edge node[sloped, below] {$\phi_3$}(v4);
\node (v6) at (-4,-1.5) {$K_n$};
\draw [-latex] (v6) edge node[sloped, below] {$\phi_n$} (v5);
\draw [-latex] (v1) edge node[sloped, above] {$\lambda_1$}(v3);
\draw  [-latex](v1) edge node[sloped, above] {$\lambda_2$}(v4);
\draw [-latex] (v1) edge node[sloped, above] {$\lambda_n$}(v6);
\end{tikzpicture}
\end{center}

Evidently, $\lambda_n:H\to K_n$ is an immersion with $k$ subdivisions, which, by the induction hypothesis, implies that $H$ is an immersion-minor of $K_n$. By its construction, $K_n$ is an immersion-minor of $G$, which implies that $H$ is an immersion-minor of $G$.
\end{proof}

A morphism $\lambda$ is called \textbf{strong} if for any subdivision $e$ of $\lambda$, all internal vertices of $\lambda(e)$ are null-vertices of $\lambda$. In other words, an immersion is strong if all its subdividing-vertices are null-vertices (with no pre-images). It is not difficult to see that strong immersions are closed under composition.

The following notion is parallel to that of a strong-immersion-minor in ordinary graph theory.
\begin{defn}
A causal-net $H$ is a \textbf{strong-immersion-minor} of $G$, if there is a strong immersion $\lambda:H\to G$.
\end{defn}
 The immersion in Example \ref{immersion} is not strong, because the subdividing-vertex $v$ is not a null-vertex (with $w$ as a pre-image). An immersion is called a \textbf{weak immersion} when it is not a strong one.
\begin{ex}\label{strong-immersion}
We show an example of a strong-immersion-minor with $\lambda(h_1)=e_3e_1, \lambda(h_2)=e_4e_2$ and $\lambda(h)=e$. The unique subdividing-vertex $v$ is a null-vertex.
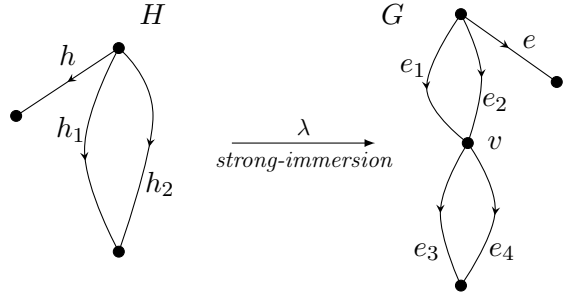
\begin{center}
\begin{tikzpicture}[scale=0.9]
\node (v1) at (-0.5,2) {};
\node (v2) at (-0.5,-1) {};
\draw [fill](v1) circle [radius=0.08];\draw [fill](v2) circle [radius=0.08];
\draw  plot[smooth, tension=.7] coordinates {(v1) (-1,0.5) (v2)}[postaction={decorate, decoration={markings,mark=at position .55 with {\arrow[black]{stealth}}}}];
\draw  plot[smooth, tension=.7] coordinates {(-0.5,2) (0,1) (-0.5,-1)}[postaction={decorate, decoration={markings,mark=at position .5 with {\arrow[black]{stealth}}}}];
\node (v3) at (4.5,2.5) {};
\node (v4) at (4.6,0.6) {};
\node (v5) at (4.5,-1.5) {};
\draw [fill](v3) circle [radius=0.08];
\draw [fill](v4) circle [radius=0.08];
\draw [fill](v5) circle [radius=0.08];
\draw  plot[smooth, tension=.7] coordinates {(v3) (4,1.4) (v4)}[postaction={decorate, decoration={markings,mark=at position .55 with {\arrow[black]{stealth}}}}];
\draw  plot[smooth, tension=.7] coordinates {(v4) (4.2,-0.4) (v5)}[postaction={decorate, decoration={markings,mark=at position .5 with {\arrow[black]{stealth}}}}];
\draw  plot[smooth, tension=.7] coordinates {(v3) (4.8,1.6) (4.6,0.6)}[postaction={decorate, decoration={markings,mark=at position .5 with {\arrow[black]{stealth}}}}];
\draw  plot[smooth, tension=.7] coordinates {(4.6,0.6) (5,-0.4) (v5)}[postaction={decorate, decoration={markings,mark=at position .5 with {\arrow[black]{stealth}}}}];
\node (v6) at (1,0.6) {};
\node (v7) at (3.4,0.6) {};
\draw   [-latex](v6) edge node[sloped,scale=0.8, below] {strong-immersion}node[sloped,scale=0.8, above] {$\lambda$} (v7);
\node at (-1.2,0.8) {$h_1$};
\node at (0.1,0) {$h_2$};
\node at (3.8,1.7) {$e_1$};
\node at (5,1.2) {$e_2$};
\node at (4,-1) {$e_3$};
\node at (5.1,-1) {$e_4$};
\node at (0,2.5) {$H$};
\node at (3.5,2.5) {$G$};
\draw [fill](-2,1) circle [radius=0.08];
\draw(-0.5,2)--(-2,1)[postaction={decorate, decoration={markings,mark=at position .5 with {\arrow[black]{stealth}}}}];
\draw(4.5,2.5)--(5.9,1.5)[postaction={decorate, decoration={markings,mark=at position .5 with {\arrow[black]{stealth}}}}];
\draw [fill](5.9,1.5) circle [radius=0.08];
\node at (5.5,2.1) {$e$};
\node at (-1.25,1.9) {$h$};
\node at (5,0.6) {$v$};
\end{tikzpicture}
\end{center}
\end{ex}

We will end this subsection by listing several corollaries of Theorem \ref{coar-inclu}.

\begin{cor}
Any edge-disjoint morphism $\lambda$ can be uniquely represented as a composition of  a coarse-graining $\lambda_{cg}$ and an immersion $\lambda_{im}$.
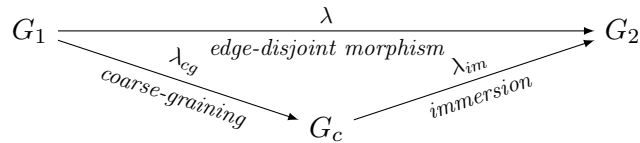
\begin{center}
\begin{tikzpicture}[scale=1.3]
\node (v1) at (-2,0.5) {$G_1$};
\node (v2) at (4,0.5) {$G_2$};
\draw[-latex]  (v1) edge node[sloped,scale=0.8, below] {edge-disjoint morphism} node[sloped,scale=0.8, above] {$\lambda$}(v2);
\node (v4) at (1,-0.5) {$G_c$};
\draw[-latex]  (v1) edge node[sloped,scale=0.8, above] {$\lambda_{cg}$}node[sloped,scale=0.8, below] {coarse-graining} (v4);
\draw[-latex]  (v4) edge node[sloped,scale=0.8, above] {$\lambda_{im}$}node[sloped,scale=0.8, below] {immersion}(v2);
\end{tikzpicture}
\end{center}
\end{cor}
\begin{proof}
By Theorem \ref{coar-inclu}, $\lambda$ can be uniquely decomposed as a composition of a coarse-graining $\lambda_{cg}$ and an inclusion $\lambda_{in}$. Since $\lambda$ is edge-disjoint, then $\lambda_{in}$ is edge-disjoint, which means that $\lambda_{in}$ is an immersion.
\end{proof}

\begin{cor}
Any strong morphism $\lambda$ can be uniquely represented as a composition of  a coarse-graining $\lambda_{cg}$ and a strong inclusion $\lambda_{str-in}$.
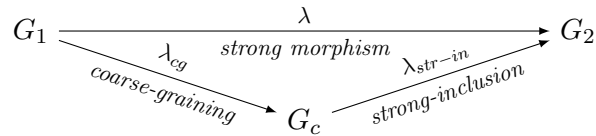
\begin{center}
\begin{tikzpicture}[scale=1.2]
\node (v1) at (-2,0.5) {$G_1$};
\node (v2) at (4,0.5) {$G_2$};
\draw[-latex]  (v1) edge node[sloped,scale=0.8, below] {strong morphism} node[sloped,scale=0.8, above] {$\lambda$}(v2);
\node (v4) at (1,-0.5) {$G_c$};
\draw[-latex]  (v1) edge node[sloped,scale=0.8, above] {$\lambda_{cg}$}node[sloped,scale=0.8, below] {coarse-graining} (v4);
\draw[-latex]  (v4) edge node[sloped,scale=0.8, above] {$\lambda_{str-in}$}node[sloped,scale=0.8, below] {strong-inclusion}(v2);
\end{tikzpicture}
\end{center}
\end{cor}
\begin{proof}
The proof is similar to that of the above corollary. The fact that $\lambda$ is strong implies that $\lambda_{in}$ is strong.
\end{proof}

\begin{cor}
Any strong edge-disjoint morphism $\lambda$ can be uniquely represented as a composition of  a coarse-graining $\lambda_{cg}$ and a strong-immersion $\lambda_{str-im}$.
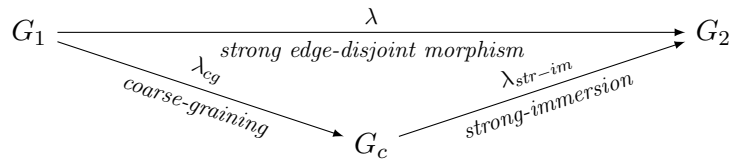
\begin{center}
\begin{tikzpicture}[scale=1.5]
\node (v1) at (-2,0.5) {$G_1$};
\node (v2) at (4,0.5) {$G_2$};
\draw[-latex]  (v1) edge node[sloped,scale=0.8, below] {strong edge-disjoint morphism} node[sloped,scale=0.8, above] {$\lambda$}(v2);
\node (v4) at (1,-0.5) {$G_c$};
\draw[-latex]  (v1) edge node[sloped,scale=0.8, above] {$\lambda_{cg}$}node[sloped,scale=0.8, below] {coarse-graining} (v4);
\draw[-latex]  (v4) edge node[sloped,scale=0.8, above] {$\lambda_{str-im}$}node[sloped,scale=0.8, below] {strong-immersion}(v2);
\end{tikzpicture}
\end{center}
\end{cor}
\begin{proof}
This is a direct consequence of the above two corollaries.
\end{proof}

\subsection{Weak and topological embedding}\label{wte}
In this subsection, we introduce more restricted types of inclusions. We need more terminologies.
Two subdivisions $e_1,e_2$ of morphism $\lambda$ are called \textbf{vertex-disjoint} if their images $\lambda(e_1)$ and $\lambda(e_2)$ have no common \textbf{internal} vertices. We call a morphism \textbf{vertex-disjoint} if any two of its subdivisions are vertex-disjoint. Vertex-disjoint is a more restricted property than edge-disjoint. Both  Example \ref{immersion} and \ref{strong-immersion} are not vertex-disjoint because there $\lambda(h_1)$ and $\lambda(h_2)$ have a common internal vertex $v$.

The following notion is a special type of an immersion, which may be useful somewhere.
\begin{defn}
An inclusion is called a \textbf{weak-embedding} if it is vertex-disjoint.
\end{defn}

As the following example shown, a weak-embedding may create joint-vertices, but none of those joint-vertices is a null-vertex.
\begin{ex}
The following figure shows an example of a weak-embedding, with $\lambda(h_1)=e_1$, $\lambda(h_2)=e_3e_2$. In this example, the weak-embedding $\lambda$ creates a joint-vertex $v$, which is not a null-vertex, hence it is not a strong-immersion.
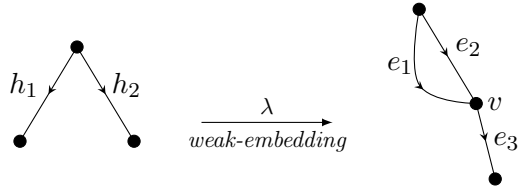
\begin{center}

\begin{tikzpicture}[scale=1]
\node (v1) at (1.25,1.25) {};
\node (v2) at (0.5,0) {};
\node (v3) at (2,0) {};
\node (v4) at (5.75,1.75) {};
\node (v6) at (6.75,-0.5) {};
\node (v5) at (6.5,0.5) {};
\draw [fill](v1) circle [radius=0.08];
\draw [fill](v2) circle [radius=0.08];
\draw [fill](v3) circle [radius=0.08];
\draw [fill](v4) circle [radius=0.08];
\draw [fill](v5) circle [radius=0.08];
\draw [fill](v6) circle [radius=0.08];
\draw  (1.25,1.25) -- (0.5,0) [postaction={decorate, decoration={markings,mark=at position .5 with {\arrow[black]{stealth}}}}];
\draw  (1.25,1.25)-- (2,0)[postaction={decorate, decoration={markings,mark=at position .5 with {\arrow[black]{stealth}}}}];
\draw  (5.75,1.75) -- (6.5,0.5)[postaction={decorate, decoration={markings,mark=at position .5 with {\arrow[black]{stealth}}}}];
\draw  (6.5,0.5)-- (6.75,-0.5)[postaction={decorate, decoration={markings,mark=at position .5 with {\arrow[black]{stealth}}}}];
\draw  plot[smooth, tension=.7] coordinates {(v4) (5.75,0.75) (v5)}[postaction={decorate, decoration={markings,mark=at position .6 with {\arrow[black]{stealth}}}}];
\node (v7) at (2.75,0.25) {};
\node (v8) at (4.75,0.25) {};
\draw[-latex]  (v7) edge node[sloped,scale=0.8, below] {weak-embedding}node[sloped,scale=0.8, above] {$\lambda$}(v8);
\node at (0.55,0.75) {$h_1$};
\node at (1.9,0.75) {$h_2$};
\node at (5.5,1) {$e_1$};
\node at (6.4,1.25) {$e_2$};
\node at (6.9,0) {$e_3$};
\node at (6.75,0.5) {$v$};
\end{tikzpicture}
\end{center}

\end{ex}

A single edge-lifting, as shown in Example \ref{edge-lift}, can be represented in the opposite direction by a weak-embedding. As mentioned above, both  Example \ref{immersion} and \ref{strong-immersion} are not weak-embeddings, but they are immersions, which, by Theorem \ref{immersion-minor},  implies that a composition of a sequence of edge-liftings in general may not be represented in the opposite direction by a weak-embedding. So weak-embeddings are not closed under composition.

  The following notion is more common.
\begin{defn}
An inclusion is called a \textbf{topological embedding} if it is strong and vertex-disjoint.
\end{defn}

By definition, it is not difficult to see that topological embeddings are closed under composition.

A causal-net $H$ is called a \textbf{topological minor} of  $G$, if $G$ can be obtained from $H$ by a sequence of operations of adding an isolated vertex, adding an edge, and subdividing an edge (inserting a middle vertex into an edge). All these three types of graphical operations can be represented in the same direction by morphisms of causal-nets.
\begin{ex}
The following two morphisms $\lambda_1$ and $\lambda_2$ represent the topological minor relation of $H$ and $G$, where $\lambda_1(h_1)=e_2, \lambda_1(h_3)=e_4e_3$ and
$\lambda_2(h_1)=e_3e_2, \lambda_2(h_3)=e_4$. In this example, $\lambda_1$ subdivides $h_3$ and $\lambda_2$ subdivides $h_1$.
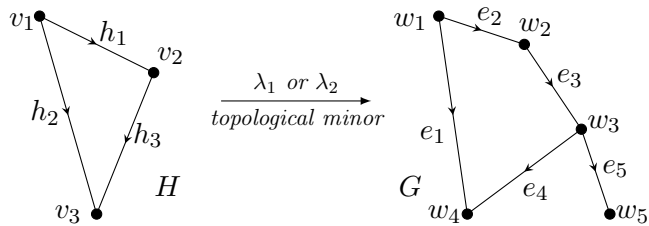
\begin{center}
\begin{tikzpicture}[scale=0.75]
\node (v1) at (-2.5,0.5) {};
\node (v2) at (-1.5,-3) {};
\node (v3) at (-0.5,-0.5) {};
\draw  (-2.5,0.5) -- (-1.5,-3)[postaction={decorate, decoration={markings,mark=at position .5 with {\arrow[black]{stealth}}}}];
\draw  (-2.5,0.5) -- (-0.5,-0.5)[postaction={decorate, decoration={markings,mark=at position .5 with {\arrow[black]{stealth}}}}];
\draw   (-0.5,-0.5) -- (-1.5,-3)[postaction={decorate, decoration={markings,mark=at position .5 with {\arrow[black]{stealth}}}}];

\node (v4) at (4.5,0.5) {};
\node (v5) at (5,-3) {};
\node (v6) at (6,0) {};
\node (v7) at (7,-1.5) {};
\node (v8) at (7.5,-3) {};
\draw  (4.5,0.5)-- (5,-3)[postaction={decorate, decoration={markings,mark=at position .5 with {\arrow[black]{stealth}}}}];
\draw   (4.5,0.5)--(6,0)[postaction={decorate, decoration={markings,mark=at position .5 with {\arrow[black]{stealth}}}}];
\draw  (6,0) -- (7,-1.5)[postaction={decorate, decoration={markings,mark=at position .5 with {\arrow[black]{stealth}}}}];
\draw (7,-1.5)-- (5,-3)[postaction={decorate, decoration={markings,mark=at position .5 with {\arrow[black]{stealth}}}}];
\draw (7,-1.5) -- (7.5,-3)[postaction={decorate, decoration={markings,mark=at position .5 with {\arrow[black]{stealth}}}}];
\node (v9) at (0.5,-1) {};
\node (v10) at (3.5,-1) {};
\draw  [-latex](v9) edge node[sloped,scale=0.8, below] {topological minor}node[sloped,scale=0.8, above] {$\lambda_1$ or $\lambda_2$} (v10);
\node at (-2.8,0.4) {$v_1$};
\node at (-2,-3) {$v_3$};
\node at (-0.2,-0.2) {$v_2$};
\node at (-1.2,0.2) {$h_1$};
\node at (-2.4,-1.2) {$h_2$};
\node at (-0.6,-1.6) {$h_3$};
\node at (4,0.4) {$w_1$};
\node at (4.6,-3) {$w_4$};
\node at (6.2,0.3) {$w_2$};
\node at (7.4,-1.4) {$w_3$};
\node at (7.9,-3) {$w_5$};
\node at (4.4,-1.6) {$e_1$};
\node at (5.4,0.5) {$e_2$};
\node at (6.8,-0.6) {$e_3$};
\node at (6.2,-2.6) {$e_4$};
\node at (7.6,-2.2) {$e_5$};
\draw [fill](v1) circle [radius=0.09];\draw [fill](v2) circle [radius=0.09];\draw [fill](v3) circle [radius=0.09];\draw [fill](v4) circle [radius=0.09];\draw [fill](v5) circle [radius=0.09];\draw [fill](v6) circle [radius=0.09];\draw [fill](v7) circle [radius=0.09];\draw [fill](v8) circle [radius=0.09];
\node at (-0.25,-2.5) {$H$};
\node at (4,-2.5) {$G$};
\end{tikzpicture}
\end{center}
\end{ex}

The graph-theoretic notion of a topological minor can be characterized by the category-theoretic notion of a topological embedding.
\begin{thm}\label{tm}
Let $H$, $G$ be two causal-nets.
$H$ is a topological-minor of $G$ if and only if there is a topological embedding $\lambda:H\to G$.
\end{thm}
\begin{proof}
$(\Rightarrow)$ If $H$ is a topological-minor of $G$, then $G$ can be obtained from $H$ by a sequence of operations of adding an isolated vertex, adding an edge and subdividing an edge, which means that there is a sequence of causal-nets $H=K_0\overset{O_1}{\rightsquigarrow} K_1\overset{O_2}{\rightsquigarrow}\cdots\overset{O_{n-1}}{\rightsquigarrow} K_{n-1}\overset{O_n}{\rightsquigarrow} K_{n}=G$ connected by three types of operations $O_i:K_{i-1}\rightsquigarrow K_{i}$ $(1\leq i\leq n)$. Each $O_i:K_{i-1}\rightsquigarrow K_{i}$ of these three types of operations  can be represented in the same direction by a topological embedding $\lambda_{i}:K_{i-1}\to K_{i}$. Since topological embeddings are closed under composition, then $\lambda=\lambda_1\circ\cdots\circ\lambda_n:H\to G$ is a topological embedding.

$(\Leftarrow)$ We prove this direction by induction on the number $n$ of subdivisions. Let $\lambda:H\to G$ be a topological minor with $n$ subdivisions. If $n=0$, then the image of $H$ is just a sub-causal-net of $G$. In this case, $G$ can be obtained from $H$ by a series of operations of adding an isolated vertex and adding an edge, which means that $H$ is a topological-minor of $G$.

Assume that for $n\leq k$, the existence of a topological embedding implies the topological-minor relation. For $n=k+1$, let $h$ be an subdivision of $\lambda$ and assume $\lambda(h)=e_{n}e_{n-1}\cdots e_0$, then we can construct a sequence of causal-nets $K_0,K_1, \cdots, K_n$ with $K_0=H$ and a sequence of operations of subdividing an edge $H=K_0\overset{O_1}{\rightsquigarrow} K_1\overset{O_2}{\rightsquigarrow}\cdots\overset{O_{n-1}}{\rightsquigarrow} K_{n-1}\overset{O_n}{\rightsquigarrow} K_{n}$, where $O_1$ subdivides  $h$ to path $\overline{h}_1h_0$, $O_2$ subdivides  $\overline{h}_1$ to path $\overline{h}_2h_1$, $\cdots$, $O_n$ subdivides  $\overline{h}_{n-1}$ to path $\overline{h}_nh_{n-1}$.  The sequence of subdivisions can be represented in the same direction by a sequence of topological embeddings $H=K_0\overset{\phi_1}{\to} K_{1}\overset{\phi_{2}}{\to}\cdots\overset{\phi_{n-1}}{\to} K_{n-1}\overset{\phi_n}{\to} K_n$. Evidently, for any $1\leq i\leq n$,  $\phi_i\circ\phi_{i-1}\circ\cdots\circ\phi_1(h)=\overline{h}_ih_{i-1}h_{i-2}\cdots h_1h_0$.

Notice that $\lambda$ is a topological embedding, it is strong and vertex-disjoint. Therefore, $\lambda$ has a sequence of natural extensions $\lambda_i:K_i\to G$ with respect to the sequence of subdivisions $\phi_i:K_{i-1}\to K_i$ such that  $\lambda_1(h_0)=e_0$, $\lambda_1(\overline{h}_1)=e_ne_{n-1}\cdots e_1$; $\lambda_2(h_1)=e_1$, $\lambda_2(\overline{h}_2)=e_ne_{n-1}\cdots e_2$; $\cdots$; $\lambda_n(h_{n-1})=e_{n-1}$, $\lambda_n(\overline{h}_n)=e_n$;  and the images of all other vertices and edges unchanged. Evidently, for any $1\leq i\leq n$, $\lambda_{i}\circ\phi_i\circ\phi_{i-1}\circ\cdots\circ\phi_1(h)=e_ne_{n-1}\cdots e_1e_0$. Especially, $\lambda_n(\overline{h}_nh_{n-1}\cdots h_1h_0)=e_ne_{n-1}\cdots e_1e_0$. In summary, we have the following commutative diagram of morphisms.

\begin{center}
\begin{tikzpicture}
\node (v1) at (-2.5,1) {$H=K_0$};
\node (v2) at (-1,1) {$K_1$};
\node (v3) at (0.5,1) {$\cdots$};
\node (v4) at (2,1) {$K_{n-1}$};
\node (v5) at (3.5,1) {$K_n$};
\draw [-latex] (v1) edge node[sloped, above] {$\phi_1$} (v2);
\draw  [-latex] (v2) edge node[sloped, above] {$\phi_2$} (v3);
\draw [-latex]  (v3) edge node[sloped, above] {$\phi_{n-1}$} (v4);
\draw [-latex]  (v4) edge node[sloped, above] {$\phi_n$}(v5);
\node (v6) at (1.5,-1.5) {$G$};
\draw [-latex]  (v1) edge node[sloped, below] {$\lambda$} (v6);
\draw [-latex]  (v2) edge node[sloped, above] {$\lambda_1$} (v6);
\draw [-latex]  (v4) edge node[sloped, above] {$\lambda_{n-1}$}(v6);
\draw [-latex]  (v5) edge node[sloped, below] {$\lambda_n$} (v6);
\end{tikzpicture}
\end{center}

Evidently, $\lambda_n:K_n\to G$ is a topological embedding with $k$ subdivisions, which, by the induction hypothesis, implies that $K_n$ is a topological-minor of $G$. By the construction of $K_n$, $H$ is an topological-minor of $K_n$, which implies that $H$ is a topological-minor of $G$.
\end{proof}

\subsection{Subdivision and homeomorphism}
Let $\lambda:H\to G$ be a morphism, $v\in V(G)$ and $h\in E(H)$, we say that $h$ is a \textbf{cover} of $v$ or $v$ is \textbf{covered} by $h$,  if $v$ is a vertex of the directed path $\lambda(h)$.

\begin{defn}
A topological embedding $\lambda:H\to G$ is called a \textbf{subdivision} if each non-isolated vertex of $G$ has a cover and all isolated vertices of $G$ are simple.
\end{defn}

The operation of subdividing an edge is invertible, whose inverse operation is called a \textbf{smoothing}.
We call causal-net $G$ a \textbf{subdivision} of causal-net $H$ (or call $H$ a \textbf{smoothing} of G)  if $G$ can be obtained from $H$ by a sequence of operations of subdividing an edge, or equivalently, if  $H$ can be obtained from $G$ by a sequence of smoothings.

Subdividing an edge can be uniquely represented in the same direction by a subdivision, while smoothing can be represented in two ways and in the same direction by path-contractions.

\begin{ex} A subdivision $\lambda$ may have different left inverses $\chi_1,\chi_2$, where $\lambda(v_1)=w_1$, $\lambda(v_2)=w_3$; $\chi_1(w_1)=\chi_1(w_2)=v_1,\chi_1(w_3)=v_2$; $\chi_2(w_1)=v_1,\chi_2(w_2)=\chi_2(w_3)=v_2$. Both $\chi_1$ and $\chi_2$ are path-contractions, representing the same smoothing.
\begin{center}
  \begin{tikzpicture}[scale=0.7]
\node (v1) at (-0.5,1) {};
\node (v2) at (-0.5,-1.5) {};
\node (v3) at (4.5,1.5) {};
\node (v5) at (4.5,-2.5) {};
\node (v4) at (4.5,-0.5) {};
\node (v6) at (9.5,1) {};
\node (v7) at (9.5,-1.5) {};
\node (v8) at (0.5,-0.5) {};
\node (v9) at (3,-0.5) {};
\node (v10) at (6,-0.5) {};
\node (v11) at (8.5,-0.5) {};
\draw [fill](v1) circle [radius=0.075];
\draw [fill](v2) circle [radius=0.075];
\draw [fill](v3) circle [radius=0.075];
\draw [fill](v4) circle [radius=0.075];
\draw [fill](v5) circle [radius=0.075];
\draw [fill](v6) circle [radius=0.075];
\draw [fill](v7) circle [radius=0.075];
\draw   (-0.5,1) -- (-0.5,-1.5)[postaction={decorate, decoration={markings,mark=at position .5 with {\arrow[black]{stealth}}}}];
\draw (4.5,1.5) -- (4.5,-0.5)[postaction={decorate, decoration={markings,mark=at position .5 with {\arrow[black]{stealth}}}}];
\draw (4.5,-0.5) -- (4.5,-2.5)[postaction={decorate, decoration={markings,mark=at position .5 with {\arrow[black]{stealth}}}}];
\draw  (9.5,1)--(9.5,-1.5)[postaction={decorate, decoration={markings,mark=at position .5 with {\arrow[black]{stealth}}}}];
\node at (-1,1) {$v_1$};
\node at (-1,-1.5) {$v_2$};
\node at (4,1.5) {$w_1$};
\node at (5,-0.5) {$w_2$};
\node at (4,-2.5) {$w_3$};
\node at (10,1) {$v_1$};
\node at (10,-1.5) {$v_2$};
\draw [-latex]  (v8) edge  node[sloped,scale=0.8, below] {subdivision}node[sloped,scale=0.8, above] {$\lambda$} (v9);
\draw [-latex]  (v10) edge  node[sloped,scale=0.8, below] {smoothing}node[sloped,scale=0.8, above] {$\chi_1,\chi_2$} (v11);
\end{tikzpicture}
\end{center}
\end{ex}

\begin{thm}\label{subd}
Let $H$, $G$ be two causal-nets. $G$ is a subdivision of $H$ if and only if there is a subdivision $\lambda:H\to G$.
\end{thm}
\begin{proof}
$(\Rightarrow)$ If $G$ is a subdivision of $H$, then $G$ can be obtained from $H$ by a sequence of operations of subdividing an edge, that is, there is a sequence of causal-nets $H=K_0\overset{O_1}{\rightsquigarrow} K_1\overset{O_2}{\rightsquigarrow}\cdots\overset{O_{n-1}}{\rightsquigarrow} K_{n-1}\overset{O_n}{\rightsquigarrow} K_{n}=G$ connected by edge-subdivisions $O_i:K_{i-1}\rightsquigarrow K_{i}$ $(1\leq i\leq n)$. Each $O_i:K_{i-1}\rightsquigarrow K_{i}$ can be represented in the same direction by a subdivision $\lambda_{i}:K_{i-1}\to K_{i}$. Since subdivisions are closed under composition, then $\lambda=\lambda_1\circ\cdots\circ\lambda_n:H\to G$ is a subdivision.

$(\Leftarrow)$ We prove this direction by induction on the number $n$ of subdivisions. Let $\lambda:H\to G$ be a subdivision with $n$ subdivisions. If $n=0$, then by the definition of a subdivision, we see that $\lambda$ is an isomorphism.

Assume that for $n\leq k$, the existence of a subdivision implies the subdivision relation. For $n=k+1$, let $h$ be an subdivision of $\lambda$ and assume $\lambda(h)=e_{n}e_{n-1}\cdots e_0$, then, just as the proof of Theorem \ref{tm}, we can construct a sequence of causal-nets $K_0,K_1, \cdots, K_n$ with $K_0=H$ and a sequence of operations of subdividing an edge $H=K_0\overset{O_1}{\rightsquigarrow} K_1\overset{O_2}{\rightsquigarrow}\cdots\overset{O_{n-1}}{\rightsquigarrow} K_{n-1}\overset{O_n}{\rightsquigarrow} K_{n}$, where $O_1$ subdivides  $h$ to path $\overline{h}_1h_0$, $O_2$ subdivides  $\overline{h}_1$ to path $\overline{h}_2h_1$, $\cdots$, $O_n$ subdivides  $\overline{h}_{n-1}$ to path $\overline{h}_nh_{n-1}$.  The sequence of operations can be represented in the same direction by a sequence of subdivisions $H=K_0\overset{\phi_1}{\to} K_{1}\overset{\phi_{2}}{\to}\cdots\overset{\phi_{n-1}}{\to} K_{n-1}\overset{\phi_n}{\to} K_n$. Evidently, for any $1\leq i\leq n$,  $\phi_i\circ\phi_{i-1}\circ\cdots\circ\phi_1(h)=\overline{h}_ih_{i-1}h_{i-2}\cdots h_1h_0$.

Since $\lambda$ is a subdivision, then $\lambda$ has a sequence of natural extensions $\lambda_i:K_i\to G$ with respect to the sequence of subdivisions $\phi_i:K_{i-1}\to K_i$ such that  $\lambda_1(h_0)=e_0$, $\lambda_1(\overline{h}_1)=e_ne_{n-1}\cdots e_1$; $\lambda_2(h_1)=e_1$, $\lambda_2(\overline{h}_2)=e_ne_{n-1}\cdots e_2$; $\cdots$; $\lambda_n(h_{n-1})=e_{n-1}$, $\lambda_n(\overline{h}_n)=e_n$;  and the images of all other vertices and edges unchanged. Evidently, for any $1\leq i\leq n$, $\lambda_{i}\circ\phi_i\circ\phi_{i-1}\circ\cdots\circ\phi_1(h)=e_ne_{n-1}\cdots e_1e_0$. Especially, $\lambda_n(\overline{h}_nh_{n-1}\cdots h_1h_0)=e_ne_{n-1}\cdots e_1e_0$. In summary, we have the following commutative diagram of morphisms.

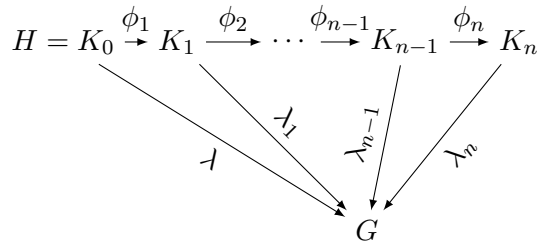
\begin{center}
\begin{tikzpicture}[scale=1]
\node (v1) at (-2.5,1) {$H=K_0$};
\node (v2) at (-1,1) {$K_1$};
\node (v3) at (0.5,1) {$\cdots$};
\node (v4) at (2,1) {$K_{n-1}$};
\node (v5) at (3.5,1) {$K_n$};
\draw [-latex] (v1) edge node[sloped, above] {$\phi_1$} (v2);
\draw  [-latex] (v2) edge node[sloped, above] {$\phi_2$} (v3);
\draw [-latex]  (v3) edge node[sloped, above] {$\phi_{n-1}$} (v4);
\draw [-latex]  (v4) edge node[sloped, above] {$\phi_n$}(v5);
\node (v6) at (1.5,-1.5) {$G$};
\draw [-latex]  (v1) edge node[sloped, below] {$\lambda$} (v6);
\draw [-latex]  (v2) edge node[sloped, above] {$\lambda_1$} (v6);
\draw [-latex]  (v4) edge node[sloped, above] {$\lambda_{n-1}$}(v6);
\draw [-latex]  (v5) edge node[sloped, below] {$\lambda_n$} (v6);
\end{tikzpicture}
\end{center}

Evidently, $\lambda_n:K_n\to G$ is a subdivision with $k$ subdivisions, which, by the induction hypothesis, implies that $G$ is a subdivision of $K_n$. By the construction of $K_n$, it is a subdivision of $H$, which implies that $G$ is a subdivision of $H$.
\end{proof}

If there is a subdivision $\lambda:H\to G$, then the geometric realization $|\lambda|:|H|\to |G|$ of $\lambda$  must be a homeomorphism (topological isomorphism). But conversely, if the geometric realizations $|H|$ and $|G|$ are homeomorphic,  in general there may not be a subdivision between $H$ and $G$. So we introduce the following notion.

\begin{defn}
Two causal-nets $G_1$ and $G_2$ are called \textbf{homeomorphic}, if there exist a causal-net $K$ and two subdivisions $\lambda_1:G_1\to K$ and $\lambda_2:G_2\to K$, or equivalently, if there exist a causal-net $L$ and two subdivisions $\kappa_1:L\to G_1$ and $\kappa_2:L\to G_2$.
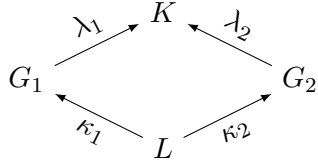
\begin{center}
\begin{tikzpicture}[scale=0.9]
\node (v1) at (-2,0) {$G_1$};
\node (v2) at (0,1) {$K$};
\node (v3) at (2,0) {$G_2$};
\node (v4) at (0,-1) {$L$};
\draw [-latex] (v1) edge node[sloped,scale=1, above] {$\lambda_1$}(v2);
\draw [-latex] (v3) edge node[sloped,scale=1, above] {$\lambda_2$}(v2);
\draw  [-latex](v4) edge node[sloped,scale=1, below] {$\kappa_1$}(v1);
\draw [-latex] (v4) edge node[sloped,scale=1, below] {$\kappa_2$} (v3);
\end{tikzpicture}
\end{center}
\end{defn}
In this definition, the causal-nets $K$ and $L$ are called a \textbf{common subdivision} and a \textbf{common smoothing} of $G_1$ and $G_2$, respectively.
We call a zig-zag of subdivisions a \textbf{homeomorphism}.

\begin{prop}
Two causal-nets are homeomorphic if and only if there is a homeomorphism between them.
\end{prop}
The relation of homeomorphism is an equivalence relation among causal-nets.

\subsection{Embedding as sub-causal-net}
In this subsection, we consider the simplest type of inclusions.

\begin{defn}
A morphism is called an \textbf{embedding} if it is an inclusion with no subdivisions.
\end{defn}

It is straightforward to verify the following result.
\begin{thm}\label{sub}
Let $H$, $G$ be two causal-nets. Then the following conditions are equivalent:

$(1)$ $H$ is a sub-causal-net of $G$;

$(2)$ there is an embedding $\lambda:H\to G$;

$(3)$ $H$ can be obtained from $G$ by a sequence of operations of deleting an edge, deleting an isolated vertex;

$(4)$ $G$ can be obtained from $H$ by a sequence of operations of adding an edge, adding an isolated vertex.
\end{thm}
The operation of adding an edge can be represented by the second type of fundamental morphism (see Section \ref{Relation}), and the operation of adding an isolated vertex can be represented by the third type of fundamental morphism. Clearly, any embedding is a composition of the second and third types of fundamental morphisms.

The following theorem is of particular importance for a category-theoretic reformulation of the Baez construction of spin network states \cite{[B96]}.
\begin{thm}\label{sce}
Any morphism $\lambda:G_1\to G_2$ can be uniquely represented as a composition of a subdivision $\lambda_s$, a coarse-graining $\lambda_{cg}$ and an embedding $\lambda_\iota$.
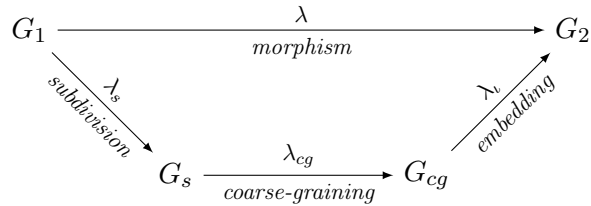
\begin{center}
\begin{tikzpicture}[scale=1.1]
\node (v1) at (-2.5,0.5) {$G_1$};
\node (v2) at (4,0.5) {$G_2$};
\draw[-latex]  (v1) edge  node[sloped,scale=0.8, above] {$\lambda$}node[sloped,scale=0.8, below] {morphism}(v2);
\node (v3) at (-0.75,-1.25) {$G_s$};
\node (v4) at (2.25,-1.25) {$G_{cg}$};
\draw[-latex]  (v1) edge node[sloped,scale=0.8, above] {$\lambda_s$}node[sloped,scale=0.8, below] {subdivision} (v3);
\draw[-latex]  (v3) edge node[sloped,scale=0.8, above] {$\lambda_{cg}$}node[sloped,scale=0.8, below] {coarse-graining}(v4);
\draw [-latex] (v4) edge  node[sloped,scale=0.8, above] {$\lambda_\iota$} node[sloped,scale=0.8, below] {embedding}(v2);
\end{tikzpicture}
\end{center}
\end{thm}

\begin{proof}
The subdivision $G_s$ is obtained from $G_1$ by applying the operation that replacing each subdivision $e\in Subd(\lambda)$ by a copy of its image $\lambda(e)$, which can be represented by the subdivision $\lambda_s:G_1\to G_s$ according to Theorem \ref{subd}. Then $\lambda$ can be naturally represented as  a composition of $\lambda_s$ and a morphism $\lambda':G_s\to G_2$ without subdivisions, that is, we have the following commutative diagram

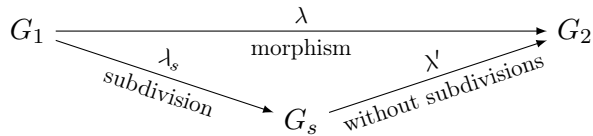
\begin{center}
\begin{tikzpicture}[scale=1.2]
\node (v1) at (-2,0.5) {$G_1$};
\node (v2) at (4,0.5) {$G_2$};
\draw[-latex]  (v1) edge node[sloped,scale=0.8, below] {morphism} node[sloped,scale=0.8, above] {$\lambda$}(v2);
\node (v4) at (1,-0.5) {$G_s$};
\draw[-latex]  (v1) edge node[sloped,scale=0.8, above] {$\lambda_s$}node[sloped,scale=0.8, below] {subdivision} (v4);
\draw[-latex]  (v4) edge node[sloped,scale=0.8, above] {$\lambda'$}node[sloped,scale=0.8, below] {without subdivisions}(v2);
\end{tikzpicture}
\end{center}

By Theorem \ref{coar-inclu}, $\lambda'$ can be decomposed as a composition of a coarse-graining $\lambda_{cg}:G_s\to G_{cg}$ and an inclusion $\lambda_\iota:G_{cg}\to G_2$. Since $\lambda'$ has no subdivisions, $\lambda_\iota$ has no subdivisions, which means that $\lambda_\iota$ is an embedding. That is, we get the following commutative diagram

\begin{center}
\begin{tikzpicture}[scale=1.2]
\node (v1) at (-2,0.5) {$G_s$};
\node (v2) at (4,0.5) {$G_2$};
\draw[-latex]  (v1) edge node[sloped,scale=0.8, below] {without subdivisions} node[sloped,scale=0.8, above] {$\lambda'$}(v2);
\node (v4) at (1,-0.5) {$G_{cg}$};
\draw[-latex]  (v1) edge node[sloped,scale=0.8, above] {$\lambda_{cg}$}node[sloped,scale=0.8, below] {coarse-graining} (v4);
\draw[-latex]  (v4) edge node[sloped,scale=0.8, above] {$\lambda_\iota$}node[sloped,scale=0.8, below] {embedding}(v2);
\end{tikzpicture}
\end{center}

Combining the above two commutative diagrams, we complete the proof.
\end{proof}

\begin{thm}\label{fund}
A morphism is a coarse-graining if and only if it is a composition of fundamental coarse-grainings.
\end{thm}

\begin{proof}
$(\Leftarrow)$ Fundamental coarse-grainings are coarse-grainings whose compositions are coarse-grainings.

$(\Rightarrow)$ By Theorem \ref{v-e}, any coarse-graining is a composition of a vertex-coarse-graining and an edge-coarse-graining. By Theorem \ref{c-m}, any vertex-coarse-graining is a composition of a contraction and a merging. By Theorem \ref{contraction}, any contraction is a composition of simple contractions. Then the proof is reduced to the following three facts that any simple contraction is a composition of contracting an edge and a series of coarse-graining of two parallel edges, that any edge-coarse-graining is a composition of coarse-graining of two parallel edges, and that any merging is a composition of merging two vertices, which are almost obvious.
\end{proof}

The following result is one main result of this paper, called the \textbf{fundamental theorem of $\mathbf{Cau}$}, which is a consequence of Theorem \ref{subd}, Theorem \ref{sub}, Theorem \ref{sce} and Theorem \ref{fund}.
\begin{thm}\label{com}
Any morphism of causal-nets is a composition of six types of fundamental morphisms.
\end{thm}

In the following, we list several results, which may be helpful in understanding the combinatorial structure of $\mathbf{Cau}$.
\begin{cor}\label{inclu}
Any inclusion $\lambda:G_1\to G_2$ can be uniquely represented as a composition of a subdivision $\lambda_s$, a fusion $\lambda_{fu}$ and an embedding $\lambda_\iota$.
\begin{center}
\begin{tikzpicture}[scale=1]
\node (v1) at (-2.5,0.5) {$G_1$};
\node (v2) at (4,0.5) {$G_2$};
\draw[-latex]  (v1) edge  node[sloped,scale=0.8, above] {$\lambda$}node[sloped,scale=0.8, below] {inclusion}(v2);
\node (v3) at (-0.75,-1.25) {$G_s$};
\node (v4) at (2.25,-1.25) {$G_{fu}$};
\draw[-latex]  (v1) edge node[sloped,scale=0.8, above] {$\lambda_s$}node[sloped,scale=0.8, below] {subdivision} (v3);
\draw[-latex]  (v3) edge node[sloped,scale=0.8, above] {$\lambda_{fu}$}node[sloped,scale=0.8, below] {fusion}(v4);
\draw [-latex] (v4) edge  node[sloped,scale=0.8, above] {$\lambda_\iota$} node[sloped,scale=0.8, below] {embedding}(v2);
\end{tikzpicture}
\end{center}
\end{cor}

\begin{proof}
By Theorem \ref{sce}, $\lambda$ can be uniquely decomposed as a composition of a subdivision $\lambda_s$, a coarse-graining $\lambda_{cg}$ and an embedding $\lambda_{\iota}$. Since $\lambda$ is an inclusion, it has no contractions, therefore the coarse-graining $\lambda_{cg}$ has no contractions, which means that it is in fact a fusion.
\end{proof}

\begin{cor}\label{im}
Any immersion $\lambda:G_1\to G_2$ can be uniquely represented as a composition of a subdivision $\lambda_s$, a merging $\lambda_{m}$ and an embedding $\lambda_\iota$.
\begin{center}
\begin{tikzpicture}[scale=1]
\node (v1) at (-2.5,0.5) {$G_1$};
\node (v2) at (4,0.5) {$G_2$};
\draw[-latex]  (v1) edge  node[sloped,scale=0.8, above] {$\lambda$}node[sloped,scale=0.8, below] {immersion}(v2);
\node (v3) at (-0.75,-1.25) {$G_s$};
\node (v4) at (2.25,-1.25) {$G_{m}$};
\draw[-latex]  (v1) edge node[sloped,scale=0.8, above] {$\lambda_s$}node[sloped,scale=0.8, below] {subdivision} (v3);
\draw[-latex]  (v3) edge node[sloped,scale=0.8, above] {$\lambda_{m}$}node[sloped,scale=0.8, below] {merging}(v4);
\draw [-latex] (v4) edge  node[sloped,scale=0.8, above] {$\lambda_\iota$} node[sloped,scale=0.8, below] {embedding}(v2);
\end{tikzpicture}
\end{center}
\end{cor}

\begin{proof}
By Corollary \ref{inclu} and Corollary \ref{fusion-decomposition}, $\lambda$ can be uniquely decomposed as a composition of a subdivision $\lambda_s$, a merging $\lambda_m$, an edge-coarse-graining $\lambda_{\varepsilon}$ and an embedding $\lambda_{\iota}$. Notice that $\lambda$ is an immersion, which is edge-disjoint, so the edge-coarse-graining $\lambda_{\varepsilon}$ must be trivial.
\end{proof}

\begin{cor}
Any topological embedding $\lambda:G_1\to G_2$ can be uniquely represented as a composition of a subdivision $\lambda_s$ and an embedding $\lambda_\iota$.
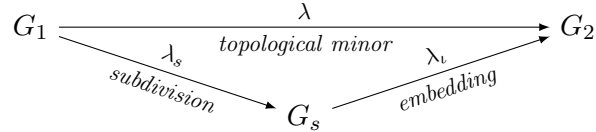
\begin{center}
\begin{tikzpicture}[scale=1.2]
\node (v1) at (-2,0.5) {$G_1$};
\node (v2) at (4,0.5) {$G_2$};
\draw[-latex]  (v1) edge node[sloped,scale=0.8, below] {topological minor} node[sloped,scale=0.8, above] {$\lambda$}(v2);
\node (v4) at (1,-0.5) {$G_s$};
\draw[-latex]  (v1) edge node[sloped,scale=0.8, above] {$\lambda_s$}node[sloped,scale=0.8, below] {subdivision} (v4);
\draw[-latex]  (v4) edge node[sloped,scale=0.8, above] {$\lambda_\iota$}node[sloped,scale=0.8, below] {embedding}(v2);
\end{tikzpicture}
\end{center}
\end{cor}

\begin{proof}
By Corollary \ref{im},  $\lambda$ can be uniquely decomposed as a composition of a subdivision $\lambda_s$, a merging $\lambda_m$ and an embedding $\lambda_{\iota}$. Notice that $\lambda$ is a topological embedding, it is strong and vertex-disjoint, therefore the merging $\lambda_m$ must be trivial.
\end{proof}

We depict the relations among various types of inclusions as follows.\\
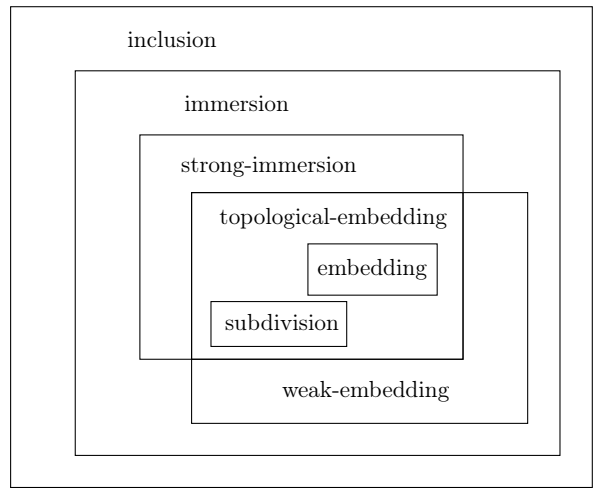
\begin{center}
\begin{tikzpicture}[scale=0.85]
\draw  (-0.4,0.8) rectangle (1.6,0);
\node [scale=0.8] at (0.6,0.4) {embedding};
\draw  (-2.2,1.6) node (v2) {} rectangle (2,-1) node (v1) {};

\node [scale=0.8] at (0,1.2) {topological-embedding};
\draw  (-3,2.5) rectangle (v1);
\node [scale=0.8]at (-1,2) {strong-immersion};
\draw  (v2) rectangle (3,-2);
\node [scale=0.8]at (0.5,-1.5) {weak-embedding};
\draw  (-4,3.5) rectangle (3.5,-2.5);
\node [scale=0.8]at (-1.5,3) {immersion};
\draw  (-5,4.5) rectangle (4,-3);
\node [scale=0.8]at (-2.5,4) {inclusion};
\draw  (-1.9,-0.1) rectangle (0.2,-0.8);
\node [scale=0.8] at (-0.8,-0.42) {subdivision};
\end{tikzpicture}
\end{center}

\section{Generalized minor theory}
This section is devoted to a category-theoretic understanding of various kinds of generalized minors in $\mathbf{Cau}$.

\subsection{Categorical minor theory}
In this subsection,  we introduce a formal and unified framework for theories of minors in any small category, where the notions of a minor category and a gauged minor category serve as the pillars of this framework.

Let  $\mathcal{C}$ be an arbitrary small category and we are free to fix it in this subsection.
\begin{defn}\label{mp}
A \textbf{minor pair} $\{\mathcal{Q},\mathcal{E}\}$  of $\mathcal{C}$ consists of a composition-closed class $\mathcal{Q}$ of epimorphisms and a composition-closed class $\mathcal{E}$ of monomorphisms in $\mathcal{C}$, such that
$\mathcal{Q}\cap\mathcal{E}$ contains the class of isomorphisms in $\mathcal{C}$.
\end{defn}

A morphism in $\mathcal{Q}$ is called a \textbf{$\mathcal{Q}$-quotient}, and a morphism in $\mathcal{E}$ is called an \textbf{$\mathcal{E}$-embedding}.
\begin{ex}\label{e1}
Take $\mathcal{C}=\mathbf{Cau}$, there are several common minor pairs. We can fix the class $\mathcal{Q}$ of epimorphisms to be that of isomorphisms, and take the class $\mathcal{E}$ of monomorphisms to be that of immersions, strong-immersions, topological embeddings, and embeddings of causal-nets, respectively.
\end{ex}

\begin{ex}\label{e2}
Take $\mathcal{C}=\mathbf{Cau}$, $\mathcal{Q}=\{contractions\}$ and  $\mathcal{E}=\{embeddings\  of \ causal\text{-}nets\}$.
\end{ex}

The following notion, called a \textbf{categorical minor}, aims to formalize the operational type definitions of common minors.

\begin{defn}
Let $G_1, G_2$ be two objects of $\mathcal{C}$. A \textbf{${\{\mathcal{Q},\mathcal{E}\}}$-minor}, denoted by $T:G_1\rightsquigarrow G_2$, from source $G_1$ to target $G_2$ is a finite ordered zig-zag of $\mathcal{Q}$-quotients and $\mathcal{E}$-embeddings, that is, an alternating sequence $[K_0;\alpha_0;K_1;\alpha_1;\cdots;\alpha_{n-1};K_n]$ of objects $K_i$ and morphisms $\alpha_i\in \mathcal{Q}\cup\mathcal{E}$, such that $(1)$ $G_1=K_0$; $(2)$ $G_2=K_n$; $(3)$ if $\alpha_i$ is a backward morphism, that is, of the form $\alpha_i:K_{i+1}\to K_i$, then it must be a $\mathcal{Q}$-quotient,  and if $\alpha_i$ is a forward morphism, that is, of the form $\alpha_i:K_{i}\to K_{i+1}$, then it must be  an $\mathcal{E}$-embedding.
\end{defn}
$G_1$ and $G_2$ are called the \textbf{source} and the \textbf{target} of $T:G_1\rightsquigarrow G_2$, respectively. Two $\{\mathcal{Q},\mathcal{E}\}$-minors are called \textbf{parallel} if they have the same source and target.

Each $\alpha_i$ is called a \textbf{factor} of $T=[K_0;\alpha_0; K_1; \cdots; \alpha_{n-1}; K_n]$, and the number $n$ of factors is called the \textbf{length} of $T$. For simplicity, we also represent $T$ as $\alpha_0\dashv \alpha_1\dashv\cdots\dashv \alpha_n$. A successive pair $\alpha_i\dashv\alpha_{i+1}$ of factors is called a \textbf{composable pair} if the two factors are composable, that is, both of them are either forward or backward; otherwise, it is called a \textbf{defect} of $T$. A factor $\alpha_i$ is called an \textbf{iso-factor}, if $\alpha_i$ is an isomorphism. A minor is called an \textbf{iso-minor} if all its factors are iso-factors. Clearly, the source and target of an iso-minor are isomorphic.

\begin{ex}\label{m}
The following shows an example of a minor with length $6$, source $K_0$, target $K_6$. The factors $q_1,q_4, q_6$ are quotients and $\iota_2,\iota_3, \iota_5$ are embeddings. $\iota_2\dashv\iota_3$ is a composable pair; $q_1\dashv\iota_2$, $\iota_3\dashv q_4$, $q_4\dashv\iota_5$ and $\iota_5\dashv q_6$ are  defects.
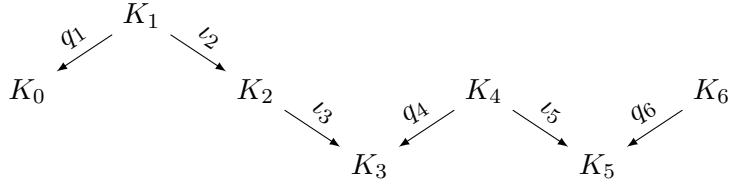
\begin{center}
\begin{tikzpicture}
\node (v2) at (-1.5,0) {$K_0$};
\node (v1) at (0,1) {$K_1$};
\node (v3) at (1.5,0) {$K_2$};
\node (v4) at (3,-1) {$K_3$};
\node (v5) at (4.5,0) {$K_4$};
\node (v6) at (6,-1) {$K_5$};
\node (v7) at (7.5,0) {$K_6$};
\draw [-latex] (v1) edge node[sloped,scale=1, above] {$q_1$}(v2);
\draw [-latex] (v1) edge node[sloped,scale=1, above] {$\iota_2$}(v3);
\draw [-latex] (v3) edge node[sloped,scale=1, above] {$\iota_3$}(v4);
\draw  [-latex](v5) edge node[sloped,scale=1, above] {$q_4$}(v4);
\draw  [-latex](v5) edge node[sloped,scale=1, above] {$\iota_5$}(v6);
\draw [-latex] (v7) edge node[sloped,scale=1, above] {$q_6$}(v6);
\end{tikzpicture}
\end{center}
\end{ex}

We say that $[K_0;\alpha_0;K_1;\alpha_1;\cdots;\alpha_{n-1};K_n]$ and $[L_0;\beta_0;K_1;\beta_1;\cdots;\beta_{m-1};L_m]$ have the same \textbf{type} if they have the same length (that is, $n=m$) and for each $i$, $\alpha_i$ and $\beta_i$ have the same direction.
For simplicity, we will freely call a $\{\mathcal{Q},\mathcal{E}\}$-minor, a $\mathcal{Q}$-quotient and an $\mathcal{E}$-embedding simply a categorical-minor (or just minor), a quotient and an embedding, respectively, when the context is clear.

Two minors $T_1:G_1\rightsquigarrow G_2$ and $T_2:G_2\rightsquigarrow G_3$ are composable, whose composition $T_3=T_2\circ T_1:G_1\rightsquigarrow G_3$ is defined as the juxtaposition of sequences. That is, if $T_1=[G_1;\alpha_0;K_1;\alpha_1;\cdots;\alpha_{n};G_2]$, $T_2=[G_2;\beta_0;L_1;\beta_1;\cdots;\beta_{m};G_3]$, then $T_2\circ T_1=[G_1;\alpha_0;K_1;\alpha_1;\cdots;\alpha_{n};G_2;\beta_0;L_1;\beta_1;\cdots;\beta_{m};G_3]$. Clearly, the composition of minors is associative. A minor without a factor, that is, with length zero is called \textbf{trivial}.

We define a category $\mathbf{Min}$, called the \textbf{minor category} of $\mathcal{C}$ with respect to $\{\mathcal{Q},\mathcal{E}\}$,  as follows. The objects and morphisms of $\mathbf{Min}$ are objects of $\mathcal{C}$ and  $\{\mathcal{Q},\mathcal{E}\}$-minors, respectively. The composition of $\mathbf{Min}$ is defined as the composition of $\{\mathcal{Q},\mathcal{E}\}$-minors, which is clearly associative, and for each object $G$, its identity morphism is the trivial minor $[G]$.

\begin{defn}
 Let $H$ and $G$ be two objects of $\mathcal{C}$. $H$ is called a \textbf{$\{\mathcal{Q},\mathcal{E}\}$-minor} of $G$ if there is a $\{\mathcal{Q},\mathcal{E}\}$-minor $T:H\rightsquigarrow G$.
\end{defn}

In all these cases of Example \ref{e1}, we will see that, by Theorem \ref{immersion-minor}, \ref{tm} and \ref{sub}, the theory of $\{\mathcal{Q},\mathcal{E}\}$-minor categorifies the usual theory of immersion-minor, strong-immersion-minor, topological-minor and sub-causal-net relation, respectively. In the case of Example \ref{e2}, we will see that,  by Theorem \ref{contraction}, the theory of $\{\mathcal{Q},\mathcal{E}\}$-minor categorifies the usual theory of minor.

There are two types of defects. A defect is called a \textbf{quotient-sub} in $\{\mathcal{Q},\mathcal{E}\}$, if it is of the form $q\dashv i$ with $q\in \mathcal{Q}$ and $i\in \mathcal{E}$. A defect is called a \textbf{sub-quotient} in $\{\mathcal{Q},\mathcal{E}\}$, if it is of the form $i\dashv q$ with $q\in \mathcal{Q}$ and $i\in \mathcal{E}$.

\begin{defn}
A quotient-sub  $q_1\dashv\iota_1$ and a sub-quotient $\iota_2\dashv q_2$ are called \textbf{dual} to each other if they satisfy $\iota_2\circ q_1=q_2\circ \iota_1$, that is, the following diagram commutes.
\end{defn}
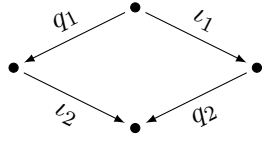
\begin{center}
\begin{tikzpicture}[scale=0.8]
\node (v1) at (0,0) {};
\node (v2) at (2,-1) {};
\node (v3) at (4,0) {};
\node (v4) at (2,1) {};
\draw [fill](v1) circle [radius=0.075];
\draw [fill](v2) circle [radius=0.075];
\draw [fill](v3) circle [radius=0.075];
\draw [fill](v4) circle [radius=0.075];
\draw [-latex] (v1) edge node[sloped,scale=1, below] {$\iota_2$}(v2);
\draw [-latex] (v3) edge node[sloped,scale=1, below] {$q_2$}(v2);
\draw [-latex] (v4) edge node[sloped,scale=1, above] {$q_1$}(v1);
\draw [-latex] (v4) edge node[sloped,scale=1, above] {$\iota_1$} (v3);
\end{tikzpicture}
\end{center}

In this case, we call them \textbf{dual defects} or simply \textbf{duals} of each other. Clearly, a dual of a quotient-sub is a sub-quotient  and a dual of a sub-quotient is a quotient-sub.

\begin{ex}\label{def}
Not all quotient-subs  have dual defects. If we take $\mathcal{C}=\mathbf{Cau}$, $\mathcal{Q}$ to be the class of all contractions and $\mathcal{E}$ to be the class of all embeddings, then we can see that the following defect does not have a dual defect. We can prove this by contradiction. If there is a dual defect $i'\dashv q'$, then on one hand, we must have $\iota'\circ q(e_3)\rightarrow \iota'\circ q(e_1)$, where $\rightarrow$ denotes the reachable order; and on the other hand, we must have $q'\circ\iota(e_1)\rightarrow q'\circ\iota(e_3)$. By the dual relation, we have $\iota'\circ q(e_1)=q'\circ\iota(e_1)$ and  $\iota'\circ q(e_3)=q'\circ\iota(e_3)$, which leads to a contradiction.

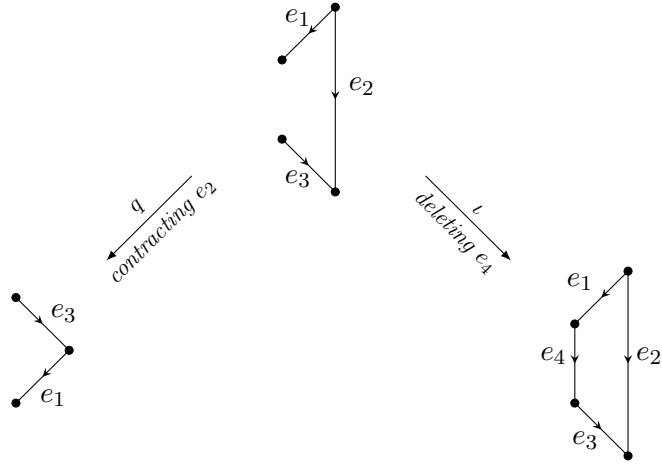
\begin{center}
\begin{tikzpicture}[scale=0.7]

\node (v1) at (0.5,2) {};
\node (v2) at (-0.5,1) {};
\node (v3) at (-0.5,-0.5) {};
\node (v4) at (0.5,-1.5) {};
\draw [fill](v1) circle [radius=0.075];
\draw [fill](v2) circle [radius=0.075];
\draw [fill](v3) circle [radius=0.075];
\draw [fill](v4) circle [radius=0.075];
\draw  (0.5,2) -- (-0.5,1)[postaction={decorate, decoration={markings,mark=at position .5 with {\arrow[black]{stealth}}}}];
\draw (0.5,2) -- (0.5,-1.5) [postaction={decorate, decoration={markings,mark=at position .5 with {\arrow[black]{stealth}}}}];
\draw (-0.5,-0.5)-- (0.5,-1.5) [postaction={decorate, decoration={markings,mark=at position .5 with {\arrow[black]{stealth}}}}];

\node (v5) at (6,-3) {};
\node (v6) at (5,-4) {};
\node (v7) at (5,-5.5) {};
\node (v8) at (6,-6.5) {};
\draw [fill](v5) circle [radius=0.075];
\draw [fill](v6) circle [radius=0.075];
\draw [fill](v7) circle [radius=0.075];
\draw [fill](v8) circle [radius=0.075];
\draw (6,-3)-- (5,-4)[postaction={decorate, decoration={markings,mark=at position .5 with {\arrow[black]{stealth}}}}];
\draw (5,-4)-- (5,-5.5)[postaction={decorate, decoration={markings,mark=at position .5 with {\arrow[black]{stealth}}}}];
\draw  (5,-5.5) --(6,-6.5)[postaction={decorate, decoration={markings,mark=at position .5 with {\arrow[black]{stealth}}}}];
\draw  (6,-3)-- (6,-6.5)[postaction={decorate, decoration={markings,mark=at position .5 with {\arrow[black]{stealth}}}}];
\node (v10) at (-4.5,-4.5) {};
\node (v11) at (-5.5,-5.5) {};
\node (v9) at (-5.5,-3.5) {};
\draw [fill](v9) circle [radius=0.075];
\draw [fill](v10) circle [radius=0.075];
\draw [fill](v11) circle [radius=0.075];
\draw (-5.5,-3.5)-- (-4.5,-4.5)[postaction={decorate, decoration={markings,mark=at position .5 with {\arrow[black]{stealth}}}}];
\draw  (-4.5,-4.5) -- (-5.5,-5.5)[postaction={decorate, decoration={markings,mark=at position .5 with {\arrow[black]{stealth}}}}];
\node (v12) at (-2,-1) {};
\node (v13) at (-4,-3) {};
\draw [-latex] (v12) edge node[sloped,scale=0.8, below] {contracting $e_2$} node[sloped,scale=0.8, above] { $q$}(v13);
\node (v14) at (2,-1) {};
\node (v15) at (4,-3) {};
\draw [-latex] (v14) edge node[sloped,scale=0.8, below] {deleting $e_4$} node[sloped,scale=0.8, above] { $\iota$}(v15);

\node at (1,0.5) {$e_2$};
\node at (-0.2,1.8) {$e_1$};
\node at (-0.2,-1.25) {$e_3$};
\node at (5.1,-3.2) {$e_1$};
\node at (6.4,-4.6) {$e_2$};
\node at (5.2,-6.2) {$e_3$};
\node at (4.6,-4.6) {$e_4$};
\node at (-4.8,-5.4) {$e_1$};
\node at (-4.6,-3.8) {$e_3$};
\end{tikzpicture}
\end{center}
\end{ex}
There are seven types of fundamental operations on minors.
\begin{defn}
A \textbf{dual transformation} of a minor $T$ is an operation of  replacing  a defect of $T$ with its dual defect.
\end{defn}

\begin{ex}
The following shows a  dual transformation of the minor in Example \ref{m}, where the defect $\iota_5\dashv q_6$ is replaced by its dual defect $q'_5 \dashv\iota'_6$.

\begin{center}
\begin{tikzpicture}
\node (v2) at (-1.5,0) {$K_0$};
\node (v1) at (0,1) {$K_1$};
\node (v3) at (1.5,0) {$K_2$};
\node (v4) at (3,-1) {$K_3$};
\node (v5) at (4.5,0) {$K_4$};
\node (v6) at (6,1) {$K'_5$};
\node (v8) at (6,-1) {$K_5$};
\node (v7) at (7.5,0) {$K_6$};
\draw [-latex] (v1) edge node[sloped,scale=1, above] {$q_1$}(v2);
\draw [-latex] (v1) edge node[sloped,scale=1, above] {$\iota_2$}(v3);
\draw [-latex] (v3) edge node[sloped,scale=1, above] {$\iota_3$}(v4);
\draw  [-latex](v5) edge node[sloped,scale=1, above] {$q_4$}(v4);
\draw  [latex-](v5) edge node[sloped,scale=1, above] {$q'_5$}(v6);
\draw [latex-] (v7) edge node[sloped,scale=1, above] {$\iota'_6$}(v6);
\draw  [-latex,dotted](v5) edge node[sloped,scale=1, above] {$\iota_5$}(v8);
\draw [-latex,dotted] (v7) edge node[sloped,scale=1, above] {$q_6$}(v8);
\end{tikzpicture}
\end{center}
\end{ex}

\begin{defn}
A \textbf{flip transformation} of $T=[K_0;\alpha_0; K_1; \cdots; \alpha_{n-1}; K_n]$ is an operation of replacing an iso-factor $\alpha_i$ of $T$ with its inverse $\alpha^{-1}_i$.
\end{defn}

Both dual and flip transformations change the types of minors, but they do not change the sources, targets and lengths of minors.

There are also operations that change the lengths of minors. Given a minor $$[K_0;\alpha_0;\cdots;\alpha_{n-1};K_n]$$ with length $n$, if it  has a composable pair $\alpha_i\dashv\alpha_{i+1}$,  then after composing $\alpha_i$ and $\alpha_{i+1}$, we get a parallel minor with length $n-1$ $$[K_0;\alpha_0;\cdots;K_i;\beta;K_{i+2};\cdots;\alpha_{n-1};K_n],$$ where $\beta=\alpha_i\circ \alpha_{i+1}$ or $\alpha_{i+1}\circ \alpha_i$. This operation of composing a composable pair  is called a \textbf{composable reduction}. The inverse operation that decomposes a factor into a composable pair is called a \textbf{factor decomposition}.

\begin{ex}
The following shows a composable reduction of the minor in Example \ref{m}.
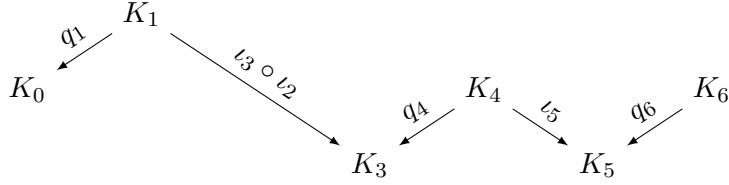
\begin{center}
\begin{tikzpicture}
\node (v2) at (-1.5,0) {$K_0$};
\node (v1) at (0,1) {$K_1$};

\node (v4) at (3,-1) {$K_3$};
\node (v5) at (4.5,0) {$K_4$};
\node (v6) at (6,-1) {$K_5$};
\node (v7) at (7.5,0) {$K_6$};
\draw [-latex] (v1) edge node[sloped,scale=1, above] {$q_1$}(v2);
\draw [-latex] (v5) edge node[sloped,scale=1, above] {$q_4$}(v4);
\draw  [-latex](v1) edge node[sloped,scale=1, above] {$\iota_3\circ\iota_2$}(v4);
\draw  [-latex](v5) edge node[sloped,scale=1, above] {$\iota_5$}(v6);
\draw [-latex] (v7) edge node[sloped,scale=1, above] {$q_6$}(v6);
\end{tikzpicture}
\end{center}
\end{ex}

If $K_i=K_{i+1}$ and the factor $\alpha_i:K_i\to K_{i+1}$ is an automorphism, then we can remove $\alpha_i$ and get a parallel minor with length $n-1$
$$[K_0;\alpha_0;\cdots;K_i;\alpha_{i+1};K_{i+2};\cdots;\alpha_{n-1};K_n].$$
The operation of removing an automorphism  is called an \textbf{automorphism deletion}. The inverse operation of inserting an automorphism is called an \textbf{automorphism insertion}.

There are operations that change sources and targets.
\begin{defn}
Two minors $T=[K_0;\alpha_0;\cdots;\alpha_{n-1};K_{n}]$ and $T'=[L_0;\beta_0;\cdots;\beta_{n-1};L_{n}]$ of the same type are called \textbf{isomorphic} if there exists one isomorphism $\phi_i:K_i\to L_i$ for each $i=0,1,2,\cdots,n$, such that $\beta_i\circ \phi_i=\phi_{i+1}\circ\alpha_i$, if both $\alpha_i$ and $\beta_i$ are embeddings; or $\beta_i\circ \phi_{i+1}=\phi_{i}\circ\alpha_i$, if both $\alpha_i$ and $\beta_i$  are quotients.
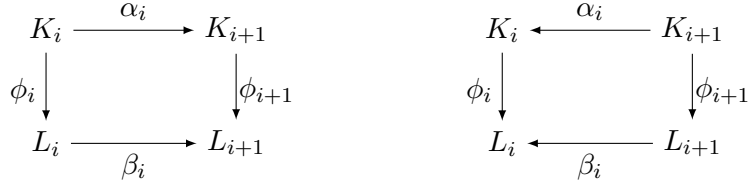
\begin{center}
\begin{tikzpicture}
\node (v1) at (-1.5,1) {$K_i$};
\node (v2) at (1,1) {$K_{i+1}$};
\node (v4) at (1,-0.5) {$L_{i+1}$};
\node (v3) at (-1.5,-0.5) {$L_{i}$};
\draw [-latex] (v1) edge node[sloped,scale=1, above] {$\alpha_i$}(v2);
\draw[-latex]  (v3) edge node[sloped,scale=1, below] {$\beta_i$} (v4);
\draw[-latex]  (v2) edge (v4);
\draw[-latex]  (v1) edge (v3);
\node at (-1.8,0.2) {$\phi_i$};
\node at (1.4,0.2) {$\phi_{i+1}$};
\node (v6) at (4.5,1) {$K_i$};
\node (v5) at (7,1) {$K_{i+1}$};
\node (v8) at (4.5,-0.5) {$L_i$};
\node (v7) at (7,-0.5) {$L_{i+1}$};
\draw [-latex] (v5) edge node[sloped,scale=1, above] {$\alpha_i$}(v6);
\draw[-latex]  (v7) edge node[sloped,scale=1, below] {$\beta_i$} (v8);
\draw [-latex] (v6) edge (v8);
\draw  [-latex](v5) edge (v7);
\node at (4.2,0.2) {$\phi_i$};
\node at (7.4,0.2) {$\phi_{i+1}$};
\end{tikzpicture}
\end{center}
\end{defn}
Isomorphic minors have isomorphic sources and targets, and have the same length.
An \textbf{isomorphic transformation} of a minor $T$ is an operation of replacing $T$ with an isomorphic  minor $T'$.

We list properties of fundamental operations as follows.
\begin{center}
\begin{tikzpicture}[scale=0.9]

\draw  (-3.5,2) rectangle (1.5,1) node (v1) {};
\draw  (-3.5,1) rectangle (1.5,0) node (v4) {};
\draw  (-3.5,0) rectangle (1.5,-1) node (v7) {};
\draw  (-3.5,-1) rectangle (1.5,-2) node (v10) {};
\draw  (-3.5,-2) rectangle (1.5,-3) node (v13) {};
\draw  (1.5,2) rectangle (4.5,1) node (v2) {};
\draw  (4.5,2) rectangle (7.5,1) node (v3) {};
\draw  (7.5,2) rectangle (10.5,1);
\draw  (-3.5,-3) rectangle (1.5,-4);
\draw  (-3.5,-4) rectangle (1.5,-5) node (v18) {};
\draw  (v1) rectangle (4.5,0) node (v5) {};
\draw  (v2) rectangle (7.5,0) node (v6) {};
\draw  (v3) rectangle (10.5,0);
\draw  (v4) rectangle (4.5,-1) node (v8) {};
\draw  (v5) rectangle (7.5,-1) node (v9) {};
\draw  (v6) rectangle (10.5,-1);
\draw  (v7) rectangle (4.5,-2) node (v11) {};
\draw  (v8) rectangle (7.5,-2) node (v12) {};
\draw  (v9) rectangle (10.5,-2);
\draw  (v10) rectangle (4.5,-3) node (v14) {};
\draw  (v11) rectangle (7.5,-3) node (v15) {};
\draw  (v12) rectangle (10.5,-3);

\node at (-1,0.5) {dual transformation};
\node at (-1,-0.5) {flip transformation};
\node at (-1,-1.5) {composable reduction};
\node at (-1,-3.5) {automorphism deletion};
\node at (-1,-4.5) {automorphism insertion};
\node at (-1,-2.5) {factor decomposition};

\draw  (-3.5,-5) rectangle (1.5,-6);
\node at (-1,-5.5) {isomorphic transformation};
\node at (3,1.5) {type};
\node at (6,1.5) {source $\&$ target};
\node at (9,1.5) {length};
\node at (3,0.5) {$\surd$};
\node at (6,0.5) {$\times$};
\node at (9,0.5) {$\times$};
\node at (3,-0.5) {$\surd$};
\node at (6,-0.5) {$\times$};
\node at (9,-0.5) {$\times$};
\node at (3,-1.5) {$\surd$};
\node at (6,-1.5) {$\times$};
\node at (9,-1.5) {$\surd$};
\node at (3,-2.5) {$\surd$};
\node at (6,-2.5) {$\times$};
\node at (9,-2.5) {$\surd$};
\node at (3,-3.5) {$\surd$};
\node at (6,-3.5) {$\times$};
\node at (9,-3.5) {$\surd$};
\node at (3,-4.5) {$\surd$};
\node at (6,-4.5) {$\times$};
\node at (9,-4.5) {$\surd$};
\node at (3,-5.5) {$\times$};
\node at (6,-5.5) {$\surd/\times$};
\node at (9,-5.5) {$\times$};
\draw  (v13) rectangle (4.5,-4) node (v16) {};
\draw  (v14) rectangle (7.5,-4) node (v17) {};
\draw  (v15) rectangle (10.5,-4);
\draw  (1.5,-4) rectangle (4.5,-5) node (v19) {};
\draw  (v16) rectangle (7.5,-5) node (v20) {};
\draw  (v17) rectangle (10.5,-5);
\draw  (v18) rectangle (4.5,-6);
\draw  (v19) rectangle (7.5,-6);
\draw  (v20) rectangle (10.5,-6);
\node at (4,-6.5) {change=$\surd$,\ \  unchange=$\times$};
\end{tikzpicture}
\end{center}

\begin{defn}
Two minors are called \textbf{gauge equivalent} if they are connected by a series of fundamental operations, which are dual transformations, flip transformations,  composable reductions, factor decompositions, automorphism reductions, automorphism insertions and isomorphic transformations.
\end{defn}
In other words, the gauge equivalence relation is the equivalence relation among minors generated by the seven fundamental operations.

In the following lemmas, we show some properties of the gauge equivalence relation.
\begin{lem}\label{t1}
Any iso-minor is gauge equivalent to a trivial minor.
\end{lem}
\begin{proof}
Suppose $T=[K_0;\alpha_0;K_1;\alpha_1;\cdots;\alpha_{n-1};K_{n}]$ is an iso-minor, we will show that it is gauge equivalent to the trivial minor $[K_0]$. Since all $\alpha_i$s $(i=0,1,\cdots,n-1)$ are isomorphisms, then after applying flip transformations and composable reductions, $T$ can be transformed to a length $1$ minor $[K_0;\beta; K_n]$ with $\beta$ being a forward isomorphism, which, by the following commutative diagram, is isomorphic to the minor $[K_0;Id_{K_0};K_0]$. After an automorphism deletion, $[K_0;Id_{K_0};K_0]$  is transformed to the trivial minor $[K_0]$, thus $T$ is gauge equivalent to the trivial minor $[K_0]$. Similarly, we can show that $T$ is also gauge equivalent to the trivial minor $[K_n]$.
\begin{center}
\begin{tikzpicture}
\node (v1) at (-1.5,1) {$K_0$};
\node (v2) at (1,1) {$K_n$};
\node (v4) at (1,-0.5) {$K_0$};
\node (v3) at (-1.5,-0.5) {$K_0$};
\draw [-latex] (v1) edge node[sloped,scale=1, above] {$\beta$}(v2);
\draw[-latex]  (v3) edge node[sloped,scale=1, below] {$Id_{K_0}$} (v4);
\draw[-latex]  (v2) edge (v4);
\draw[-latex]  (v1) edge (v3);
\node at (-1.85,0.2) {$Id_{K_0}$};
\node at (1.4,0.2) {$\beta^{-1}$};
\end{tikzpicture}
\end{center}
\end{proof}

\begin{lem}\label{t2}
Gauge equivalent minors have isomorphic sources and targets.
\end{lem}
\begin{proof}
Suppose $T_1:K_1\rightsquigarrow L_1$ and $T_2:K_2\rightsquigarrow L_2$ are gauge equivalent, then there exist a series of minors $S_i$ and fundamental operations $O_j$, such that $T_1=S_0\overset{O_1}{\Rightarrow} S_1\overset{O_2}{\Rightarrow}\cdots\overset{O_{n-1}}{\Rightarrow} S_{n-1}\overset{O_n}{\Rightarrow} S_{n}=T_2$. If all $O_j$s are not isomorphic transformations, then, by the fact that all other fundamental operations do not change sources and targets, we must have that $K_1=K_2$ and $L_1=L_2$. If some of these $O_j$s are isomorphic transformations, then, by the fact that isomorphic minors have isomorphic sources and targets, we must have that $K_1$ is isomorphic to $K_2$ and $L_1$ is isomorphic to $L_2$.
\end{proof}

\begin{lem}\label{t3}
Let $E_1:K\rightsquigarrow H$, $T:H\rightsquigarrow G$, $E_2: G\to L$ be three composable minors. If $E_1$ and $E_2$ are iso-minors, then $E_2\circ T\circ E_1$ is gauge equivalent to $T$.
\end{lem}
\begin{proof}
We only need to show that $E_2\circ T$ is gauge equivalent to $T$, and the fact that $T\circ E_1$ is gauge equivalent to $T$ can be proved similarly. Without loss of generality, we assume that $E_2$ is a length one iso-minor $[G;\beta; L]$ with $\beta:G\to L$ being forward. By the following commutative diagram, we see that $[G;\beta; L]\circ T$ is isomorphic to $[G;Id_G; G]\circ T$. By applying an isomorphism deletion, we have that $[G;Id_G; G]\circ T$ is gauge equivalent to $T$. Therefore, $E_2\circ T$ is gauge equivalent to $T$.
\begin{center}
\begin{tikzpicture}
\node (v1) at (-1.5,1) {$G$};
\node (v2) at (1,1) {$L$};
\node (v4) at (1,-0.5) {$G$};
\node (v3) at (-1.5,-0.5) {$G$};
\draw [-latex] (v1) edge node[sloped,scale=1, above] {$\beta$}(v2);
\draw[-latex]  (v3) edge node[sloped,scale=1, below] {$Id_{G}$} (v4);
\draw[-latex]  (v2) edge (v4);
\draw[-latex]  (v1) edge (v3);
\node at (-1.85,0.2) {$Id_{G}$};
\node at (1.4,0.2) {$\beta^{-1}$};
\end{tikzpicture}
\end{center}
\end{proof}

\begin{lem}
There is a well-defined composition of gauge equivalence classes of minors.
\end{lem}

\begin{proof}
For any two minors $T_1:H\rightsquigarrow G$ and $T_2: K\rightsquigarrow L$, if there is an isomorphism $\alpha$ of $G$ and $K$, then we define the composition of their gauge equivalence classes $\overline{T_1:H\rightsquigarrow G}$ and $\overline{T_2: K\rightsquigarrow L}$ to be the gauge equivalence class of the minor $T_2\circ [G;\alpha;K]\circ T_1$. We need to show that the composition of gauge equivalence classes is well-defined. For this, we need to prove that the composition $\overline{T_2\circ [G;\alpha;K]\circ T_1}$ is independent of the choice of $\alpha$. There are several cases.

$(1)$ If $\alpha$ and $\beta$ are two forward isomorphisms, that is, $\alpha:G\to K$ and $\beta:G\to K$, we want to show that the
compositions $T_2\circ [G;\alpha;K]\circ T_1$ and $T_2\circ [G;\beta;K]\circ T_1$ are gauge equivalent. In fact, since both $\beta\circ\alpha^{-1}:K\to K$ and $Id_G:G\to G$ are automorphisms, we have, by applying automorphism insertions, that  $T_2\circ [G;\alpha;K]\circ T_1$ and $T_2\circ [G;\beta;K]\circ T_1$ are gauge equivalent to $T_2\circ [K;\beta\circ\alpha^{-1};K]\circ [G;\alpha;K]\circ T_1$ and $T_2\circ [G;\beta;K]\circ [G;Id_{G};G]\circ T_1$, respectively. By the following commutative diagram, we see that $T_2\circ [K;\beta\circ\alpha^{-1};K]\circ [G;\alpha;K]\circ T_1$ and $T_2\circ [G;\beta;K]\circ [G;Id_{G};G]\circ T_1$ are isomorphic and hence gauge equivalent, therefore, $T_2\circ [G;\alpha;K]\circ T_1$ and $T_2\circ [G;\beta;K]\circ T_1$ are gauge equivalent.
\begin{center}
\begin{tikzpicture}
\node (v1) at (-2,1) {$G$};
\node (v2) at (0.5,1) {$K$};
\node (v3) at (3,1) {$K$};
\node (v4) at (-2,-0.5) {$G$};
\node (v5) at (0.5,-0.5) {$G$};
\node (v6) at (3,-0.5) {$K$};
\draw[-latex]  (v1) edge (v2);
\draw[-latex]  (v2) edge (v3);
\draw[-latex]  (v4) edge (v5);
\draw[-latex]  (v5) edge (v6);
\draw[-latex]  (v1) edge (v4);
\draw[-latex]  (v2) edge (v5);
\draw [-latex] (v3) edge (v6);
\node at (-0.75,1.25) {$\alpha$};
\node at (1.75,1.25) {$\beta\circ\alpha^{-1}$};
\node at (-0.75,-0.75) {$Id_G$};
\node at (1.75,-0.75) {$\beta$};
\node at (-2.32,0.25) {$Id_G$};
\node at (0.9,0.25) {$\alpha^{-1}$};
\node at (3.35,0.25) {$Id_K$};
\end{tikzpicture}
\end{center}

$(2)$ If $\alpha$ and $\beta$ are two backward isomorphisms, the prove is similar to case $(1)$.

$(3)$ If one of $\alpha$ and $\beta$ is backward and the other is forward, then by applying a flip operation on $[G;\alpha;K]$,  we can transform these cases to cases $(1)$ and $(2)$.
\end{proof}

Summarizing above lemmas, we see that the gauge equivalence relation induces a \textbf{regular generalized congruence} \cite{[BBP99]} $\sim$ on the minor category $\mathbf{Min}$, which is an equivalence relation $\sim_o$ on objects of $\mathbf{Min}$ together with an equivalence relation $\sim_m$ on $\{\mathcal{Q},\mathcal{E}\}$-minors (morphisms of $\mathbf{Min}$) such that  $(T_1:K_1\rightsquigarrow L_1)\sim_m (T_2:K_2\rightsquigarrow L_2)$, by Lemma \ref{t2}, implies that  $K_1\sim_o K_2$ and $L_1\sim_o L_2$, where the equivalence relation $\sim_o$ is just the isomorphism relation of objects of $\mathcal{C}$ (also objects of $\mathbf{Min}$) and the equivalence relation $\sim_m$ is the gauge equivalence relation of $\{\mathcal{Q},\mathcal{E}\}$-minors.

The quotient category $\mathbf{Min}/\sim$ is called the \textbf{gauged minor category} of $\mathcal{C}$ with respect to $\{\mathcal{Q},\mathcal{E}\}$ and is denoted by $\textbf{GM}$. The objects of $\textbf{GM}$ are isomorphism classes of objects of $\mathcal{C}$. For any two objects $\overline{K}$ and $\overline{L}$ of $\textbf{GM}$, their hom-set is the set of gauge equivalence classes of minors  $T:K\rightsquigarrow L$. By Lemma \ref{t3}, the identity morphisms in $\textbf{GM}$ are the gauge equivalence classes of iso-minors.

The following notion was introduce in \cite{[K08]}, which naturally generalizes the notions of a poset and a causal-net.
\begin{defn}\label{acy}
An \textbf{acyclic category}  is a small category without non-tivial loops, that is, it satisfies the following two
conditions:  $(1)$ for any object $X$, $Hom(X,X)$ has only one element, the identity morphism; $(2)$  for any two objects $X$ and $Y$, $Hom(X, Y)\neq\emptyset$ and $Hom(Y, X)\neq\emptyset$ imply that $X=Y$.
\end{defn}
An acyclic category is also called a \textbf{loop-free category}. The second condition in the definition means that for any two distinct objects $X$ and $Y$, either $Hom(X,Y)$ or $Hom(Y,X)$ is empty. In an acyclic category, all isomorphisms and endomorphisms are identity morphisms. It is easy to see that all objects of an acyclic category form a poset with respect to the relation that $X\leq Y$ $\Longleftrightarrow$ $Hom(X,Y)\neq \emptyset$, where the first and second conditions imply the reflexivity and anti-symmetric properties of $\leq$, respectively. This poset is called the \textbf{object poset} of the acyclic category.

\begin{ex}
$(1)$ For any causal-net, its path category is an acyclic category.

$(2)$ Any sub-category of an acyclic category is an acyclic category.

$(3)$ If $\mathcal{C}$ is an acyclic category, $F:\mathcal{C}\to \mathcal{D}$ is a functor, then the image of $F$ is an acyclic category.

$(4)$ If $F:\mathcal{C}\to \mathcal{D}$ is a functor between acyclic categories, then it induces a monotonic mapping between the corresponding object posets.
\end{ex}

\begin{defn}
A minor pair $\{\mathcal{Q},\mathcal{E}\}$ is called \textbf{strict} if for any $\{\mathcal{Q},\mathcal{E}\}$-minor $$[K_0;\alpha_0;K_1;\cdots;\alpha_{n-1};K_{n}],$$ the condition that $K_0$ and $K_{n}$ are isomorphic implies that all $\alpha_i$s ($i=0,1,\cdots,n-1$) are isomorphisms, that is, any endo-$\{\mathcal{Q},\mathcal{E}\}$-minor is an iso-minor.
\end{defn}

The following is a fundamental result in categorical minor theory.
\begin{thm}
For any strict minor pair of a small category, the gauged minor category is an acyclic category.
\end{thm}
\begin{proof}
Let $\{\mathcal{Q},\mathcal{E}\}$ be a strict minor pair of small category $\mathcal{C}$. We want to show that its gauged minor category $\mathbf{GM}$ is an acyclic category.
If $T=[K;\alpha_0;K_1;\alpha_1;\cdots;\alpha_{n-1};L]$ is a minor with $K$ isomorphic to $L$, then all $\alpha_i$s ($i=0,1,\cdots,n-1$) are isomorphisms. By Lemma \ref{t1}, $T$ is gauge equivalent to trivial minor $[K]$. This proves that $\mathbf{GM}$ satisfies the first condition in the definition of an acyclic category.

Now we show that $\mathbf{GM}$ satisfies the second condition. Suppose $T_1=[H_1;\alpha_0;K_1;\cdots;\alpha_{n-1};G_1]$ and $T_2=[G_2;\beta_0;L_1;\cdots;\beta_{m-1};H_2]$ are two minors with $H_1$  isomorphic to $H_2$ and $G_1$ isomorphic to $G_2$, which means that there exist isomorphisms $\gamma_1:G_1\to G_2$ and $\gamma_2:H_2\to H_1$, then,  as shown below,  $[H_2;\gamma_2;H_1]\circ T_2\circ[G_1;\gamma_1;G_2]\circ T_1$ is an endo-minor of $H_1$. Since $\{\mathcal{Q},\mathcal{E}\}$ is strict, then all $\alpha_i$s and $\beta_j$s are isomorphisms, which implies that $H_1$ is isomorphic to $G_1$ and $H_2$ is isomorphic to $G_2$. This completes the proof.
\begin{center}

\begin{tikzpicture}[scale=0.9]
\node (v1) at (-2,1.5) {$H_1$};
\node (v2) at (1.5,1.5) {$G_1$};
\node (v3) at (1.5,-0.5) {$G_2$};
\node (v4) at (-2,-0.5) {$H_2$};
\draw [decorate,decoration={zigzag, amplitude=1},-latex] (v1) -- (v2);
\draw[decorate,decoration={zigzag, amplitude=1},-latex]   (v2) -- (v3);
\draw[decorate,decoration={zigzag, amplitude=1},-latex]   (v3) -- (v4);
\draw [decorate,decoration={zigzag, amplitude=1},-latex]  (v4) -- (v1);
\node at (-0.2,1.8) {$T_1$};
\node at (2.7,0.5) {$[G_1;\gamma_1;G_2]$};
\node at (-0.2,-0.82) {$T_2$};
\node at (-3.2,0.5) {$[H_2;\gamma_2; H_1]$};
\end{tikzpicture}
\end{center}

\end{proof}
In case that the  minor pair is strict, the object poset of \textbf{GM} is called a  \textbf{minor poset} of $\mathcal{C}$.

\begin{defn}
A minor is called a \textbf{quotient-sub} if it is gauge equivalent to a minor with the only defect being a quotient-sub. A minor is called a \textbf{sub-quotient} if it is gauge equivalent to a minor with the only defect being a sub-quotient.
\end{defn}
A minor is called \textbf{exact} if it is both a quotient-sub and a sub-quotient. In Example \ref{def}, since each of the three causal-nets has the identity morphism as its only automorphism, it is easy to see that the minor is not gauge equivalent to any sub-quotient, hence not an exact minor.

\subsection{Coarse-graining minor}
In this subsection, we take $\mathcal{C}=\mathbf{Cau}$, $\mathcal{Q}$ to be the class $\mathbf{CG}$ of all coarse-grainings and $\mathcal{E}$ to be the class $\mathbf{EM}$ of all embeddings.  In this case, a $\{\mathbf{CG}, \mathbf{EM}\}$-minor is called a \textbf{coarse-graining minor} or simply \textbf{CG-minor}. Clearly, $\{\mathbf{CG}, \mathbf{EM}\}$ is a strict minor pair.
We list four definitions of CG-minor relations, and show that they are equivalent.

\begin{defn}\label{D1}
$H$ is called a \textbf{CG-minor} of $G$ if there is a CG-minor $T:H\rightsquigarrow G$.
\end{defn}

By Theorem \ref{fund}, any coarse-graining is a composition of fundamental coarse-grainings,  that is, a composition of the fourth, fifth and sixth types of fundamental morphisms, which means that if $H$ is a coarse-graining of $G$, then $H$ can be obtained from $G$ by a sequence of operations of merging two vertices, coarse-graining of two parallel edges and contracting an edge. Similarly, by Theorem \ref{sub}, any embedding is a composition of fundamental embeddings, that is, a composition of the second and third types of fundamental morphisms, which means that if $H$ is a sub-causal-net of $G$, then  $H$ can be obtained from $G$ by a sequence of operations of deleting an isolated vertex and deleting an edge. It is easy to see that  Definition \ref{D1} can be equivalently formulated as follows.

\begin{defn}\label{D2}
$H$ is called a \textbf{CG-minor} of $G$ if $H$ can be obtained from $G$ by a sequence of the following five types of fundamental operations:
$(1)$ deleting an edge; $(2)$ deleting an isolated vertex; $(3)$ contracting an edge; $(4)$ merging two vertices; $(5)$ coarse-graining two parallel edges.
\end{defn}

Since contracting an edge can be viewed as a special case of contracting a multi-edge, and conversely,  contracting a multi-edge is equivalent to a composition of contracting an edge and a series of coarse-graining two parallel edges, then the third type of operation in Definition \ref{D2} can be equivalently replaced by that of contracting a multi-edge.

\begin{defn}
$H$ is called a \textbf{CG-minor} of $G$ if $H$ can be obtained from $G$ by a sequence of the following five types of fundamental operations:
$(1)$ deleting an edge; $(2)$ deleting an isolated vertex; $(3)'$ contracting a multi-edge; $(4)$ merging two vertices; $(5)$ coarse-graining two parallel edges.
\end{defn}

The fourth definition of a CG-minor is as follows.
\begin{defn}\label{D4}
$H$ is called a \textbf{CG-minor} of $G$ if  $H$ is a coarse-graining of a sub-causal-net of $G$, or equivalently, there is a causal-net $K$ with an embedding $\lambda_{\iota}:K\to G$ and a coarse-graining $\lambda_{cg}:K\to H$.
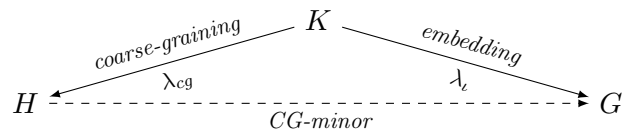
\begin{center}
\begin{tikzpicture}[scale=1.1]
\node (v1) at (0,-1) {$K$};
\node (v2) at (3.5,-2) {$G$};
\node (v3) at (-3.5,-2) {$H$};
\draw [-latex] (v1) edge node[sloped,scale=0.8, below] {$\lambda_{\iota}$}node[sloped,scale=0.8, above] {embedding}  (v2);
\draw [-latex] (v1) edge node[sloped,scale=0.8, below] {$\lambda_{cg}$}node[sloped,scale=0.8, above] {coarse-graining}(v3);
\draw [dashed,-latex] (v3) -- node[sloped,scale=0.8, below] {CG-minor} (v2);
\end{tikzpicture}
\end{center}

\end{defn}

To prove that Definition \ref{D1} is equivalent to Definition \ref{D4}, we need to show that the minor pair $\{\mathbf{CG}, \mathbf{EM}\}$ satisfies the following \textbf{dominating property}: any sub-quotient in  $\{\mathbf{CG}, \mathbf{EM}\}$ has a dual quotient-sub, that is, for any defect $\iota\dashv q$ with $q\in \mathbf{CG}$ and $\iota\in \mathbf{EM}$, there exists a defect  $q'\dashv \iota'$ with $q'\in \mathbf{CG}$ and an $\iota'\in \mathbf{EM}$, such that $q\circ \iota'=\iota\circ q'$. In this case, we say that the class $\mathbf{CG}$ dominates the class $\mathbf{EM}$, which means that the operations of coarse-grainings can always be placed after the operations of embeddings.

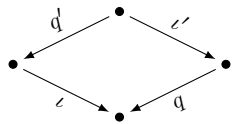
\begin{center}
\begin{tikzpicture}[scale=0.7]
\node (v1) at (0,0) {};
\node (v2) at (2,-1) {};
\node (v3) at (4,0) {};
\node (v4) at (2,1) {};
\draw [fill](v1) circle [radius=0.075];
\draw [fill](v2) circle [radius=0.075];
\draw [fill](v3) circle [radius=0.075];
\draw [fill](v4) circle [radius=0.075];
\draw [-latex] (v1) edge node[sloped,scale=0.8, below] {$\iota$}(v2);
\draw [-latex] (v3) edge node[sloped,scale=0.8, below] {$q$}(v2);
\draw [-latex] (v4) edge node[sloped,scale=0.8, above] {$q'$}(v1);
\draw [-latex] (v4) edge node[sloped,scale=0.8, above] {$\iota'$} (v3);
\end{tikzpicture}
\end{center}

The following lemma shows that when a fundamental embedding and a simple contraction form a sub-quotient, they are commutative in a suitable way. But, as shown in Example \ref{def}, when a fundamental embedding and a simple contraction form a quotient-sub, they are not commutative in general.
\begin{lem}\label{c1}
Let $\lambda_c:K_1\to H$ and $\lambda_\iota:K_2\to H$ be two morphisms of causal-nets. If $\lambda_c$ is a simple contraction, $\lambda_\iota$ is a fundamental embedding, then there exist a causal-net $K_3$ and two morphisms $\lambda'_c: K_3\to K_2$, $\lambda'_\iota:K_3\to K_1$ with $\lambda'_c$ being a simple contraction and $\lambda'_\iota$ being an embedding,  such that $\lambda_c\circ \lambda'_\iota=\lambda_\iota\circ \lambda_c'$.
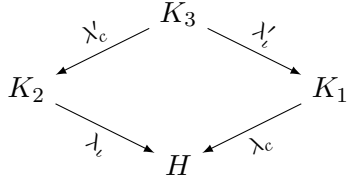
\begin{center}
\begin{tikzpicture}
\node (v1) at (0,0) {$K_2$};
\node (v2) at (2,-1) {$H$};
\node (v3) at (4,0) {$K_1$};
\node (v4) at (2,1) {$K_3$};
\draw [-latex] (v1) edge node[sloped,scale=0.8, below] {$\lambda_{\iota}$}(v2);
\draw [-latex] (v3) edge node[sloped,scale=0.8, below] {$\lambda_{c}$}(v2);
\draw [-latex] (v4) edge node[sloped,scale=0.8, above] {$\lambda_c'$}(v1);
\draw [-latex] (v4) edge node[sloped,scale=0.8, above] {$\lambda_{\iota}'$} (v3);
\end{tikzpicture}
\end{center}
\end{lem}

\begin{proof}
Suppose $\lambda_c$ contracts a multi-edge $\varepsilon=\{e_1, \cdots,e_n\}$ of $K_1$ with source $v_1$ and target $v_2$ and $v=\lambda_c(v_1)=\lambda_c(v_2)$.

$(1)$ If $\lambda_\iota$ deletes an isolated vertex $w$ of $H$, then we have two cases: $w=v$ and $w\neq v$.
In the former case, the multi-edge $\varepsilon$ must be an isolated multi-edges of $K_1$. We define $K_3$ as the sub-causal-net $(K_1-\{e_1, \cdots,e_n\})-\{v_1,v_2\}$, $\lambda_\iota':K_3\to K_1$ as the embedding representing the composition of operations of deleting $\varepsilon$, deleting $v_1$ and $v_2$; $\lambda_c':K_3\to K_2$ as the natural isomorphism.
In the latter case, $\lambda^{-1}(w)$ must be an isolated vertex of $K_1$, hence $\lambda_c^{-1}(w)\neq v_1, \lambda_c^{-1}(w)\neq v_2$. Define $K_3$ as the sub-causal-net $K_1-\{\lambda_c^{-1}(w)\}$, $\lambda_\iota':K_3\to K_1$ as the fundamental embedding of deleting $\lambda_c^{-1}(w)$ and $\lambda_c':K_3\to K_2$ as the simple contraction that contracts $\varepsilon$. In both cases, it is easy to check that $\lambda_c\circ \lambda'_\iota=\lambda_\iota\circ \lambda_c'$.

$(2)$ If $\lambda_\iota$ deletes an edge $e$ of $H$, then $\lambda_c^{-1}(e)$ is also an edge of $K_1$.  Clearly, $\lambda_c^{-1}(e)\not\in \varepsilon$. Define $K_3$  as the sub-causal-net $K_1-\{\lambda_c^{-1}(e)\}$, $\lambda_\iota':K_3\to K_1$ as the fundamental embedding representing the operation of  deleting $\lambda_c^{-1}(e)$ and $\lambda_c':K_3\to K_2$ as the simple contraction that contracts $\varepsilon$. It is easy to check that $\lambda_c\circ \lambda'_\iota=\lambda_\iota\circ \lambda_c'$.
\end{proof}

The following two lemmas shows that  fundamental embeddings and mergings of two vertices are commutative in a suitable way.

\begin{lem}\label{m1}
Let $\lambda_m:K_1\to H$ and $\lambda_\iota:K_2\to H$ be two morphisms of causal-nets. If $\lambda_m$ merges two vertices and $\lambda_\iota$ is a fundamental embedding, then there exist a causal-net $K_3$ and two morphisms $\lambda'_m: K_3\to K_2$, $\lambda'_\iota:K_3\to K_1$ with $\lambda'_m$  either merging two vertices or being an isomorphism and $\lambda'_\iota$ being an embedding, such that $\lambda_m\circ\lambda_\iota'=\lambda_\iota\circ\lambda_m'$.
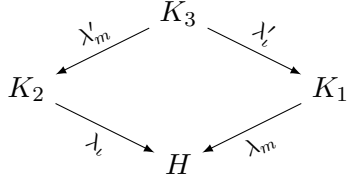
\begin{center}
\begin{tikzpicture}
\node (v1) at (0,0) {$K_2$};
\node (v2) at (2,-1) {$H$};
\node (v3) at (4,0) {$K_1$};
\node (v4) at (2,1) {$K_3$};
\draw [-latex] (v1) edge node[sloped,scale=0.8, below] {$\lambda_{\iota}$}(v2);
\draw [-latex] (v3) edge node[sloped,scale=0.8, below] {$\lambda_{m}$}(v2);
\draw [-latex] (v4) edge node[sloped,scale=0.8, above] {$\lambda_m'$}(v1);
\draw [-latex] (v4) edge node[sloped,scale=0.8, above] {$\lambda_{\iota}'$} (v3);
\end{tikzpicture}
\end{center}
\end{lem}

\begin{proof}
Suppose $\lambda_m$ merges two vertices $v_1, v_2$ of $K_1$ to a vertex $w$ of $H$.

$(1)$ If $\lambda_\iota$ deletes an isolated vertex $u$ of $H$, we have two possible cases:  $u\neq w$ and $u=w$. In the former case, $v=\lambda_m^{-1}(u)$  must be an isolated vertex of $K_1$. We define $K_3=K_1-\{v\}$ that formed by deleting the isolated vertex $v$ from $K_1$. Define $\lambda_\iota':K_3\to K_1$ to be the fundamental embedding that deleting $v$, and $\lambda_m':K_3\to K_2$ be the fundamental morphism merging $v_1$ and $v_2$. In the latter case,  $w$ is an isolated vertex of $H$ implies that both $v_1$ and $v_2$ are isolated vertices of $K_1$. In this case, we define $K_3=K_1-\{v_1,v_2\}$, $\lambda_\iota'$ to be the embedding that deleting $v_1,v_2$ and $\lambda_m':K_3\to K_2$ to be the natural isomorphism.

$(2)$ If $\lambda_\iota$ deletes an edge $h$ of $H$,  we define $K_3=K_1-\{\lambda_\iota^{-1}(h)\}$, $\lambda_\iota':K_3\to K_1$ to be the fundamental embedding that deletes the edge $\lambda_\iota^{-1}(h)$ and $\lambda_m':K_3\to K_2$ to be the fundamental morphism merging $v_1$ and $v_2$.

In all the above cases, it is easy to check that $\lambda_m\circ\lambda_\iota'=\lambda_\iota\circ\lambda_m'$.
\end{proof}

\begin{lem}\label{m2}
Let $\lambda_\iota:H\to L_1$ and $\lambda_m:H\to L_2$ be two morphisms of causal-nets. If $\lambda_\iota$ is a fundamental embedding and $\lambda_m$ merges two vertices, then there exist a causal-net $L_3$ and two morphisms $\lambda'_m: L_1\to L_3$, $\lambda'_\iota:L_2\to L_3$ with $\lambda'_m$ being a fundamental morphism that merges two vertices, such that $\lambda_m'\circ \lambda_\iota=\lambda_\iota'\circ \lambda_m$, and $\lambda_{\iota'}$ being a fundamental embedding.
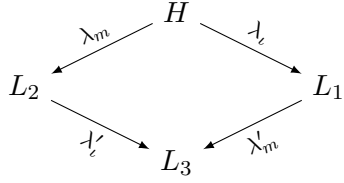
\begin{center}
\begin{tikzpicture}
\node (v1) at (0,0) {$L_2$};
\node (v2) at (2,-1) {$L_3$};
\node (v3) at (4,0) {$L_1$};
\node (v4) at (2,1) {$H$};
\draw [-latex] (v1) edge node[sloped,scale=0.8, below] {$\lambda_{\iota}'$}(v2);
\draw [-latex] (v3) edge node[sloped,scale=0.8, below] {$\lambda_{m}'$}(v2);
\draw [-latex] (v4) edge node[sloped,scale=0.8, above] {$\lambda_m$}(v1);
\draw [-latex] (v4) edge node[sloped,scale=0.8, above] {$\lambda_\iota$} (v3);
\end{tikzpicture}
\end{center}
\end{lem}

\begin{proof}
Suppose $\lambda_m$ merges two vertices $v_1, v_2$ of $H$ to a vertex $w$ of $L_2$.

$(1)$ If $\lambda_\iota$ deletes an isolated vertex $u$ of $L_1$, then $u\neq\lambda_\iota^{-1}(v_1)$, $u\neq\lambda_\iota^{-1}(v_1)$. We define $L_3=L_2\sqcup \{u\}$ that formed by adding an isolated vertex $u$ to $L_2$, $\lambda_m':L_1\to L_3$ to be the morphism merging $\lambda_\iota^{-1}(v_1)$ and $\lambda_\iota^{-1}(v_2)$ to the vertex $w$ and $\lambda_\iota':L_2\to L_3$ to be the fundamental embedding that deleting the isolated vertex $u$.

$(2)$ If $\lambda_\iota$ deletes an edge $e$ of $L_1$,  we define $L_3$ to be the quotient $L_1/\{\lambda_\iota^{-1}(v_1), \lambda_\iota^{-1}(v_2)\}$ that merges $\lambda_\iota^{-1}(v_1)$ and $\lambda_\iota^{-1}(v_2)$, $\lambda_m':L_1\to L_3$ to be the fundamental morphism that merges $\lambda_\iota^{-1}(v_1)$ and $\lambda_\iota^{-1}(v_2)$ and $\lambda_\iota':L_2\to L_3$ to be fundamental embedding that deletes the edge $\lambda_m'(e)$.

In both cases, it is easy to check that $\lambda_m\circ\lambda_\iota'=\lambda_\iota\circ\lambda_m'$.
\end{proof}

The following two lemmas shows that  fundamental embeddings and coarse-grainings of two parallel edges are commutative in a suitable way.

\begin{lem}\label{cos1}
Let $\lambda_{cg}:K_1\to H$ and $\lambda_\iota:K_2\to H$ be two morphisms of causal-nets. If $\lambda_{cg}$ coarse-grainings two parallel edges and $\lambda_\iota$ is a fundamental embedding, then there exist a causal-net $K_3$ and two morphisms $\lambda'_{cg}: K_3\to K_2$, $\lambda'_\iota:K_3\to K_1$ with $\lambda'_{cg}$ either coarse-graining two parallel edges or being an isomorphism and $\lambda'_\iota$ being either a fundamental embedding or a composition of fundamental embeddings, such that $\lambda_{cg}\circ\lambda_\iota'=\lambda_\iota\circ\lambda_{cg}'$.
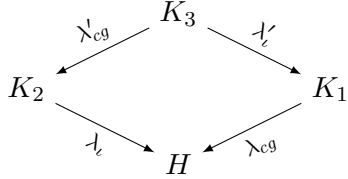
\begin{center}
\begin{tikzpicture}
\node (v1) at (0,0) {$K_2$};
\node (v2) at (2,-1) {$H$};
\node (v3) at (4,0) {$K_1$};
\node (v4) at (2,1) {$K_3$};
\draw [-latex] (v1) edge node[sloped,scale=0.8, below] {$\lambda_{\iota}$}(v2);
\draw [-latex] (v3) edge node[sloped,scale=0.8, below] {$\lambda_{cg}$}(v2);
\draw [-latex] (v4) edge node[sloped,scale=0.8, above] {$\lambda_{cg}'$}(v1);
\draw [-latex] (v4) edge node[sloped,scale=0.8, above] {$\lambda_{\iota}'$} (v3);
\end{tikzpicture}
\end{center}
\end{lem}

\begin{proof}
Suppose $\lambda_{cg}$ coarse-grainings two parallel edges $e_1, e_2$ of $K_1$ to an edge $e$ of $H$.

$(1)$ If $\lambda_\iota$ deletes an isolated vertex $v$ of $H$, then $w=\lambda_{cg}^{-1 }(v)$ must be an isolated vertex of $K_1$.  We define $K_3=K_1-\{w\}$ formed by deleting $w$ from $K_1$,  $\lambda_\iota':K_3\to K_1$ to be the fundamental embedding that deletes $w$ and $\lambda_{cg}':K_3\to K_2$ being the coarse-graining of $e_1,e_2$ to the edge $\lambda_{\iota}^{-1}(e)$.

$(2)$ If $\lambda_\iota$ deletes an edge $h$ of $H$, we have two possible cases: $h=e$ and $h\neq e$. In the former case, we define $K_3=K_1-\{e_1,e_2\}$ formed by deleting $e_1$ and $e_2$ from $K_1$, $\lambda_\iota':K_3\to K_1$ to be the morphism deleting $e_1$ and $e_2$ and $\lambda_{cg}':K_3\to K_2$ to be the natural isomorphism. In the latter case, we define $K_3=K_1-\{\lambda_{\iota}^{-1}(h)\}$ formed by deleting the edge $\lambda_{\iota}^{-1}(h)$, $\lambda_\iota':K_3\to K_1$ to be the fundamental embedding deleting $\lambda_{\iota}^{-1}(h)$ and $\lambda_{cg}:K_3\to K_2$ to be the fundamental morphism that coarse-grainings $e_1$ and $e_2$ to the edge $\lambda^{-1}_\iota(e)$.

In all the above cases, it is easy to check that $\lambda_{cg}\circ\lambda_\iota'=\lambda_\iota\circ\lambda_{cg}'$.
\end{proof}

\begin{lem}\label{cos2}
Let $\lambda_\iota:H\to L_1$ and $\lambda_{cg}:H\to L_2$ be two morphisms of causal-nets. If $\lambda_\iota$ is a fundamental embedding and $\lambda_{cg}$ coarse-grainings two parallel edges, then there exist a causal-net $L_3$ and two morphisms $\lambda'_{cg}: L_1\to L_3$, $\lambda'_\iota:L_2\to L_3$ with $\lambda'_{cg}$ coarse-graining two parallel edges and $\lambda_{\iota}'$ being a fundamental embedding, such that $\lambda_{cg}'\circ \lambda_\iota=\lambda_\iota'\circ \lambda_{cg}$.
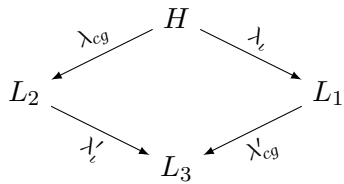
\begin{center}
\begin{tikzpicture}
\node (v1) at (0,0) {$L_2$};
\node (v2) at (2,-1) {$L_3$};
\node (v3) at (4,0) {$L_1$};
\node (v4) at (2,1) {$H$};
\draw [-latex] (v1) edge node[sloped,scale=0.8, below] {$\lambda_{\iota}'$}(v2);
\draw [-latex] (v3) edge node[sloped,scale=0.8, below] {$\lambda_{cg}'$}(v2);
\draw [-latex] (v4) edge node[sloped,scale=0.8, above] {$\lambda_{cg}$}(v1);
\draw [-latex] (v4) edge node[sloped,scale=0.8, above] {$\lambda_\iota$} (v3);
\end{tikzpicture}
\end{center}
\end{lem}

\begin{proof}
Suppose $\lambda_{cg}$ coarse-grainings two parallel edges $e_1, e_2$ of $H$ to an edge  $e$ of $L_2$.

$(1)$ If $\lambda_\iota$ deletes an isolated vertex $u$ of $L_1$, then we define $L_3=L_2\sqcup \{u\}$, $\lambda_{cg}':L_1\to L_3$ to be the fundamental morphism that coarse-grainings the parallel edges $\lambda^{-1}_\iota(e_1), \lambda^{-1}_\iota(e_1)$ to $e$ and $\lambda_{\iota}':L_2\to L_3$ to be the fundamental embedding deleting the isolated vertex $u$.

$(2)$ If $\lambda_\iota$ deletes an edge $h$ of $L_1$, then $h\neq \lambda^{-1}_\iota(e_1)$ and $h\neq \lambda^{-1}_\iota(e_2)$. We define $L_3$ to be the quotient-causal-net of $L_1$ formed by coarse-graining edges $\lambda^{-1}_\iota(e_1)$ and $\lambda^{-1}_\iota(e_2)$ to an edge $h'$, $\lambda_{cg}':L_1\to L_3$ to be the fundamental morphism that coarse-grainings $\lambda^{-1}_\iota(e_1)$, $\lambda^{-1}_\iota(e_2)$ and $\lambda_{\iota}':L_2\to L_3$ to be the fundamental embedding deleting $\lambda_{cg}'(h)$ and with $\lambda_\iota'(e)=h'$.

In both cases, it is easy to check that $\lambda_{cg}'\circ\lambda_\iota=\lambda_\iota'\circ\lambda_{cg}$.
\end{proof}

\begin{thm}
The class $\mathbf{CG}$ dominates the class $\mathbf{EM}$ and Definition \ref{D1} is equivalent to  Definition \ref{D4}.
\end{thm}

\begin{proof}
By Theorem \ref{sub}, Theorem \ref{fund}, Lemma \ref{c1}, Lemma \ref{m1} and Lemma \ref{cos1}, we can see that $\mathbf{CG}$ dominates  $\mathbf{EM}$, which implies that Definition \ref{D1} and  Definition \ref{D4} are equivalent.
\end{proof}

This equivalence of Definition \ref{D1} and  Definition \ref{D4} implies that all CG-minors are quotient-sub and all sub-quotient CG-minor are exact.

We call $H$ an\textbf{ exact CG-minor} of $G$ if there exists an exact CG-minor $T:H\rightsquigarrow G$.
Exact CG-minors can be characterized as sub-coarse-grainings.
\begin{thm}
$H$ is an exact CG-minor of $G$ if and only if  $H$ is a sub-coarse-graining of $G$, or equivalently, there is a causal-net $L$ with an embedding $\lambda_{\iota}:K\to G$ and a coarse-graining $\lambda_{cg}:K\to H$.
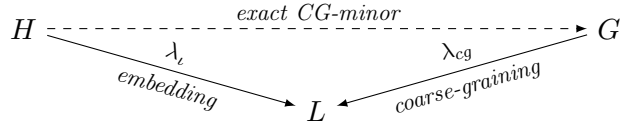
\begin{center}
\begin{tikzpicture}[scale=1.1]
\node (v2) at (3.5,-2) {$G$};
\node (v3) at (-3.5,-2) {$H$};
\draw [dashed,-latex] (v3) -- node[sloped,scale=0.8, above] {exact CG-minor} (v2);
\node (v4) at (0,-3) {$L$};
\draw [latex-] (v4) edge node[sloped,scale=0.8, above] {$\lambda_{cg}$}node[sloped,scale=0.8, below] {coarse-graining}  (v2);
\draw [latex-] (v4) edge node[sloped,scale=0.8, above] {$\lambda_{\iota}$}node[sloped,scale=0.8, below] {embedding}(v3);
\end{tikzpicture}
\end{center}
\end{thm}

As shown in Example \ref{def},  not all CG-minors are exact.

\subsection{Contraction minor}

In this subsection, we take $\mathcal{C}=\mathbf{Cau}$, $\mathcal{Q}$ to be the class $\mathbf{Con}$ of all contractions and $\mathcal{E}$ to be the class $\mathbf{EM}$ of all embeddings. In this case, a $\{\mathbf{Con}, \mathbf{EM}\}$-minor is called a \textbf{contraction minor} or simply \textbf{minor}. Clearly, $\{\mathbf{Con}, \mathbf{EM}\}$ is a strict sub-minor pair of $\{\mathbf{CG}, \mathbf{EM}\}$ and all contraction minors are CG-minors. By Lemma \ref{c1}, we see that $\mathbf{Con}$ dominates $\mathbf{EM}$.

We list three definitions of contraction minor relations, and show that they are equivalent.

\begin{defn}\label{cm1}
$H$ is called a \textbf{contraction minor} of $G$ if there is a contraction minor $T:H\rightsquigarrow G$.
\end{defn}

As in the case of CG-minor, Theorem \ref{contraction}  and Theorem  \ref{sub}  imply the equivalence of Definition \ref{cm1} and the following definition.

\begin{defn}\label{cm2}
$H$ is called a \textbf{contraction minor} of $G$ if $H$ can be obtained from $G$ by a sequence of the following three types of fundamental operations:
$(1)$ deleting an edge; $(2)$ deleting an isolated vertex; $(3)$ contracting a multi-edge.
\end{defn}

The equivalence of Definition \ref{cm1} and the following definition is a direct consequence of Theorem \ref{contraction}  and Theorem  \ref{sub} and  the dominating property of $\{\mathbf{Con},\mathbf{EM}\}$.
\begin{defn}\label{cm3}
$H$ is called a \textbf{contraction minor} of $G$ if  $H$ is a contraction of a sub-causal-net of $G$, or equivalently, there is a causal-net $K$ with an embedding $\lambda_{\iota}:K\to G$ and a contraction $\lambda_{con}:K\to H$.
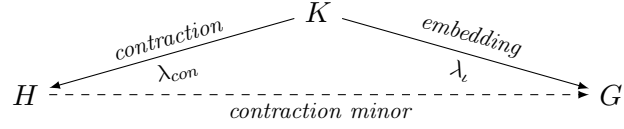
\begin{center}
\begin{tikzpicture}[scale=1.1]
\node (v1) at (0,-1) {$K$};
\node (v2) at (3.5,-2) {$G$};
\node (v3) at (-3.5,-2) {$H$};
\draw [-latex] (v1) edge node[sloped,scale=0.8, below] {$\lambda_{\iota}$}node[sloped,scale=0.8, above] {embedding}  (v2);
\draw [-latex] (v1) edge node[sloped,scale=0.8, below] {$\lambda_{con}$}node[sloped,scale=0.8, above] {contraction}(v3);
\draw [dashed,-latex] (v3) -- node[sloped,scale=0.8, below] {contraction minor} (v2);
\end{tikzpicture}
\end{center}
\end{defn}
We call $H$ an\textbf{ exact contraction minor} of $G$ if there exists an exact contraction minor $T:H\rightsquigarrow G$.
Exact contraction minors can be characterized as sub-contractions.
\begin{thm}
$H$ is an exact contraction minor of $G$ if and only if  $H$ is a sub-contraction of $G$, or equivalently, there is a causal-net $L$ with an embedding $\lambda_{\iota}:K\to G$ and a contraction $\lambda_{con}:K\to H$.
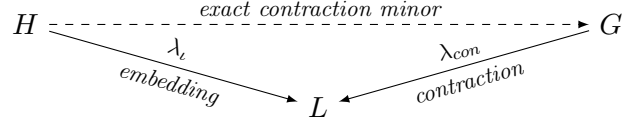
\begin{center}
\begin{tikzpicture}[scale=1.1]
\node (v2) at (3.5,-2) {$G$};
\node (v3) at (-3.5,-2) {$H$};
\draw [dashed,-latex] (v3) -- node[sloped,scale=0.8, above] {exact contraction minor} (v2);
\node (v4) at (0,-3) {$L$};
\draw [latex-] (v4) edge node[sloped,scale=0.8, above] {$\lambda_{con}$}node[sloped,scale=0.8, below] {contraction}  (v2);
\draw [latex-] (v4) edge node[sloped,scale=0.8, above] {$\lambda_{\iota}$}node[sloped,scale=0.8, below] {embedding}(v3);
\end{tikzpicture}
\end{center}
\end{thm}
\begin{proof}
This theorem is a consequence of the definitions and  Lemma \ref{c1}.
\end{proof}

\begin{ex}
The following figure shows that $H$ is an exact minor of $G$.
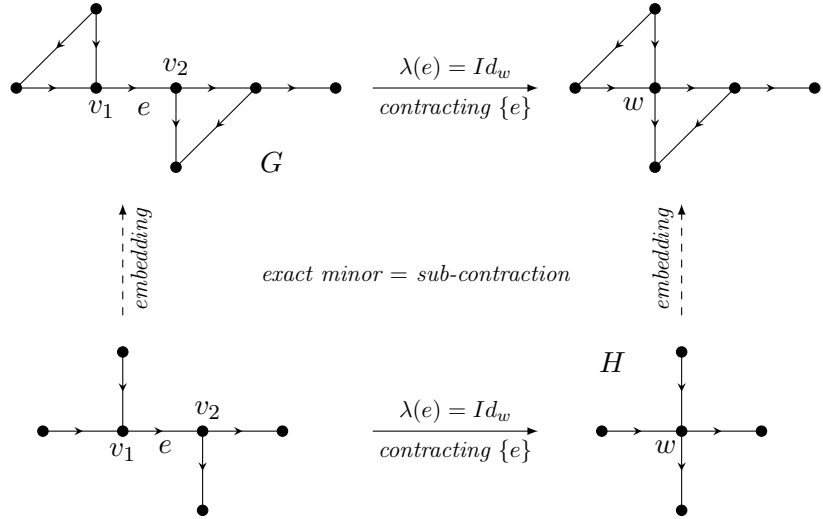
\begin{center}
\begin{tikzpicture}[scale=0.7]
\node (v1) at (-2,-2) {};
\node (v2) at (-2,-3.5) {};
\node (v3) at (-3.5,-3.5) {};
\node (v4) at (-0.5,-3.5) {};
\node (v5) at (1,-3.5) {};
\node (v6) at (-0.5,-5) {};
\node (v8) at (-2.5,3) {};
\node (v7) at (-2.5,4.5) {};
\node (v9) at (-4,3) {};
\node (v10) at (-1,3) {};
\node (v12) at (0.5,3) {};
\node (v11) at (-1,1.5) {};
\draw  (-2,-2) -- (-2,-3.5)[postaction={decorate, decoration={markings,mark=at position .5 with {\arrow[black]{stealth}}}}];
\draw (-3.5,-3.5) -- (-2,-3.5)[postaction={decorate, decoration={markings,mark=at position .5 with {\arrow[black]{stealth}}}}];
\draw (-2,-3.5) -- (-0.5,-3.5)[postaction={decorate, decoration={markings,mark=at position .5 with {\arrow[black]{stealth}}}}];
\draw  (-0.5,-3.5) -- (1,-3.5)[postaction={decorate, decoration={markings,mark=at position .5 with {\arrow[black]{stealth}}}}];
\draw  (-0.5,-3.5) -- (-0.5,-5)[postaction={decorate, decoration={markings,mark=at position .5 with {\arrow[black]{stealth}}}}];
\draw  (-2.5,4.5) -- (-2.5,3)[postaction={decorate, decoration={markings,mark=at position .5 with {\arrow[black]{stealth}}}}];
\draw  (-2.5,4.5) -- (-4,3)[postaction={decorate, decoration={markings,mark=at position .5 with {\arrow[black]{stealth}}}}];
\draw  (-4,3) --(-2.5,3)[postaction={decorate, decoration={markings,mark=at position .5 with {\arrow[black]{stealth}}}}];
\draw  (-2.5,3) -- (-1,3)[postaction={decorate, decoration={markings,mark=at position .5 with {\arrow[black]{stealth}}}}];
\draw  (-1,3) --  (-1,1.5)[postaction={decorate, decoration={markings,mark=at position .5 with {\arrow[black]{stealth}}}}];
\draw  (-1,3) -- (0.5,3) [postaction={decorate, decoration={markings,mark=at position .5 with {\arrow[black]{stealth}}}}];
\draw  (0.5,3)  --  (-1,1.5)[postaction={decorate, decoration={markings,mark=at position .5 with {\arrow[black]{stealth}}}}];
\node (v13) at (2,3) {};
\draw  (0.5,3) -- (2,3)[postaction={decorate, decoration={markings,mark=at position .5 with {\arrow[black]{stealth}}}}];
\node (v14) at (8.5,-2) {};
\node (v15) at (8.5,-3.5) {};
\node (v18) at (8.5,-5) {};
\node (v16) at (7,-3.5) {};
\node (v17) at (10,-3.5) {};
\node (v20) at (8,3) {};
\node (v21) at (8,4.5) {};
\node (v24) at (8,1.5) {};
\node (v19) at (6.5,3) {};
\node (v22) at (9.5,3) {};
\node (v23) at (11,3) {};
\draw  (8.5,-2) -- (8.5,-3.5)[postaction={decorate, decoration={markings,mark=at position .5 with {\arrow[black]{stealth}}}}];
\draw   (7,-3.5)  --(8.5,-3.5)[postaction={decorate, decoration={markings,mark=at position .5 with {\arrow[black]{stealth}}}}];
\draw  (8.5,-3.5) -- (10,-3.5)[postaction={decorate, decoration={markings,mark=at position .5 with {\arrow[black]{stealth}}}}];
\draw  (8.5,-3.5) -- (8.5,-5)[postaction={decorate, decoration={markings,mark=at position .5 with {\arrow[black]{stealth}}}}];
\draw (6.5,3)-- (8,3)[postaction={decorate, decoration={markings,mark=at position .5 with {\arrow[black]{stealth}}}}];
\draw  (8,4.5)-- (8,3)[postaction={decorate, decoration={markings,mark=at position .5 with {\arrow[black]{stealth}}}}];
\draw  (8,4.5) -- (6.5,3)[postaction={decorate, decoration={markings,mark=at position .5 with {\arrow[black]{stealth}}}}];
\draw  (8,3) -- (9.5,3)[postaction={decorate, decoration={markings,mark=at position .5 with {\arrow[black]{stealth}}}}];
\draw  (9.5,3) -- (11,3)[postaction={decorate, decoration={markings,mark=at position .5 with {\arrow[black]{stealth}}}}];
\draw  (8,3) -- (8,1.5)[postaction={decorate, decoration={markings,mark=at position .5 with {\arrow[black]{stealth}}}}];
\draw  (9.5,3) -- (8,1.5)[postaction={decorate, decoration={markings,mark=at position .5 with {\arrow[black]{stealth}}}}];
\node (v25) at (-2,-1.5) {};
\node (v26) at (-2,1) {};
\node (v31) at (8.5,-1.5) {};
\node (v32) at (8.5,1) {};
\node (v29) at (2.5,-3.5) {};
\node (v30) at (6,-3.5) {};
\node (v27) at (2.5,3) {};
\node (v28) at (6,3) {};
\draw [-latex,dashed] (v25) edge node[sloped,scale=0.8, below] {embedding}(v26);
\draw [-latex] (v27) edge node[sloped,scale=0.8, above] {$\lambda(e)=Id_w$ }node[sloped,scale=0.8, below] {contracting $\{e\}$} (v28);
\draw [-latex] (v29) edge node[sloped,scale=0.8, above] {$\lambda(e)=Id_w$} node[sloped,scale=0.8, below] {contracting $\{e\}$}(v30);
\draw [-latex,dashed] (v31) edge node[sloped,scale=0.8, above] {embedding} (v32);
\node[scale=0.8] at (3.5,-0.5) {exact minor $=$ sub-contraction};
\draw [fill](v1) circle [radius=0.098];
\draw [fill](v2) circle [radius=0.098];
\draw [fill](v3) circle [radius=0.098];
\draw [fill](v4) circle [radius=0.098];
\draw [fill](v5) circle [radius=0.098];
\draw [fill](v6) circle [radius=0.098];
\draw [fill](v7) circle [radius=0.098];
\draw [fill](v8) circle [radius=0.098];
\draw [fill](v9) circle [radius=0.098];
\draw [fill](v10) circle [radius=0.098];
\draw [fill](v11) circle [radius=0.098];
\draw [fill](v12) circle [radius=0.098];
\draw [fill](v13) circle [radius=0.098];
\draw [fill](v14) circle [radius=0.098];
\draw [fill](v15) circle [radius=0.098];
\draw [fill](v16) circle [radius=0.098];
\draw [fill](v17) circle [radius=0.098];
\draw [fill](v18) circle [radius=0.098];
\draw [fill](v19) circle [radius=0.098];
\draw [fill](v20) circle [radius=0.098];
\draw [fill](v21) circle [radius=0.098];
\draw [fill](v22) circle [radius=0.098];
\draw [fill](v23) circle [radius=0.098];
\draw [fill](v24) circle [radius=0.098];
\node at (-2.4,2.6) {$v_1$};
\node at (-1.6,2.6) {$e$};
\node at (-1,3.4) {$v_2$};
\node at (7.6,2.7) {$w$};
\node at (-2,-3.9) {$v_1$};
\node at (-0.4,-3.1) {$v_2$};
\node at (8.2,-3.8) {$w$};
\node at (-1.2,-3.8) {$e$};
\node at (0.8,1.6) {$G$};
\node at (7.2,-2.2) {$H$};
\end{tikzpicture}
\end{center}
\end{ex}
As shown in Example \ref{def},  not all contraction-minors are exact.

\subsection{Fusion minor}

In this subsection, we take $\mathcal{C}=\mathbf{Cau}$, $\mathcal{Q}$ to be the class $\mathbf{FU}$ of all fusions and $\mathcal{E}$ to be the class $\mathbf{EM}$ of all embeddings. In this case, a $\{\mathbf{FU}, \mathbf{EM}\}$-minor is called a \textbf{fusion minor}. Clearly, $\{\mathbf{FU}, \mathbf{EM}\}$ is a strict sub-minor pair of $\{\mathbf{CG}, \mathbf{EM}\}$ and all fusion minors are CG-minors.

A minor pair is called \textbf{balanced} if any defect with respect to it has a dual defect. By Lemma \ref{m1}, Lemma \ref{m2}, Lemma \ref{cos1} and Lemma \ref{cos2}, we see that $\{\mathbf{FU}, \mathbf{EM}\}$ is balanced, which implies that all fusion minors are exact.

We list four equivalent definitions of fusion minor relations.

\begin{defn}\label{fu1}
$H$ is called a \textbf{fusion minor} of $G$ if there is fusion minor $T:H\rightsquigarrow G$.
\end{defn}

By Corollary \ref{fusion-decomposition}, any fusion is a composition of a merging and an edge-coarse-graining, both of which can be represent as a series of operations of  merging two vertices and a series of operations of coarse-graining two parallel edges, respectively. Therefore, together with Theorem \ref{sub},  Corollary \ref{fusion-decomposition} implies that Definition \ref{fu1} is equivalent to the following definition.

\begin{defn}\label{fu2}
$H$ is called a \textbf{fusion minor} of $G$ if $H$ can be obtained from $G$ by a sequence of the following three types of fundamental operations:
$(1)$ deleting an edge; $(2)$ deleting an isolated vertex; $(3)$ merging two vertices; $(4)$ coarse-graining two parallel edges.
\end{defn}
The equivalence of this definition and the following one follows from Corollary \ref{fusion-decomposition}, Theorem \ref{sub}, Lemma \ref{m1} and Lemma \ref{cos1}.

\begin{defn}\label{cm3}
$H$ is called a \textbf{fusion minor} of $G$ if  $H$ is a fusion of a sub-causal-net of $G$, or equivalently, there is a causal-net $K$ with an embedding $\lambda_{\iota}:K\to G$ and a fusion $\lambda_{fu}:K\to H$.
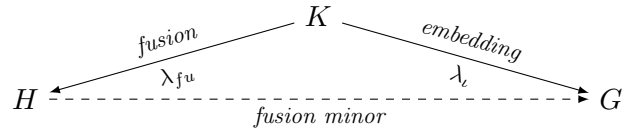
\begin{center}
\begin{tikzpicture}[scale=1.1]
\node (v1) at (0,-1) {$K$};
\node (v2) at (3.5,-2) {$G$};
\node (v3) at (-3.5,-2) {$H$};
\draw [-latex] (v1) edge node[sloped,scale=0.8, below] {$\lambda_{\iota}$}node[sloped,scale=0.8, above] {embedding}  (v2);
\draw [-latex] (v1) edge node[sloped,scale=0.8, below] {$\lambda_{fu}$}node[sloped,scale=0.8, above] {fusion}(v3);
\draw [dashed,-latex] (v3) -- node[sloped,scale=0.8, below] {fusion minor} (v2);
\end{tikzpicture}
\end{center}
\end{defn}
The equivalence of this definition  and the following one follows from Corollary \ref{fusion-decomposition}, Theorem \ref{sub}, Lemma \ref{m1}, Lemma \ref{m2}, Lemma \ref{cos1} and Lemma \ref{cos2}.

\begin{defn}
$H$ is called a \textbf{fusion minor} of $G$ if  $H$ is a sub-fusion of $G$, or equivalently, there is a causal-net $L$ with an embedding $\lambda_{\iota}:K\to G$ and a fusion $\lambda_{fu}:K\to H$.
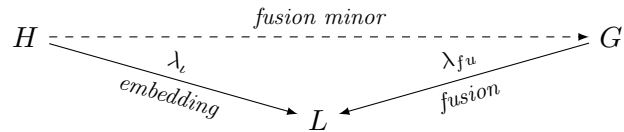
\begin{center}
\begin{tikzpicture}[scale=1.1]
\node (v2) at (3.5,-2) {$G$};
\node (v3) at (-3.5,-2) {$H$};
\draw [dashed,-latex] (v3) -- node[sloped,scale=0.8, above] {fusion minor} (v2);
\node (v4) at (0,-3) {$L$};
\draw [latex-] (v4) edge node[sloped,scale=0.8, above] {$\lambda_{fu}$}node[sloped,scale=0.8, below] {fusion}  (v2);
\draw [latex-] (v4) edge node[sloped,scale=0.8, above] {$\lambda_{\iota}$}node[sloped,scale=0.8, below] {embedding}(v3);
\end{tikzpicture}
\end{center}
\end{defn}

We call $H$ an \textbf{exact fusion minor} of $G$ if there exists an exact fusion minor $T:H\rightsquigarrow G$. Since all fusion minors are exact, therefor $H$ is an exact fusion minor of $G$ if and only if $H$ is a fusion minor of $G$.

\section*{Acknowledgement}
The author thanks the reviewers of ACT2022 for their insightful feedbacks. Special thanks to one reviewer of ACT2022 for pointing some errors in the earlier version of this paper, showing counter-examples and more importantly asking some fundamental questions that inspire this long paper. Example \ref{str}, \ref{r2} and \ref{r3} are attributed to him. The author also thanks Z.Chen and H.T.Sun for helpful communications.


\textbf{Xuexing Lu}\hfill \\  Email: xxlu@uzz.edu.cn

\end{document}